\documentclass[9pt, a4paper, centertags, openright,
]{amsart}

\usepackage{etex}

\usepackage{amsmath, amsfonts, amsthm, 
amssymb, amsxtra, bbm, enumerate, stmaryrd}

\usepackage{mathtools}

\usepackage[dvipsnames]{xcolor}
\usepackage[retainorgcmds]{IEEEtrantools}
\usepackage[mathscr]{euscript}
\usepackage{graphicx}
\usepackage[all, cmtip]{xy}
\xyoption{curve}

\usepackage{cite}

\usepackage{wasysym}
\usepackage[english]{babel}
\usepackage[utf8]{inputenc}
\usepackage{paralist}
\usepackage{textcomp}
\usepackage[colorinlistoftodos]{todonotes}
\usepackage{caption}

\usepackage{fancyhdr}
\usepackage[bottom]{footmisc}
\usepackage[hyphens]{url}

\definecolor{Myblue}{rgb}{0,0,0.6}
\usepackage[a4paper, colorlinks=true,
linkcolor=red,
citecolor=blue,
]{hyperref}

\usepackage{multicol}
\usepackage{breakurl}

\theoremstyle{theorem}
\newtheorem{lem}{Lemma}[section]
\newtheorem{prop}[lem]{Proposition}
\newtheorem{cor}[lem]{Corollary}
\newtheorem{thm}[lem]{Theorem}

\newtheorem*{Thm1}{Theorem~1}
\newtheorem*{Thm2}{Theorem~2}
\newtheorem*{Thm3}{Theorem~3}
\newtheorem*{Thm4}{Theorem~4}

\theoremstyle{remark}
\newtheorem{rem}[lem]{Remark}

\theoremstyle{definition}
\newtheorem{exa}[lem]{Example}
\newtheorem{exas}[lem]{Examples}
\newtheorem{defn}[lem]{Definition}
\newtheorem{quest}[lem]{Question}
\newtheorem{nn}[lem]{}

\newtheorem*{Quest}{Question}
\numberwithin{equation}{section}

\newcommand{\Proj}{\operatorname{\mathsf{Proj}}\nolimits}

\newcommand{\Der}{\operatorname{Der}\nolimits}
\newcommand{\Inn}{\operatorname{Inn}\nolimits}

\newcommand{\pd}{\operatorname{pd}\nolimits}
\newcommand{\id}{\operatorname{id}\nolimits}
\newcommand{\Id}{\operatorname{Id}\nolimits}

\newcommand{\Mod}{\operatorname{\mathsf{Mod}}\nolimits}

\newcommand{\Aut}{\operatorname{Aut}\nolimits}

\newcommand{\End}{\operatorname{End}\nolimits}
\newcommand{\Hom}{\operatorname{Hom}\nolimits}

\newcommand{\grade}{\operatorname{grade}\nolimits}

\renewcommand{\Im}{\operatorname{Im}\nolimits}
\newcommand{\Ker}{\operatorname{Ker}\nolimits}

\newcommand{\Coker}{\operatorname{Coker}\nolimits}
\newcommand{\coker}{\operatorname{coker}\nolimits}

\newcommand{\Ext}{\operatorname{Ext}\nolimits}

\newcommand{\Tor}{\operatorname{Tor}\nolimits}

\newcommand{\ev}{\mathrm{ev}}

\newcommand{\HH}{\operatorname{HH}\nolimits}

\newcommand{\Mor}{\operatorname{Mor}\nolimits}
\newcommand{\Ob}{\operatorname{Ob}\nolimits}

\newcommand{\Out}{\operatorname{Out}\nolimits}

\newcommand{\abs}[1]{\ensuremath{\left\vert#1\right\vert}}

\def\A{{\mathsf A}}
\def\B{{\mathsf B}}

\def\D{{\mathsf D}}

\def\E{{\mathsf E}}
\def\F{{\mathsf F}}

\def\P{{\mathsf P}}

\def\U{{\mathsf U}}

\def\Z{{\mathbb Z}}

\def\op{\mathrm{op}}

\def\C{{\mathsf C}}

\def\bbm1{{\mathbbm 1}}

\newdir{ >}{{}*!/-5pt/@{>}}
\newdir{> }{{}*!/+5pt/@{>}}
\newdir{ <}{{}*!/-5pt/@{<}}
\newdir{>>> }{{}*!/+10pt/@{>}}

\entrymodifiers={+!!<0pt,\fontdimen22\textfont2>}

\makeatletter
\DeclareRobustCommand{\tvdots}{%
  \vbox{\baselineskip4\p@\lineskiplimit\z@\kern0\p@\hbox{.}\hbox{.}\hbox{.}}}
\makeatother

\newcounter{saveenumerate}
\makeatletter
\newcommand{\enumeratetext}[1]{%
\setcounter{saveenumerate}{\value{enum\romannumeral\the\@enumdepth}}
\end{enumerate}
#1
\begin{enumerate}
\setcounter{enum\romannumeral\the\@enumdepth}{\value{saveenumerate}}%
}
\makeatother

\begin{document}

\title[Homological epimorphisms, recollements and Hochschild cohomology]{Homological epimorphisms, recollements and Hochschild cohomology -- with a conjecture by Snashall--Solberg in view}
\author{Reiner Hermann}
\address{Reiner Hermann\\ Institutt for matematiske fag\\
NTNU\\ 7491 Trondheim\\ Norway}
\email{reiner.hermann@math.ntnu.no}
\thanks{}
\begin{abstract}
We show that recollements of module categories give rise to  homomorphisms between the associated Hochschild cohomology algebras which preserve the strict Gerstenhaber structure, i.e., the cup product, the graded Lie bracket and the squaring map. We review various long exact sequences in Hochschild cohomology and apply our results in order to realise that the occurring maps preserve the strict Gerstenhaber structure as well. As a byproduct, we generalise a known long exact cohomology sequence of Koenig--Nagase to arbitrary surjective homological epimorphisms. We use our observations to motivate and formulate a variation of the finite generation conjecture by Snashall--Solberg.
\end{abstract}

\maketitle
\tableofcontents


\section{Introduction}

Let $K$ be a commutative ring and $A$ an algebra over $K$ which we assume to be $K$-projective. The \textit{Hochschild cohomology algebra of $A$} (\textit{over} $K$) can be expressed as $\HH^\ast(A) = \Ext^\ast_{A^\ev}(A,A)$, where $A^\ev = A \otimes_K A^\op$ is the \textit{enveloping algebra of $A$ over $K$}, over which $A$ naturally becomes a left module. The cup product $\smallsmile$ on Yoneda extensions, see \cite{CaEi56} and \cite{Yo54}, turns $\HH^\ast(A)$ into a graded commutative algebra; that is, a $\mathbb Z$-graded algebra such that $\alpha \smallsmile \beta = (-1)^{mn} \beta \smallsmile \alpha$ for homogeneous elements $\alpha$ and $\beta$ of degrees $\abs{\alpha} = m$ and $\abs{\beta} = n$. The multiplicative structure on $\HH^\ast(\Lambda)$ has been used in \cite{EHTSS04} and \cite{SnSo04} (see also \cite{Sn09} and \cite{PSS14}) to develop a support variety theory for Artin algebras $\Lambda$, analogous to the classical ones for group (see \cite{Ev91}) and cocommutative Hopf algebras. 

However, the graded commutative product is just part of richer structure on $\HH^\ast(A)$, for, in 1963, Gerstenhaber discovered a graded Lie bracket,
$$
\{-,-\}_A \colon \HH^{m}(A) \times \HH^n(A) \longrightarrow \HH^{m+n-1}(A),
$$
lowering the degree by $1$ and a (divided) squaring map,
$$
sq_A \colon \HH^{2n}(A) \longrightarrow \HH^{4n-1}(A),
$$
which takes over the role of the map $\alpha \mapsto 2^{-1}\{\alpha,\alpha\}$ if $2 \notin K^\times$. The bracket $\{-,-\}_A$ is compatible with $\smallsmile$ in that it acts on $\HH^\ast(A)$ through graded derivations, and a derived invariant as shown in \cite{Ke04} (compare with \cite{He14b} as well). In the initial article \cite{Ge63}, see also \cite{Ge68}, Gerstenhaber presented $\{-,-\}_A$ as a tool to study and classify deformations of $A$. However, by theorems such as the one of Hochschild-Kostant-Rosenberg (see \cite{HKR62}), Gerstenhaber's bracket found its way into the realm of differential and Poisson geometry, and other branches of mathematics. See \cite{Ko03}, \cite{KoSo09} and \cite{Xu94} for additional applications.
\medskip

It is, of course, a fundamental question whether, or when, an exact functor $\Mod(A^\ev) \rightarrow \Mod(B^\ev)$ gives rise to a graded homomorphism $\HH^\ast(A) \rightarrow \HH^\ast(B)$ taking the strict Gerstenhaber structure on the source to the one on the target. Significant progress in this re\-gard was initiated in \cite{NeRe96}, \cite{Re86} and \cite{Sch98} by taking homotopy theoretical properties of \textit{extension categories} $\mathcal Ext^n_{A^\ev}(A,A)$ into account, leading to a description of the bracket $\{-,-\}_A$ in terms of Yoneda extensions; see also \cite{Bu11} for some adaptations of those methods in the setting of triangulated categories. In \cite{He14b}, we formulated and investigated a generalisation of the construction presented in \cite{Sch98}, leading to the insight, that many monoidal functors $(\Mod(A^\ev), \otimes_A, A) \rightarrow (\Mod(B^\ev), \otimes_B, B)$ give rise to very well-behaved maps between Hochschild cohomology algebras. This is a powerful tool that we will instrument at various stages of this article, as it will allow us to study the Lie bracket in Hochschild cohomology by means of homological algebra.
\medskip

Our first main result involves \textit{recollements of abelian categories}, or more specifically, \textit{recollements of module categories} which are the analogue of split short exact sequences taken in the ``category'' of abelian categories. As such, a recollement $\mathcal R(\A,\B,\C)$ between abelian categories $\A$, $\B$ and $\C$ thus comes with exact functors $\A \rightarrow \B \rightarrow \C$, which admit ``nice'' left and right adjoints. Recollements of abelian categories first appeared in the famous article of Be{\u\i}linson--Bernstein--Deligne on perverse sheaves, see \cite{BBD82}, wherein they were obtained from their triangulated analogues, i.e., recollements of triangulated categories. Since then, recollements have been studied by various authors, with different backgrounds; see e.g. \cite{ClPaSc88} and \cite{Ku94}. Note that frequently a recollement of abelian categories induces such of triangulated categories, by deriving the whole diagram.

If the categories $\A$, $\B$ and $\C$ in $\mathcal R(\A,\B,\C)$ are module categories over $K$-algebras $A$, $B$ and $C$, with $K$-linear functors in between them, we call $\mathcal R(\A,\B,\C)$ a \textit{$K$-linear recollement of module categories}, and denote it by $\mathcal R(A,B,C)$. The theorem in this context predicts that -- under mild assumptions -- a recollement of module categories gives rise to a map between Hochschild cohomology algebras that preserve all of the rich structure that was mentioned above. It makes use of recent results by Psaroudakis--Vit{\'o}ria; see \cite{PsVi14}.

\begin{Thm1}[$=$ Proposition \ref{prop:envrecoll} \& Theorem \ref{thm:mainthm}] Let $\mathcal R = \mathcal R(A,B,C)$ be a $K$-linear recollement of module categories. Assume that $B$ is projective as a $K$-module. Further, denote by $j$ the functor $\Mod(B) \rightarrow \Mod(C)$ in $\mathcal R$ and by $(-)^\vee$ the functor $\Hom_{B^\op}(-,B)$. 

\begin{enumerate}[\rm(1)]
\item The functor $j^\ev(-) = j(B) \otimes_B (-) \otimes_B j(B)^\vee$ determines a $K$-linear recollement $\mathcal R^\ev = \mathcal R(\Ker(j^\ev), \Mod(B^\ev), \Mod(C^\ev))$. Moreover, the image of $B$ under $j^\ev \colon \Mod(B^\ev) \rightarrow \Mod(C^\ev)$ is $K$-projective and canonically isomorphic to $C$ $($as a $C^\ev$-module$)$. 
\item Under the additional assumption that 
$$
\Tor^C_1(j(B)^\vee, j(B)) = 0\,,
$$
the graded $K$-algebra homomorphism
$$
j^\ev_\ast \colon \HH^\ast(B) \xrightarrow{\ j^\ev \ } \Ext^\ast_{C^\ev}(j^\ev(B), j^\ev(B)) \xrightarrow{\ \sim \ } \HH^\ast(C)
$$
is a homomorphism of strict Gerstenhaber algebras, i.e., the diagrams
$$
\xymatrix@C=35pt{
\HH^m(B) \times \HH^n(B) \ar[r]^-{\{-,-\}_B} \ar[d]_-{j^\ev_m \times j^\ev_n} & \HH^{m+n-1}(B) \ar[d]^-{j^\ev_{m + n -1}} \\
\HH^{m}(C) \times \HH^n(C) \ar[r]^-{\{-,-\}_C} & \HH^{m+n-1}(C)
}
\ \quad\quad \
\xymatrix@C=35pt{
\HH^{2n}(B) \ar[r]^{sq_B} \ar[d]_-{j^\ev_{2n}} & \HH^{4n-1}(B) \ar[d]^-{j^\ev_{4n-1}}\\
\HH^{2n}(C) \ar[r]^{sq_C} & \HH^{4n-1}(C)
}
$$
commute.
\end{enumerate}
\end{Thm1}

If $\Tor^C_1(j(B)^\vee, j(B)) = 0$, one may want to call the functor $j: \Mod(B) \rightarrow \Mod(C)$, or the recollement $\mathcal R(A,B,C)$ that it belongs to, \textit{pseudoflat}, a terminology that has been established in \cite{Bu98} for ring homomorphisms $R \rightarrow S$ with $\Tor^R_1(S,S) = 0$.
\medskip

Frequently, the functor $\Mod(A) \rightarrow \Mod(B)$ in a recollement $\mathcal R(A,B,C)$ of module categories is given by the restriction along a \textit{homological epimorphism} $\pi\colon B \rightarrow A$, as introduced and studied by Geigle--Lenzing in \cite{GeLe91}. Such a homological epimorphism $\pi$ gives rise to a homological epimorphism $\pi^\ev \colon B^\ev \rightarrow A^\ev$, and the left adjoint to the associated restriction functor $\pi^\ev_\star\colon \Mod(A^\ev) \rightarrow \Mod(B^\ev)$ has notable properties.

\begin{Thm2}[= Theorem \ref{thm:homepi}]
Let $\pi\colon B \rightarrow A$ be a $K$-linear homological epimorphism between $K$-projective $K$-algebras. Then the functor
$$
A \otimes_B (-) \otimes_B A \colon \Mod(B^\ev) \longrightarrow \Mod(A^\ev)
$$
maps $B$ to $A$ and, moreover, induces a homomorphism
$$
\pi_\ast \colon \HH^\ast(B) = \Ext^\ast_{B^\ev}(B,B) \longrightarrow \Ext^\ast_{A^\ev}(A,A) = \HH^\ast(A)
$$
of strict Gerstenhaber algebras, that is, $\pi_\ast$ is a graded $K$-algebra homomorphism and the diagrams
$$
\xymatrix@C=35pt{
\HH^m(B) \times \HH^n(B) \ar[r]^-{\{-,-\}_B} \ar[d]_-{\pi_m \times \pi_n} & \HH^{m+n-1}(B) \ar[d]^-{\pi_{m + n -1}} \\
\HH^{m}(A) \times \HH^n(A) \ar[r]^-{\{-,-\}_A} & \HH^{m+n-1}(A)
}
\ \quad\quad \
\xymatrix@C=35pt{
\HH^{2n}(B) \ar[r]^{sq_B} \ar[d]_-{\pi_{2n}} & \HH^{4n-1}(B) \ar[d]^-{\pi_{4n-1}}\\
\HH^{2n}(A) \ar[r]^{sq_A} & \HH^{4n-1}(A)
}
$$
commute.
\end{Thm2}

As the functor $A \otimes_B (-) \otimes_B A \colon \Mod(B^\ev) \rightarrow \Mod(A^\ev)$ is, in general, not exact, the map $\pi_\ast$ will not be given by applying $A \otimes_B (-) \otimes_B A$ to extensions directly. However, one may pass to a suitable exact subcategory $\C$ of $\Mod(B^\ev)$ on which $A \otimes_B (-) \otimes_B A$ is exact, and which is, in the sense of Paragraph \ref{nn:extclosed}, entirely extension closed in that extension groups in $\Mod(B^\ev)$ may be computed by extensions in $\C$. This is indeed some sort of a derived functor, as we will notice along the way.
\medskip

In the context of triangular algebras, one-point-extensions and stratifying ideals, various long exact sequences for Hochschild cohomology algebras have been established; see, for instance, \cite{BeGu04, Bu03, Ci00, GoGo01, GrMaSn03, GrSo02, Ha14, Ha89, KoNa09, MiPl00}. The graded maps within those sequences are known to respect the multiplicative structure. Theorems 1 and 2 may be applied, in oder to realise that even more structure is preserved. We take care of a long exact sequence of Green--Solberg first; see \cite{GrSo02}. Let $B$ be a $K$-algebra.

\begin{Thm3}[= Theorem \ref{thm:greensol}]
Let $e \in B$ be an idempotent, $e' = 1 - e$ its complementary idempotent and $C = eBe$, $C' = e'Be'$. If
$$
\Tor^{C}_i(Be,eB) = 0 = \Tor^{C'}_i(Be',e'B) \quad \text{$($for all $i > 0)$},
$$
then there is a long exact sequence
$$
\xymatrix@C=14.5pt{
\cdots \ar[r] & \Ext^{n}_{B^\ev}(B,B) \ar[r]^-{g_n} & \Ext^n_{C^\ev}(C,C) \oplus \Ext^n_{(C')^\ev}(C', C') \ar[r] & \Ext^{n}_{B^\ev}(\Omega^1_B,B) \ar[r] & \cdots
}
$$
induced from the canonical short exact sequence
$$
\xymatrix@C=14pt{
0 \ar[r] & {\Omega}^1_B \ar[r] & \displaystyle (Be \otimes_{C} eB) \oplus (Be' \otimes_{C'} e'B) \ar[r] & B \ar[r] & 0
}.
$$
If, moreover, $B$ is $K$-projective, then the map $g_\ast\colon \HH^\ast(B) \rightarrow \HH^\ast(eBe) \times \HH^\ast(e'Be')$ is a homomorphism of strict Gerstenhaber algebras.
\end{Thm3}

Note that by specialising to $B = \left(\begin{smallmatrix} R & M\\ 0 & K\end{smallmatrix}\right)$ for a $K$-algebra $R$ and an $R$-module $M$, Theorem 3 recovers the long exact sequence for one-point-extensions introduced by Happel in \cite{Ha89}. Thus the maps in Happel's sequence preserve the bracket as well. The exact sequences in the upcoming statement are due to Koenig--Nagase; see \cite{KoNa09}. Recall that an ideal $I \subseteq B$ is called \textit{stratifying}, if there is an idempotent element $e \in B$ with $I = BeB$, such that $\Tor^{eBe}_i(Be,eB) = 0$ for $i > 0$ and the multiplication map $Be \otimes_{eBe} eB \rightarrow BeB = I$ is an isomorphism. In this situation, the canonical surjection $\pi\colon B \rightarrow B/I$ is a homological epimorphism (cf.\,Lemma \ref{lem:strattorvan}).

\begin{Thm4}[$\subseteq$ Theorem \ref{thm:koenignagase}]
Let $I \subseteq B$ be a \emph{stratifying ideal}, and $A = B/I$. Assume further that $B$ and $A$ are $K$-projective. Then the following hold true.
\begin{enumerate}[\rm(1)]
\item The short exact sequence $0 \rightarrow I \rightarrow B \rightarrow A \rightarrow 0$ gives rise to a long exact sequence
$$
\xymatrix@C=18pt{
\cdots \ar[r] & \Ext^n_{B^\ev}(B,I) \ar[r] & \HH^n(B) \ar[r]^-{k_n} & \HH^n(A) \ar[r] & \Ext^{n+1}_{B^\ev}(B,I) \ar[r] & \cdots
}
$$
in Hochschild cohomology. The map $k_\ast \colon \HH^\ast(B) \rightarrow \HH^\ast(A)$ is a homomorphism of strict Gerstenhaber algebras.

\item The short exact sequence $0 \rightarrow I \rightarrow B \rightarrow A \rightarrow 0$ gives rise to a long exact sequence
$$
\xymatrix@C=18pt{
\cdots \ar[r] & \Ext_{B^\ev}^n(A,B) \ar[r] & \HH^n(B) \ar[r]^-{l_n} & \HH^n(eBe) \ar[r] & \Ext_{B^\ev}^{n+1}(A,B) \ar[r] & \cdots
}
$$
in Hochschild cohomology. The map $l_\ast \colon \HH^\ast(B) \rightarrow \HH^\ast(eBe)$ is a homomorphism of strict Gerstenhaber algebras. 
\end{enumerate}
\end{Thm4}

Recall that for a subset $S \subseteq \HH^\ast(A)$ of homogeneous elements, the (\textit{weak}) \textit{Gerstenhaber ideal} generated by $S$ is the smallest homogeneous ideal $G(S)$ of $\HH^\ast(A)$ that contains $S$ and which satisfies
$$
\{ \gamma, \gamma' \}_A \in G(S) \quad \text{(for all $\gamma, \gamma' \in G(S)$).}
$$
If $\mathrm{char}(K) = 2$ or $2 \in K^\times$, and if $S = N$ is the set of all homogeneous nilpotent elements, then the quotient (Gerstenhaber) algebra $\HH^\ast_G(A) = \HH^\ast(A) / G(N)$ is commutative. In light of Theorems 3 and 4, we end by asking the following question.

\begin{Quest}[= Question \ref{conj:snashsol}]
Let $A$ be a finite dimensional algebra over a field $K$. If $N$ denotes the set of all homogeneous nilpotent elements in $\HH^\ast(A)$, is the commutative $K$-algebra
$$
\HH^\ast_G(A) = \HH^\ast(A) / G(N)
$$
finitely generated? More specifically, can we use (any of) the long exact sequences to deduce that for the one-point-extension $B = \Gamma[M] = \left( \begin{smallmatrix} K (\mathbb Z_2 \times \mathbb Z_2) & K\mathbb Z_2\\ 0 & K \end{smallmatrix} \right)$ of $\Gamma = K(\mathbb Z_2 \times \mathbb Z_2)$ by $M = K\mathbb Z_2$, the quotient algebra
$$
\HH^\ast_G(B) = \HH^\ast(B) / G(N)
$$
is finitely generated (if $\mathrm{char}(K) = 2$)?
\end{Quest}

Note that $B = \Gamma[M]$ is the initial counterexample to the finite generation conjecture of Snashall--Solberg, see \cite{SnSo04}, which was first mention in \cite{Xu08}, and studied in further detail in \cite{Sn09} thereafter. It remains to be checked whether, or when, the algebra $\HH^\ast_G(A)$ may be used to establish a fruitful support variety theory on a sufficiently broad class of $A$-modules. This question shall be studied in future work.
\medskip

This article is organised as follows. In Sections \ref{sec:prereq} we will built up the necessary foundations on exact and monoidal categories. Afterwards, in Section \ref{sec:gerstHH}, we will give a recap on the theory of Hochschild cohomology, and the definition of the Lie bracket associated with it. We are also going to briefly discuss Schwede's exact sequence interpretation of the Lie bracket. Section \ref{sec:recoll} concerns itself with recollements of module categories and provides all the detail to state and prove Theorem 1. In Sections \ref{sec:buchweitz} and \ref{sec:greensol} we apply Theorem 1 to long exact sequences due to Buchweitz and Green--Solberg, thus covering a proof of Theorem 3. We make further observations in the context of one-point-extensions and Happel's long exact sequence. Section \ref{sec:homepi} sets up the necessary peripherals on homological epimorphisms in oder to formulate and prove Theorem 2. The latter theorem, along with Theorem 1, will be used in Section \ref{sec:koenagase} to improve results relating the long exact sequences due to Koenig--Nagase, i.e., to obtain Theorem 4. Finally, in Section \ref{sec:snashsol}, we try to motivate how our results may be used to formulate a refined version of the finite generation conjecture of Snashall--Solberg.
\medskip

The present paper is a vastly extended version of a talk given during the \textit{International Conference on Representations of Algebras 2014} in Sanya City, China. I would like to thank L.\,Angeleri-H\"ugel, \O.\,Solberg and J.\,Vit\'oria for awakening my interest in the theory of recollements and homological epimorphisms through most intriguing presentations at various occasions throughout this year. Many thanks also go to C.\,Psaroudakis for valuable discussions on the topic. Finally, I would like to thank the referees for their careful reading and numerous comments. The research that led to this paper was supported by the project ``Triangulated categories in algebra'' (Norwegian Research Council project NFR 221893).


\section{Prerequisites on exact and monoidal categories}\label{sec:prereq}

\begin{nn}
Let us recall the notions of exact and monoidal categories and structure preserving functors (exact and monoidal functors) between them. For further details on exact categories, we refer to \cite{Ke90} and \cite{Qu72}, whereas the textbooks \cite{AgMa10} and \cite{MaL98} provide background material on monoidal categories. In the following section, and in fact for the entire article, we fix a commutative ring $K$.
\end{nn}

\begin{nn}
A full subcategory $\mathsf U$ of an abelian category $\A$ is \textit{extension closed in $\A$} if $0 \rightarrow U'' \rightarrow U \rightarrow U' \rightarrow 0$ is an exact sequence in $\A$ with $U', U'' \in \U$, then also $U \in \mathsf U$. An \textit{exact} $K$-category is a pair $(\C, i_\C)$ consisting of an additive $K$-category $\C$ and a full and faithful embedding $i_\C\colon \C \rightarrow \A_\C$ into an abelian $K$-category $\A_\C$, such that the essential image
$$
\Im(i_\C) = i_\C \C = \{ A \in \A_\C \mid i_\C(C) \cong A \text{ for some $C \in \C$}\}
$$
of $i_\C$ is an extension closed subcategory of $\A_\C$. Exact categories admit a sensible notion of exact sequences. Let $\C = (\C, i_\C)$ be an exact $K$-category. A sequence $0 \rightarrow C'' \rightarrow C \rightarrow C' \rightarrow 0$ is an \textit{admissible short exact sequence in $\C$}, if its image $0 \rightarrow i_\C(C'') \rightarrow i_\C(C) \rightarrow i_\C(C') \rightarrow 0$ under $i_\C$ is exact in $\A_\C$. Given such an admissible short exact sequence $0 \rightarrow C'' \rightarrow C \rightarrow C' \rightarrow 0$, the morphism $C'' \rightarrow C$ is called an \textit{admissible monomorphism}, whereas $C \rightarrow C'$ is an \textit{admissible epimorphism}. The class of admissible short exact sequences is closed under taking isomorphisms and direct sums (in the category of chain complexes over $\C$).
\end{nn}

\begin{nn}
Let $(\C,i_\C)$ be an exact $K$-category. We say that $\C$ is
\begin{enumerate}[\rm(1)]
\item \textit{closed under kernels of epimorphisms} if the functor $i_\C \colon \C \rightarrow \A_\C$ detects admissible epimorphisms, that is, if
$$
\text{$f \in \C$ is an admissible epimorphism} \quad \Longleftrightarrow \quad \text{$i_\C(f)$ is an epimorphism in $\A_\C$}.
$$
\item \textit{closed under cokernels of monomorphisms} if the functor $i_\C \colon \C \rightarrow \A_\C$ detects admissible monomorphisms, that is, if
$$
\text{$f \in \C$ is an admissible monomorphism} \quad \Longleftrightarrow \quad \text{$i_\C(f)$ is a monomorphism in $\A_\C$}.
$$
\end{enumerate}
\end{nn}

\begin{nn}
Each exact $K$-category $(\C, i_\C)$ is closed under taking pushouts along admissible monomorphisms and pullbacks along admissible epimorphisms. Assume that $(\D, i_\D)$ is another exact $K$-category, and let $\mathfrak X\colon \C \rightarrow \D$ be an \textit{exact} $K$-linear functor, that is, it takes admissible short exact sequences in $\C$ to admissible short exact sequences in $\D$. Each such functor preserves pushouts along admissible monomorphisms and pullbacks along admissible epimorphisms.
\end{nn}

\begin{lem}\label{lem:smallexactfunc}
Let $\C$ and $\D$ be exact $K$-categories, with corresponding abelian categories $\A_\C$ and $\A_\D$. Let $\mathfrak X \colon \C \rightarrow \A_\D$ be an exact and $K$-linear functor. Then the full subcategory $\C' \subseteq \C$ defined by 
$$
X \in \C \text{ belongs to $\C'$} \quad \Longleftrightarrow \quad \mathfrak X(X) \text{ belongs to $i_\D \D$}
$$
$($that is, $\C' = \mathfrak X^{-1}(i_\D \D))$ is an exact $K$-category $($along with the obvious inclusion $\C' \hookrightarrow \C \hookrightarrow \A_\C)$ and $\mathfrak X$ restricts to an exact functor $\mathfrak X \mathord{\upharpoonright}_{\C'}\colon \C' \rightarrow i_\D \D$.
\end{lem}

\begin{proof}
Additivity and $K$-linearity are obvious. If $0 \rightarrow i_\C (X'') \rightarrow X \rightarrow i_\C (X') \rightarrow 0$ is an exact sequence in $\A_\C$ with $X', X'' \in \C'$, it follows that $X \in i_\C \C$ by definition; say $X \cong i_\C(Y)$ for some object $Y \in \C$. We have thus an admissible short exact sequence $0 \rightarrow X'' \rightarrow Y \rightarrow X' \rightarrow 0$ in $\C$, and hence, after applying $\mathfrak X$ to it, such in $\D$. As $\mathfrak X(X')$ and $\mathfrak X(X'')$ belong to $i_\D \D$, so must $\mathfrak X(Y)$. Therefore $Y$ belongs to $\C'$, as required.
\end{proof}

\begin{nn}
For an integer $n \geqslant 1$, a sequence
$$
S \quad \equiv \quad 0 \longrightarrow C'' \longrightarrow C_{n-1} \longrightarrow \cdots \longrightarrow C_0 \longrightarrow C' \longrightarrow 0
$$
of morphisms in an exact $K$-category $(\C, i_\C)$ is called an \textit{admissible $n$-extension} (\textit{of $C'$ by $C''$}) in case $i_\C S$ is exact in $\A_\C$, and $\Ker(i_\C(C_0 \rightarrow C'))$ and $\Ker(i_\C(C_k \rightarrow C_{k-1}))$ belong to $i_\C \C$ for all $k = 1, \dots, n-1$. As for abelian categories, we let $\Ext^n_\C(C',C'')$ be the set of $n$-extensions of $C'$ by $C''$ modulo the equivalence relation generated by \textit{morphisms} of admissible $n$-extensions.
\end{nn}

\begin{nn}\label{nn:extclosed}
Let $\C$ be an additive subcategory of an abelian $K$-category $\A$ and let $n \geqslant 1$ be an integer. The category $\C$ is \textit{$n$-extension closed in $\A$}, if the maps
$$
\Ext^i_\C(C,D) \longrightarrow \Ext^i_\A(C,D) \quad \text{(for all $0 \leqslant i \leqslant n$)}
$$
induced by the inclusion functor $\C \rightarrow \A$ are isomorphisms for every pair $C,D$ of objects in $\C$. If $\C$ is $n$-extension closed for all $n \in \mathbb N$, then we say that $\C$ is \textit{entirely extension closed in $\A$}. Note that $\C$ being $1$-extension closed is equivalent to $(\C, \C \hookrightarrow \A)$ being an exact category. One has the following criterion as to when $\C$ is entirely extension closed in $\A$.
\end{nn}

\begin{prop}[see {\cite[Prop.\,2.4.6]{He14b}}]\label{prop:entireextclo}
Assume that $\C$ is $1$-extension closed in $\A$ and that $\A$ has enough projective objects. If $\Proj(\A) \subseteq \C$ and $\C$ is, when viewed as an exact category, closed under kernels of epimorphisms, then $\C$ is entirely extension closed in $\A$.
\end{prop}

\begin{rem}
Independently, Coulembier--Mazorchuk and Psaroudakis introduced and studied entirely extension closed (abelian) subcategories in the context of the representation theory of algebraic groups and associative algebras (see \cite{CoMa14a, CoMa14b} and \cite{Ps14}). However, their terminology differs from ours. Coulembier--Mazorchuk refer to those subcategories as \textit{extension full subcategories}, while Psaroudakis investigated \textit{homological embeddings}, i.e., exact functors between abelian categories giving rise to isomorphisms on $\Ext$-groups.
\end{rem}

Let us turn to monoidal categories.

\begin{nn}
Recall that a \textit{monoidal category} is a 6-tuple $(\mathsf C, \otimes, \mathbbm 1, \alpha, \lambda, \varrho)$, where $\mathsf \C$ is a category, $\otimes\colon \C \times \C \rightarrow \C$ is a functor, $\mathbbm 1$ is an object in $\C$, and
\begin{gather*}
\alpha\colon - \otimes (- \otimes-) \longrightarrow (- \otimes -) \otimes - \ , \\
\lambda\colon \mathbbm 1 \otimes - \longrightarrow \Id_{\C} \ , \\
\varrho\colon - \otimes \mathbbm 1 \longrightarrow \Id_{\C}
\end{gather*}
are isomorphisms of functors such that, for all objects $W, X,Y,Z$ in $\C$,
$$
\xymatrix@C=40pt{
W \otimes (X \otimes (Y \otimes Z)) \ar[r]^-{\alpha_{W,X,Y \otimes Z}} \ar[d]_-{W \otimes \alpha_{X,Y,Z}} & (W\otimes X) \otimes (Y \otimes Z) \ar[r]^{\alpha_{W \otimes X, Y, Z}} & ((W \otimes X) \otimes Y) \otimes Z \ \ \\
W \otimes ((X \otimes Y) \otimes Z) \ar[rr]^-{\alpha_{W,X \otimes Y,Z}} & & (W \otimes (X \otimes Y)) \otimes Z \ar[u]_{\alpha_{W,X,Y} \otimes Z} \ ,
}
$$
commutes and $(\varrho_X \otimes Y) \circ \alpha_{X,\mathbbm 1,Y} = X \otimes \lambda_Y$. In this situation, $\otimes$ is a \textit{monoidal} (or \textit{tensor}) \textit{product functor for $\C$} and $\mathbbm 1$ is the \textit{$($tensor$)$ unit} of $\otimes$.
\end{nn}
\begin{rem}\label{rem:opposite} Let $(\C, \otimes, \mathbbm 1, \alpha, \lambda, \varrho)$ be a monoidal category.
\begin{enumerate}[\rm(1)]
\item We will often suppress a huge part of the structure morphisms and simply write $(\C, \otimes, \mathbbm 1)$ instead of $(\C, \otimes, \mathbbm 1, \alpha, \lambda, \varrho)$; if they are needed without priorly having been mentioned, we will refer to them as $\alpha_\C$, $\lambda_\C$ and $\varrho_\C$.
\item It follows from the axioms (cf.\,\cite[Prop.\,1.1]{JoSt93}) that the following equations hold true for all $X, Y, Z \in \Ob \C$:
$$
\lambda_\bbm1 = \varrho_\bbm1, \quad \varrho_{X \otimes Y} \circ \alpha_{X,Y,\bbm1} = X \otimes \varrho_Y, \quad (\lambda_X \otimes Y) \circ \alpha_{\bbm1, X, Y} = \lambda_{X \otimes Y} \, .
$$
\item Note that if $(\mathsf C, \otimes, \mathbbm 1, \alpha, \lambda, \varrho)$ is a monoidal category, then so is $\mathsf C^\op$ together with the structure morphisms $\alpha^{-1}$, $\lambda^{-1}$ and $\varrho^{-1}$.
\end{enumerate}
\end{rem}
\begin{nn}
We say that a monoidal category $(\C, \otimes, \mathbbm 1)$ is a \textit{tensor $K$-category}, if $\C$ is $K$-linear and the tensor product functor $\otimes\colon \C \times \C \rightarrow \C$ is $K$-bilinear on morphisms, that is, it factors through the \textit{tensor product category} $\C \otimes_K \C$ which is defined as follows:
\begin{align*}
\Ob(\C \otimes_K \C) &{\coloneqq} \Ob(\C \times \C),\\
\Hom_{\C \otimes_K \C}(\underline{X}, \underline{Y}) &{\coloneqq} \Hom_\C(X_1,Y_1) \otimes_K \Hom_\C(X_2, Y_2)
\end{align*}
for objects $\underline{X} = (X_1, X_2)$ and $\underline{Y} = (Y_1, Y_2)$ in $\C \times \C$. 
\end{nn}
\begin{nn}\label{nn:monoclosure}
The triple $(\C, \otimes, \bbm1)$ is \textit{weak exact monoidal $K$-category} if $\C$ is an exact $K$-category and $(\C, \otimes, \bbm1)$ is a tensor $K$-category such that every object in $\C$ is either \textit{flat} (that is, the functors $X \otimes - \colon \C \rightarrow \C$ are exact for all $X \in \Ob \C$) or \textit{coflat} (that is, the functors $- \otimes X \colon \C \rightarrow \C$ are exact for all $X \in \Ob \C$). A \textit{strong exact monoidal category} is a weak exact monoidal $K$-category whose underlying exact category is closed under kernels of epimorphisms or under cokernels of monomorphisms. Note that these definitions are not quite the same as introduced in \cite[Sec.\,1.2]{He14b} and \cite[Sec.\,3.3]{He14b}, however weak/strong exact monoidal categories in the above sense are weak/strong exact monoidal categories in the sense of \cite[Def.\,1.2.6]{He14b}/\cite[Def.\,3.3.3]{He14b}. For the remainder of this article, the term exact monoidal category will mostly stand for \textit{strong} exact monoidal category, unless explicitly stated otherwise.

Let $\U$ be a full and extension closed subcategory of $\C$ that contains $\bbm1$. Define $\overline \U$ to be the full subcategory of $\U$ consisting of all objects $X \in \U$ such that $X \otimes Y \in \U$ for all objects $Y \in \U$.
\end{nn}
\begin{lem}\label{lem:monoclosure}
Keeping the notations from above, the triple $(\overline \U, \otimes, \bbm1)$ is a full, additive, extension closed and monoidal subcategory of $(\C,\otimes, \bbm1)$. It is closed under kernels of epimorphisms $($cokernels of monomorphisms$)$ if $\U$ is closed under kernels of epimorphisms $($cokernels of monomorphisms$)$.
\end{lem}
\begin{proof}
Additivity follows from the additivity of $\otimes$. If $0 \rightarrow X'' \rightarrow X \rightarrow X' \rightarrow 0$ is an admissible exact sequence, then so is
$$
0 \longrightarrow X'' \otimes Y \longrightarrow X \otimes Y \longrightarrow X' \otimes Y \longrightarrow 0
$$
for all $Y \in \U$. Therefore, if $X', X''$ are in $\overline \U$, it follows that $X \otimes Y$ belongs tu $\U$, thus, that $X$ belongs to $\overline \U$. Let $X', X'' \in \overline \U$ and $Y \in \U$. Then, by the associativity isomorphisms,
$$
(X' \otimes X'') \otimes Y \cong X' \otimes (X'' \otimes Y) \ \in \ \U,
$$
i.e., $X' \otimes X''$ is a member of $\overline \U$. Of course, $\bbm1 \in \overline \U$ by the unit isomorphisms.
\end{proof}
\begin{nn}\label{def:monoidalfunc}
We are going to recall the definition of certain structure preserving functors between monoidal categories. Let $(\mathsf C, \otimes_{\mathsf C}, \mathbbm 1_{\mathsf C})$ and $(\mathsf D, \otimes_{\mathsf D}, \mathbbm 1_{\mathsf D})$ be monoidal categories. Let $\mathfrak A\colon \mathsf C \rightarrow \mathsf D$ be a functor, and
\begin{align*}
& \phi_{X,Y}\colon \mathfrak A X \otimes_{\mathsf D} \mathfrak A Y \longrightarrow \mathfrak A (X \otimes_{\mathsf C} Y),\\
& \psi_{X,Y}\colon \mathfrak A (X \otimes_{\mathsf C} Y) \longrightarrow \mathfrak A X \otimes_{\mathsf D} \mathfrak A Y,
\end{align*}
be natural morphisms in $\D$ (for $X,Y \in \Ob\mathsf C$). Further, let $\phi_0\colon \mathbbm 1_{\mathsf D} \rightarrow \mathfrak A \mathbbm 1_{\mathsf C}$ and $\psi_0\colon \mathfrak A \mathbbm 1_{\mathsf C} \rightarrow \mathbbm 1_{\mathsf D}$ be morphisms in $\mathsf D$.

The triple $(\mathfrak A, \phi, \phi_0)$ is called an \textit{almost strong monoidal functor} if $\phi_0$ is invertible and the following diagrams commute for all $X,Y,Z \in \Ob \mathsf C$.

\begin{equation*}
\xymatrix@C=40pt{
\mathfrak A X \otimes_\D (\mathfrak A Y \otimes_\D \mathfrak A Z) \ar[r]^{\mathfrak A X \otimes_\D \phi_{Y,Z}} \ar[d]_{\alpha_\D\mathfrak A} & \mathfrak A X
\otimes_\D \mathfrak A(Y \otimes_\C Z) \ar[r]^{\phi_{X,Y \otimes_\C Z}} & \mathfrak A(X \otimes_\C (Y \otimes_\C Z)) \ar[d]^{\mathfrak A \alpha_\C}\\
(\mathfrak A X \otimes_\D \mathfrak A Y) \otimes_\D \mathfrak A Z \ar[r]^-{\phi_{X,Y} \otimes_\D \mathfrak A Z} & \mathfrak A (X \otimes_\C Y) \otimes_\D \mathfrak A Z
\ar[r]^-{\phi_{X \otimes_\C Y,Z}} & \mathfrak A ((X \otimes_\C Y) \otimes_\C Z)
}
\end{equation*}
\begin{equation*}
\xymatrix{
\mathbbm 1_\D \otimes_\C \mathfrak A X  \ar[d]_{\phi_0 \otimes_\D \mathfrak A X} \ar[r]^-{\lambda_{\D}\mathfrak A} & \mathfrak A X  \\
\mathfrak A \mathbbm 1_\C \otimes_\D \mathfrak A X \ar[r]^-{\phi_{\mathbbm 1_\C, X}} & \mathfrak A(\mathbbm 1_\C \otimes_\C X) \ar[u]_{\mathfrak A \lambda_{\C}}
}
\quad
\xymatrix{
\mathfrak A X \otimes_\D  \mathbbm 1_\D \ar[d]_{\mathfrak A X \otimes_\D \phi_0} \ar[r]^-{\varrho_{\D}\mathfrak A} & \mathfrak A X  \\
\mathfrak A X \otimes_\D \mathfrak A \mathbbm 1_\C \ar[r]^-{\phi_{X, \mathbbm 1_\C}} & \mathfrak A(X \otimes_\C \mathbbm 1_\C)
\ar[u]_{\mathfrak A \varrho_{\C}}
}
\end{equation*}
The triple $(\mathfrak A, \psi, \psi_0)$ is called an \textit{almost costrong monoidal functor} if $(\mathfrak A^\op, \psi, \psi_0)$ is an almost strong monoidal functor. The triple $(\mathfrak A, \phi, \phi_0)$ is called a \textit{strong monoidal functor} if it is an almost strong monoidal functor and $\phi$ is invertible, whereas $(\mathfrak A, \psi, \psi_0)$ is called a \textit{costrong monoidal functor} if it is an almost costrong monoidal functor and $\psi$ is invertible.
\end{nn}

The following two examples will reappear at later stages of this article.

\begin{exa}
Let $A$ be a $K$-algebra which is projective when considered as a $K$-module and $A^\ev = A \otimes_K A^\op$ the \textit{enveloping algebra} of $A$ over $K$, with factorwise multiplication. Note that there are forgetful functors $\Mod(A^\ev) \rightarrow \Mod(A)$ and $\Mod(A^\ev) \rightarrow \Mod(A^\op)$. Let $\P_\lambda$ and $\P_\varrho(A)$ be the following full subcategories of $\Mod(A^\ev)$. A module $M \in \Mod(A^\ev)$ belongs to
\begin{enumerate}
\item $\P_\lambda(A)$ if, and only if, $M$ is projective as a left $A$-module.
\item $\P_\varrho(A)$ if, and only if, $M$ is projective as a right $A$-module.
\end{enumerate}
Both categories are non-empty, as they contain $A$ and $A \otimes_K A$. As short exact sequences ending in a projective module split, $\P_\lambda(A)$ is closed under extensions and kernels of admissible epimorphisms. Summands and arbitrary direct sums of projectives are projective, so that $\P_\lambda(A)$ is closed under taking direct summands and arbitrary direct sums in $\Mod(A^\ev)$. It follows that $\P_\lambda(A)$ is entirely extension closed, due to Proposition \ref{prop:entireextclo}. Moreover, if $M$ and $N$ belong to $\P_\lambda(A)$, then so does $M \otimes_A N$, since $\Hom_A(M \otimes_A N, -) \cong \Hom_A(N, \Hom_A(M,-))$ is an exact functor. Thus $\P_\lambda(A)$ is monoidal, and the functors $- \otimes_A M \colon \P_\lambda(A) \rightarrow \P_\lambda(A)$ are exact for all $M \in \P_\lambda(A)$, thus $\P_\lambda(A)$ is a strong exact monoidal subcategory of $\Mod(A^\ev)$. Likewise, the analogue statements hold true for $\P_\varrho(A)$. It follows that the category $\P(A) = \P_\lambda(A) \cap \P_\varrho(A)$ of $A^\ev$-modules which are projective on either side is entirely extension closed and strong exact monoidal as well.
\end{exa}

\begin{exa}
Let $A$ be a $K$-algebra which is flat as a $K$-module. Replacing the term $A$-projective (on the left or right) by $A$-flat in the previous example, we obtain three full subcategories $\F_\lambda(A)$, $\F_\varrho(A)$ and $\F(A) = \F_\lambda(A) \cap \F_\varrho(A)$ of $\Mod(A^\ev)$. They contain $A$ and $A \otimes_K A$, as $A$ is $K$-flat. Moreover, by the Snake Lemma, they are closed under taking extensions and tensor products over $A$, thus are weak exact monoidal subcategories of $(\Mod(A^\ev), \otimes_A, A)$. However, neither the kernel of a map $M \twoheadrightarrow N$ in $\F(A)$ nor the cokernel of a map $M \hookrightarrow N$ in $\F(A)$ needs to be $A$-flat on any side. Therefore one cannot expect those subcategories to be entirely extension closed in general.
\end{exa}


\section{Gerstenhaber algebras and Hochschild cohomology}\label{sec:gerstHH}

Fix a unital and associative $K$-algebra $A$. The symbol $\otimes$ will always stand for $\otimes_K$, i.e., the tensor product over the base ring $K$. Let $A^\ev = A \otimes A^\op$ be the enveloping algebra of $A$ over $K$. It is a very well-known fact, that $A$-bimodules with symmetric $K$-action bijectively correspond to left modules over $A^\ev$. In this section, we will recall the definition of Hochschild cohomology, and its (higher) structures, as introduced in \cite{CaEi56}, \cite{Ge63} and \cite{Ho45}.

\subsection{Reminder on Hochschild cohomology}
Let $M$ be a $A^\ev$-module. The \textit{Hochschild cocomplex} $\mathbb C(A,M) = (C^\ast(A,M), \partial_M)$ with coefficients in $M$ is the cocomplex conentrated in non-negative degrees which is given by
$$
C^n(A,M) = \Hom_K(A^{\otimes n},M) \quad \text{(for $n \geqslant 0$)}
$$
and $\partial_M^n \colon C^n(A,M) \rightarrow C^{n+1}(A,M)$,
\begin{align*}
\partial_M^n(f)(a_1 \otimes \cdots a_{n+1}) = & a_1 f(a_2 \otimes \cdots \otimes a_{n+1})\\ &+ \sum_{i=1}^{n}(-1)^n f(a_1 \otimes \cdots \otimes a_{i-1} \otimes a_i a_{i+1} \otimes a_{i+1} \otimes \cdots \otimes a_{n+1})\\
&+ (-1)^{n+1}f(a_1 \otimes \cdots \otimes a_{n})a_{n+1}.
\end{align*}
The Hochschild cocomplex admits an exterior pairing, in that there is a map
$$
\smallsmile \colon C^m(A,M) \times C^n(A,N) \longrightarrow C^{m+n}(A,M \otimes_A N)
$$
for any pair $M, N \in \Mod(A^\ev)$. Namely,
$$
(f \smallsmile g)(a_1 \otimes \cdots a_{m + n}) = f(a_1 \otimes \cdots \otimes a_{m}) \otimes g(a_{m+1} \otimes \cdots \otimes a_{m+n}).
$$
In fact, through $\smallsmile$, and after identifying $A \otimes_A A \cong A$, $M \otimes_A A \cong M \cong A \otimes_A M$, the DG module $\mathbb C(A,A)$ is a DG $K$-algebra, and $\mathbb C(A,M)$ will be a left and a right DG module over it. In particular, the cohomology of $\mathbb C(A,A)$ will be a graded $K$-algebra. We refer to the map $\smallsmile$ on $\mathbb C(A,A)$ (and also to the induced map in cohomology) as the \textit{cup product}.

\begin{defn}
For an integer $n \geqslant 0$, the \textit{$n$-th Hochschild cohomology module} with coefficients in $M$ is given by
$$
\HH^n(A,M) = H^n \mathbb C(A,M).
$$
The graded module
$$
\HH^\ast(A,M) = \bigoplus_{n \geqslant 0}\HH^n(A,M)
$$
is the \textit{Hochschild cohomology module} of $A$ with coefficients in $M$. We abbriviate $\HH^n(A) = \HH^n(A,A)$ for $n \geqslant 0$ and call $\HH^\ast(A)$ the \textit{Hochschild cohomology algebra} of $A$.
\end{defn}

\begin{nn}
There are two immediate observations.
\begin{enumerate}[\rm(1)]
\item The module $\HH^0(A,M)$ coincides with $M^A$, where
$$
M^A = \{m \in M \mid am = ma \text{ for all $a \in A$}\}.
$$
In particular, $\HH^0(A) = Z(A)$ is the center of $A$.
\item The module $\HH^1(A,M)$ coincides with the module $\Out_K(A,M)$ of \textit{outer derivations} which is given by
$$
\Out_K(A,M) = \frac{\Der_K(A,M)}{\Inn_K(A,M)},
$$
where $\Der_K(A,M)$ denotes the derivations of $M$, and $\Inn_K(A,M)$ the submodule of inner derivations.
\end{enumerate}
\end{nn}

\begin{nn}
The \textit{bar resolution} $\mathbb BA = (B_\ast, \beta_\ast)$ of $A$, is the following exact resolution of $A$ by $A^\ev$-modules. To begin with, $B_n = A^{\otimes(n+2)}$ is the $(n+2)$-fold tensor product of $A$ with itself (over $K$ and for $n \geqslant 0$). The $A^\ev$-linear map $\beta_{n+1}\colon B_{n+1} \rightarrow B_n$,
$$
\beta_{n+1}(a_0 \otimes \cdots \otimes a_{n+2}) = \sum_{i=0}^{n+1}(-1)^i a_0 \otimes \cdots \otimes a_i a_{i+1} \otimes \cdots \otimes a_{n+2},
$$
turns $\mathbb BA = (B_\ast, \beta_\ast)$ into a complex, which is acyclic in all degrees but in degree $0$, wherein its homology is isomorphic to $A$. The multiplication map $\mu\colon A \otimes A \rightarrow A$ provides a suitable augmentation $\mathbb BA \rightarrow A \rightarrow 0$.

Now, the adjunction isomorphism
$$
\Hom_{A^\ev}(A^{\otimes(n+2)}, M) \longrightarrow \Hom_K(A^{\otimes n}, M), \ \varphi \mapsto \varphi(1 \otimes - \otimes \cdots \otimes - \otimes 1)
$$
is compatible with the differentials $\partial_M$ and $\Hom_{A^\ev}(\mathbb \beta_\ast, M)$, that is, the cocomplexes $\mathbb C(A,M)$ and $\Hom_{A^\ev}(\mathbb BA, M)$ are isomorphic. Thus $\HH^\ast(A,M) \cong H^\ast \Hom_{A^\ev}(\mathbb BA, M)$ and, since $\mathbb BA \rightarrow A \rightarrow 0$ is exact, there is a graded map
$$
\chi_{M}^\ast \colon \HH^\ast(A,M) \longrightarrow \Ext^\ast_{A^\ev}(A,M).
$$
It is given by sending a cocycle $\varphi \in \Ker \Hom_{A^\ev}(\beta_{n+1}, M)$ to the equivalence class defined by the lower sequence in the pushout diagram below.
$$
\xymatrix@C=20pt{
\cdots \ar[r]^-{\beta_{n+1}} & A^{\otimes (n + 2)} \ar[r]^-{\beta_{n}} \ar[d]_-\varphi & A^{\otimes (n + 1)} \ar[r]^-{\beta_{n-1}} \ar[d] & \cdots \ar[r]^-{\beta_2} & A^{\otimes 3} \ar[r]^-{\beta_1} \ar@{=}[d] & A \otimes A \ar[r]^-{\beta_0} \ar@{=}[d] & A \ar[r] \ar@{=}[d] & 0\\
0 \ar[r] & M \ar[r] & Q \ar[r] & \cdots \ar[r]^-{\beta_2} & A^{\otimes 3} \ar[r]^-{\beta_1} & A \otimes A \ar[r]^-{\beta_0} & A \ar[r] & 0
}
$$
This map respects the graded ring structures if $M = A$, but will in general not be a bijection (as we will remark later, $\HH^\ast(A)$ is always \textit{graded commutative}, whereas a sufficient criterion for $\Ext^\ast_{A^\ev}(A,A)$ being graded commutative is $\Tor_i^K(A,A) = 0$ for all $i > 0$; see \cite[Sec.\,2.2]{BuFl08} and \cite{SnSo04}). However, one easily checks that $\mathbb BA$ will be a resolution by projective $A^\ev$-modules, if $A$ is projective over $K$. Thus the above map is going to be an ismorphism,
$$
\chi^\ast_M \colon \HH^\ast(A,M) \xrightarrow{\ \sim \ } \Ext^\ast_{A^\ev}(A,M),
$$
if $A$ is $K$-projective.
\end{nn}

\subsection{The Gerstenhaber bracket in Hochschild cohomology} Let us first recall the definition of a Gerstenhaber algebra over $K$ (following \cite{Ge63}, \cite{GeSch86} and \cite{Lei80}).

\begin{defn}\label{def:galgebra}
Let $G = \bigoplus_{n \in \Z}{G^n}$ be a graded $K$-algebra. Further, let $\{-,-\}\colon G \times G \rightarrow G$ be a $K$-bilinear map of degree $-1$ (that is, $\{a,b\} \in G^{\abs{a}+\abs{b}-1}$ for all homogeneous $a,b \in G$). The pair $(G,\{-,-\})$ is a \textit{Gerstenhaber algebra} over $K$ if
\begin{enumerate}[\rm(G1)]
\item $G$ is graded commutative, i.e., $ab = (-1)^{\abs{a}\abs{b}}ba$ for all homogeneous $a,b \in G$;
\item $\{a,b\} = -(-1)^{(\abs{a}-1)(\abs{b}-1)}\{b,a\}$ for all homogeneous $a,b \in G$;
\item $\{a,a\} = 0$ for all homogeneous $a \in G$ of odd degree;
\item $\{\{a,a\},a\} = 0$ for all homogeneous $a \in G$ of even degree;
\item\label{def:galgebra:5} the graded Jacobi identity holds:
$$
\{a,\{b,c\}\} = \{\{a,b\},c\} + (-1)^{(\abs{a}-1)(\abs{b}-1)}\{b,\{a,c\}\}
$$
for all homogeneous $a,b,c \in G$;
\item\label{def:galgebra:6} the graded Poisson identity holds:
$$
\{a,bc\} = \{a,b\}c + (-1)^{(\abs{a}-1)\abs{b}} b\{a,c\}
$$
for all homogeneous $a,b,c \in G$.
\end{enumerate}
Assume that $(G,\{-,-\})$ is a Gerstenhaber algebra over $K$. We call $(G,\{-,-\})$ a \textit{strict} Gerstenhaber algebra over $K$ if there is a map $sq\colon G^{2\ast} \rightarrow G^{4\ast-1}$ of degree $0$ such that
\begin{enumerate}[\rm(G1)]
\setcounter{enumi}{6}
\item\label{def:galgebra:7} $sq(ra) = r^2 sq(a)$ for all $r \in K$ and all homogeneous $a \in G^{2\ast}$;
\item\label{def:galgebra:8} $sq(a+b) = sq(a) + sq(b) + \{a,b\}$ for all homogeneous $a,b \in G^{2\ast}$;
\item $\{a,sq(b)\} = \{\{a,b\},b\}$ for all homogeneous $a,b \in G^{2\ast}$;
\item $sq(ab) = a^2sq(b) + sq(a)b^2 + a\{a,b\}b$ for all homogeneous $a,b \in G^{2\ast}$.
\end{enumerate}
\end{defn}

\begin{nn}
The map $\{-,-\}$ is called a \textit{Gerstenhaber bracket} for $G$, whereas $sq$ is a \textit{$($devided$)$ squaring map} for the Gerstenhaber algebra $(G, \{-,-\})$. Note that any graded commutative $K$-algebra can be viewed as (strict) a Gerstenhaber algebra over $K$ with trivial bracket (and trivial squaring map).
\end{nn}

\begin{nn}
The axioms (G\ref{def:galgebra:5}) and (G\ref{def:galgebra:6}) may be read as follows: The graded Jacobi identity measures the failure of $\{-,-\}$ from being associative, whereas the graded Poisson identity translates to $\{a,-\}$ being a graded derivation of $G$ of degree $\abs{a} - 1$ (for $a \in G$ homogeneous). Also note that by (G\ref{def:galgebra:7}) and (G\ref{def:galgebra:8}), there is precisely one squaring map $sq$ for a Gerstenhaber algebra $(G,\{-,-\})$ if $2 \in K^\times$, in which case it is given by $sq(a) = 2^{-1}\{a, a\}$.
\end{nn}

\begin{nn}\label{nn:Gbracket}
Let $M$ be a fixed $A^\ev$-module and $f \in C^m(A,M)$, $g \in C^n(A,A)$ for integers $m, n \geqslant 0$. For $i=1, \dots, m$, let $f \bullet_i g \in C^{m + n -1}(A,M)$
\begin{align*}
&\qquad(f \bullet_i g)(a_1 \otimes \cdots \otimes a_{m+n-1})\\ = f(a_1 \otimes \cdots & \otimes a_{i-1} \otimes g(a_i \otimes \cdots \otimes a_{i+n-1}) \otimes a_{i+n} \otimes \cdots \otimes a_{m+n-1}).
\end{align*}
Denote by $f \bullet g \in C^{m+n-1}(A,M)$ the alternating sum of the $f \bullet_i g$:
\begin{align*}
f \bullet g = \sum_{i=1}^m (-1)^{(i-1)(n-1)}f \bullet_i g.
\end{align*}
The product $\bullet$ is, in general, non-unital and highly non-associative. However, the external pairing $\smallsmile$ and $\bullet$ are related by the following fundamental formula:
\begin{equation}\tag{$\dagger$}\label{eq:fundform}
\partial_M(f \bullet g) + (-1)^n \partial_M(f) \bullet g = f \bullet \partial_A(g) + (-1)^{n}[g \smallsmile f - (-1)^{mn}f \smallsmile g].
\end{equation}
See \cite[Thm.\,3]{Ge63} for a proof.
\end{nn}

\begin{nn}
The fundamental formula (\ref{eq:fundform}) yields three important insights.
\begin{enumerate}[\rm(1)]
\item If $M = A$, and $f$ and $g$ are \textit{cocycles} (i.e., $\partial_A(f) = 0 = \partial_A(g)$), then $f \smallsmile g = (-1)^{mn} g \smallsmile f$. It follows that $\HH^\ast(A)$ is a graded commutative $K$-algebra.
\item The map
\begin{align*}
\{-,-&\} \colon C^m(A,A) \times C^n(A,A) \longrightarrow C^{m+n-1}(A,A)\\
&\{f,g\} = f \bullet g - (-1)^{(m-1)(n-1)}g \bullet f
\end{align*}
induces a well-defined map
$$
\{-,-\} = \{-,-\}_A \colon \HH^m(A) \times \HH^n(A) \longrightarrow \HH^{m+n-1}(A).
$$
\item The map
\begin{align*}
sq \colon C^{2n}(A,A) \longrightarrow C^{4n-1}(A,A), \ sq(f) = f \bullet f
\end{align*}
induces a well-defined map
$$
sq = sq_A \colon \HH^{2n}(A) \longrightarrow \HH^{4n-1}(A).
$$
\end{enumerate}
It is the main observation of \cite{Ge63} that the hereby obtained quadruple $(\HH^\ast(A), \smallsmile, \{-,-\}, sq)$ is a strict Gerstenhaber algebra over $K$, in the sense of Definition \ref{def:galgebra}.
\end{nn}

\subsection{Schwede's loop bracket} Let us briefly recall Schwede's exact sequence interpretation of the Gerstenhaber bracket in Hochschild cohomology; we will give a rough outline, and refer to \cite{He16a}, \cite{He14b} and \cite{Sch98} for details. Let $A$ be a $K$-algebra which is projective as a $K$-module. For an integer $n \geqslant 1$, we let $\mathcal Ext^n_{A^\ev}(A,A)$ be the \textit{category of $n$-self extensions of $A$ over $A^\ev$}. Its objects are $n$-extensions $S$ of the form $0 \rightarrow A \rightarrow E_{n-1} \rightarrow \cdots \rightarrow E_0 \rightarrow A \rightarrow 0$ in $\Mod(A^\ev)$, whereas the morphisms are given by \textit{morphisms of extensions}, i.e., morphisms $(f_n)_{n \in \mathbb Z} \colon S \rightarrow S'$ of complexes with $f_{-1} = f_n = \id_A$. The $0$-th homotopy group $\pi_0 \mathcal Ext^n_{A^\ev}(A,A)$ of $\mathcal Ext^n_{A^\ev}(A,A)$ coincides with $\Ext^n_{A^\ev}(A,A)$. Moreover, by a theorem of Retakh, see \cite{Re86}, one has isomorphisms
$$
\Ext^n_{A^\ev}(A,A) = \pi_0 \mathcal Ext^{n-1}_{A^\ev}(A,A) \xrightarrow{\, \sim \, } \pi_1(\mathcal Ext^n_{A^\ev}(A,A), S)
$$
for each base point $S \in \Ob \mathcal Ext^n_{A^\ev}(A,A)$. Given $\alpha \in \HH^m(A)$ and $\beta \in \HH^n(A)$, one may thus think of $\{\alpha, \beta\}_A$ as an element of $\pi_1 \mathcal Ext^{m+n}_{A^\ev}(A,A)$ (at some basepoint $S'$), that is, a \textit{loop} (see \cite{Qu72}, \cite{Sch98}) in the category $\mathcal Ext^{m+n}_{A^\ev}(A,A)$ based at $S'$. In fact, if $\alpha$ and $\beta$ are the equivalence classes of extensions $S = S(\alpha)$ and $T = T(\beta)$, the loop $\Omega_A(S,T)$ that corresponds to $\{\alpha, \beta\}_A$ may be realised as
$$
\xymatrix@!C=15pt@R=15pt{
& S \boxtimes_A T \ar[dr] \ar[dl] & \\
S \# T & & (-1)^{mn} T \# S\\
& (-1)^{mn} T \boxtimes_A S \ar[ur] \ar[ul] &
}
$$
by \cite[Thm.\,3.1]{Sch98}, with base point $S \#T$, where $S \# T$ denotes the Yoneda composite of $S$ and $T$. The $(m + n)$-extension $S \boxtimes_A T$ is obtained by taking the tensor product of the truncated complexes $0  \rightarrow A \rightarrow E_{m-1} \rightarrow \cdots \rightarrow E_0$ and  $0 \rightarrow A \rightarrow F_{n-1} \rightarrow \cdots \rightarrow F_0$, denoted by $S^\natural$ and $T^\natural$ for short, followed by the canonical augmentation $E_0 \otimes_A F_0 \twoheadrightarrow A \otimes_A A \cong A$. Caution is advised, as the complex $S \boxtimes_A T$ a priori does not need to be acyclic, which however may be bypassed by, for instance, considering the categories $\mathcal Ext^n_{\P(A)}(A,A)$ of admissible extensions over $\P(A)$ instead of $\mathcal Ext^n_{A^\ev}(A,A)$ (recall that $\Ext^n_{\P(A)}(A,A) \cong \Ext^n_{A^\ev}(A,A)$). Schwede's construction also covers an interpretation of the squaring map; the loop corresponding to $sq_A(\alpha)$ for some $\alpha \in \HH^{2n}(A)$ is given by the northern hemisphere of the loop $\Omega_A(S,S)$.

Note that the loop above only covers the cases $m, n \geqslant 1$; see \cite{He16a} for the remaining cases $m \geqslant 0$, $n = 0$.

\begin{lem}\label{lem:loopcoinc}
Assume that $S \in \mathcal Ext^m_{A^\ev}(A,A)$ and $T \in \mathcal Ext^n_{A^\ev}(A,A)$ such that $S \boxtimes_A T$ is exact. Let $\mathsf C$ be an exact monoidal subcategory of $(\Mod(A^\ev, \otimes_A, A)$ which is entirely extension closed. Assume further, that there are morphisms $S' \rightarrow S$ and $T' \rightarrow T$ of extensions with $S' \in \mathcal Ext^m_{\C}(A,A)$ and $T' \in \mathcal Ext^m_{\C}(A,A)$. Then there exists a morphism $v \colon S' \# T' \rightarrow S \# T$ and the conjugate of the loop $\Omega_A(S,T)$ by $f$ is homotopically equivalent to $\Omega_A(S',T')$.
\end{lem}

\begin{proof}
Of course, the morphisms $S' \rightarrow S$ and $T' \rightarrow T$ may be combined to obtain a morphism $v \colon S' \# T' \rightarrow S \# T$ as claimed. They also yield a morphism $u \colon S' \boxtimes_A T' \rightarrow S \boxtimes_A T$. These morphisms render the square
$$
\xymatrix@!C=35pt@R=15pt{
& S' \boxtimes_A T' \ar[r]^-u \ar[dl] & S \boxtimes_A T \ar[dl]\\
S' \# T' \ar[r]^-{\, v} & S \# T &
}
$$
commutative, where $S' \boxtimes_A T' \rightarrow S' \# T'$ is the northwestern face morphism in the loop $\Omega_A(S',T')$. By similar arguments, we obtain commutative diagrams for the remaining face morphisms as well, whence the claim follows.
\end{proof}

\begin{thm}\label{thm:gersthochcomp}
Let $A$ and $B$ be $K$-algebras which are projective when viewed as $K$-modules. Let $\mathsf C$ be an exact monoidal subcategory of $(\Mod(A^\ev), \otimes_A, A)$ which is entirely extension closed, and let $\mathfrak X \colon (\mathsf C, \otimes_A, A) \rightarrow (\Mod(B^\ev), \otimes_B, B)$ be an exact and almost strong $($or almost costrong$)$ monoidal functor such that one of the following conditions holds true for all objects $X \in \C$.
\begin{equation}\label{eq:exactnesscond}\tag{\textbf E}
\begin{aligned}
0 \longrightarrow L \longrightarrow &M \longrightarrow N \longrightarrow 0 \ \text{admissible exact in $\C$}\\
\Longrightarrow \quad &  0 \longrightarrow \mathfrak X(L) \otimes_B \mathfrak X(X) \longrightarrow \mathfrak X(M) \otimes_B \mathfrak X(X) \longrightarrow \mathfrak X(N) \otimes_B \mathfrak X(X) \longrightarrow 0\\
&\text{exact in $\Mod(B^\ev)$.}
\end{aligned}
\end{equation}
\begin{equation}\label{eq:exactnesscondprime}\tag{\textbf E'}
\begin{aligned}
0 \longrightarrow L \longrightarrow &M \longrightarrow N \longrightarrow 0 \ \text{admissible exact in $\C$}\\
\Longrightarrow \quad &  0 \longrightarrow \mathfrak X(X) \otimes_B \mathfrak X(L) \longrightarrow \mathfrak X(X) \otimes_B \mathfrak X(M) \longrightarrow \mathfrak X(X) \otimes_B \mathfrak X(N) \longrightarrow 0\\
&\text{exact in $\Mod(B^\ev)$.}
\end{aligned}
\end{equation}
Then the inclusion functor induces an isomophism $\Ext^\ast_{\C}(A,A) \xrightarrow{\sim } \HH^\ast(A)$ and the map
$$
\mathfrak X_\ast \colon \HH^\ast(A) \xrightarrow{\, \sim \, } \Ext^\ast_\C(A,A) \xrightarrow{\, \ \, } \Ext^\ast_{B^\ev}(B,B) \xrightarrow{\, \sim \, } \HH^\ast(B)
$$
that sends the equivalence class of an $n$-extension $0 \rightarrow A \rightarrow E_{n-1} \rightarrow \cdots \rightarrow E_0 \rightarrow A \rightarrow 0$ to the equivalence class of
$$0 \longrightarrow B \cong \mathfrak X(A) \longrightarrow \mathfrak X (E_{n-1}) \longrightarrow \cdots \longrightarrow \mathfrak X (E_0) \longrightarrow \mathfrak X(A) \cong B \longrightarrow 0$$
renders the following diagrams commutative for all $m, n \geqslant 0$.
$$
\xymatrix@C=35pt{
\HH^m(A) \times \HH^n(A) \ar[r]^-{\{-,-\}_A} \ar[d]_-{\mathfrak X_m \times \mathfrak X_n} & \HH^{m+n-1}(A) \ar[d]^-{\mathfrak X_{m + n -1}} \\
\HH^{m}(B) \times \HH^n(B) \ar[r]^-{\{-,-\}_B} & \HH^{m+n-1}(B)
}
\ \quad\quad \
\xymatrix@C=35pt{
\HH^{2n}(A) \ar[r]^{sq_A} \ar[d]_-{\mathfrak X_{2n}} & \HH^{4n-1}(A) \ar[d]^-{\mathfrak X_{4n-1}}\\
\HH^{2n}(B) \ar[r]^{sq_B} & \HH^{4n-1}(B)
}
$$
Here $\{-,-\}_A$ and $sq_A$ denote the Gerstenhaber bracket and the squaring map on $\HH^\ast(A)$, whereas $\{-,-\}_B$ and $sq_B$ denote the Gerstehaber bracket and the squaring map on $\HH^\ast(B)$.
\end{thm}

\begin{proof}
The theorem summarises some results relating the exact sequence interpretation of the Gerstenhaber bracket in Hochschild cohomology that we mentioned above, and its generalisation to exact monoidal categories, see \cite{He14b}, which incorporates the statement that, under appropriate assumptions, the (generalised) bracket commutes with almost strong (or almost costrong) exact monoidal functors.

To be more precise, the theorem is obtained by puzzling together \cite[Thm.\,3.1]{Sch98}, \cite[Lem.\,5.3.3]{He14b}, \cite[Thm.\,5.2.12]{He14b} and \cite[Thm.\,5.1]{He16a}; note that the exactness conditions (\ref{eq:exactnesscond}) and (\ref{eq:exactnesscondprime}) ensure that $\mathfrak X(S) \boxtimes_A \mathfrak X(T)$ is exact in $\Mod(B^\ev)$ for any choice of admissible extensions $S \in \mathcal Ext^m_{\C}(A,A)$ and $T \in \mathcal Ext^n_\C(A,A)$ by (the proof of) the statement \cite[Lem.\,3.3.4]{He14b}. Thus the homotopy class of the corresponding loop $\Omega_B(\mathfrak X(S),\mathfrak X(T))$ coincides (up to conjugation) with the homotopy class of the loop for the corresponding admissible exact sequences $S'$ and $T'$, with morphisms $S' \rightarrow \mathfrak X(S)$, $T' \rightarrow \mathfrak X(T)$, taken in, for instance, $\P(B)$, by Lemma \ref{lem:loopcoinc}.
\end{proof}

\begin{rem}
Let $\mathfrak X \colon (\Mod(A^\ev), \otimes_A, A) \rightarrow (\Mod(B^\ev), \otimes_B, B)$ be an exact and almost strong/costrong monoidal functor. The exactness condition on $\mathfrak X$ stated in the above theorem is automatically satisfied, if $\mathfrak X$, for instance, restricts as
$$
\mathfrak X \colon (\P(A), \otimes_A, A) \longrightarrow (\P(B), \otimes_B, B)\,.
$$
Similar situations will appear frequently hereinafter. 
\end{rem}


\section{Recollements of module categories}\label{sec:recoll}

\begin{nn}
Let $\A$, $\B$ and $\C$ be abelian categories. Recall that a diagram
\begin{equation}\tag{$\diamond$}\label{eq:recollement}
\begin{aligned}
\xymatrix{
\\\mathcal R(\A, \B, \C) \qquad \equiv \qquad\\
}
\xymatrix@!C=50pt{
\\ \A \ar[r]^-{i} & \B \ar[r]^-{j} \ar@/^2pc/[l]^-{i_\varrho} \ar@/_2pc/[l]_-{i_\lambda} & \C \ar@/^2pc/[l]^-{j_\varrho} \ar@/_2pc/[l]_-{j_\lambda}\\
}
\end{aligned}
\end{equation}
of additive functors is called a \textit{recollement} (\textit{of abelian categories}) if $(i_\lambda, i , i_\varrho)$ and $(j_\lambda, j , j_\varrho)$ are adjoint triples, the functors $i$, $j_\lambda$, $j_\varrho$ are fully faithful, and $\Im(i) = \Ker(j)$. The functor $j : \B \rightarrow \C$ is referred to as the \textit{quotient functor} of $\mathcal R(\A, \B, \C)$. We call a recollement (\ref{eq:recollement}) \textit{$K$-linear} if all the occurring categories and functors are $K$-linear. It is well-known, see \cite{FrPi04}, that in any recollement situation
\begin{enumerate}[\rm(1)]
\item the functors $i\colon \A \rightarrow \B$ and $j\colon \B \rightarrow \C$ are exact;
\item the functor $j\colon \B \rightarrow \C$ is essentially surjective;
\item the units of the adjoint pairs $(i, i_\varrho)$, $(j_\lambda, j)$ and the counits of the adjoint pairs $(i_\lambda, i)$ and $(j,j_\varrho)$ are isomorphisms;
\item for each object $B \in \B$, there are objects $A, A' \in \A$ and exact sequences
$$
0 \longrightarrow i(A) \longrightarrow (j_\lambda \circ j)(B) \longrightarrow B \longrightarrow (i \circ i_\lambda)(B) \longrightarrow 0 \, ,
$$
$$
0 \longrightarrow (i \circ i_\varrho)(B) \longrightarrow B \longrightarrow (j_\varrho \circ j)(B) \longrightarrow i(A') \longrightarrow 0 \,\,
$$
induced by the units and counits of the adjunctions.
\end{enumerate}
In particular, for all $X \in \Ob \A$, the functor $j\colon \B \rightarrow \C$ gives rise to a homomorphism
$$
j_\ast^{X} \colon \Ext^\ast_\B(X,X) \longrightarrow \Ext^\ast_\C(j(X),j(X)), \ j_\ast(\alpha) = [j(S)] \quad \text{(for $\alpha = [S] \in \Ext^n_\B(X,X)$)}
$$
of graded rings (for $j$ preserves pullbacks and pushouts). We make the following definition.
\end{nn}

\begin{defn}
Let $\B$ and $\C$ be abelian categories and let $f_\ast \colon \Ext^\ast_\B(B,B) \rightarrow \Ext^\ast_\C(C,C)$ be a homomorphism of graded rings for some objects $B \in \Ob \B$ and $C \in \Ob\C$. 

\begin{enumerate}[\rm(1)] 
\item Let $\mathfrak X \colon \B \rightarrow \C$ be an exact functor. The map $f_\ast$ is \textit{induced by the functor} $\mathfrak X : \B \rightarrow \C$ if there is an isomorphism $C \cong \mathfrak X(B)$ in $\C$ such that the diagram 
$$
\xymatrix@!C=20pt@!R=20pt{
& \Ext^\ast_\B(B,B) \ar[dr]^-{\mathfrak X^B_\ast} \ar[dl]_-{f_\ast} &\\
\Ext^\ast_\C(C,C) \ar[rr]^-\cong & & \Ext^\ast_\C(j(B),j(B))
}
$$
commutes.
\item Let $\mathcal R(\A,\B,\C)$ be a recollement of abelian categories. The map $f_\ast$ is \textit{induced by the recollement} $\mathcal R(\A,\B,\C)$ if $f_\ast$ is induced by the quotient functor $j \colon \B \rightarrow \C$ inside $\mathcal R(\A,\B,\C)$.
\end{enumerate}
\end{defn}

Note that for any given recollement $\mathcal R(\A,\B,\C)$ of abelian categories, and any isomorphism $\alpha \colon C \xrightarrow{\sim} j(B)$ in $\C$, the graded ring homomorphism
$$
j_\ast = j_\ast^{B,\alpha} \colon \Ext^\ast_\B(B,B) \xrightarrow{\ j_\ast^B \ } \Ext^\ast_\C(j(B),j(B)) \xrightarrow{ \ \sim \ } \Ext^\ast_\C(C,C)
$$
is (by definition) induced by the recollement $\mathcal R(\A,\B,\C)$.

\begin{nn}\label{nn:morrec}
Let $\mathcal R = \mathcal R(\A,\B,\C)$ and $\mathcal R' = \mathcal R(\A',\B',\C')$ be recollements of abelian categories. A triple $(F,G,H)$ of functors $F\colon \A \rightarrow \A'$, $G\colon \B \rightarrow \B'$ and $H\colon \C \rightarrow \C'$ is a \textit{morphism of recollements}, if the following functors agree up to natural isomorphisms.
\begin{multicols}{2}
\begin{enumerate}[\rm(1)]
\item\label{nn:morrec:1} $i' \circ F$ and $G \circ i$;
\item\label{nn:morrec:2} $F \circ i'_\lambda$ and $i_\lambda \circ G$;
\item\label{nn:morrec:3} $F \circ i'_\varrho$ and $i_\varrho \circ G$;
\item\label{nn:morrec:4} $j' \circ G$ and $H \circ j$;
\item\label{nn:morrec:5} $G \circ j'_\lambda$ and $j_\lambda \circ H$;
\item\label{nn:morrec:6} $G \circ j'_\varrho$ and $j_\varrho \circ G$.
\end{enumerate}
\end{multicols}
A morphism $(F,G,H)\colon \mathcal R \rightarrow \mathcal R'$ is an \textit{equivalence}, if $F$, $G$ and $H$ are equivalences of categories. Note that in this case, $j' \circ G$ and $H \circ j$ being naturally isomorphic already implies the same for the composites with the respective adjoint functors of $i$ and $i'$, and $j$ and $j'$.

Every pair $(G,H)$ of functors satisfying \ref{nn:morrec}(\ref{nn:morrec:4})--(\ref{nn:morrec:6}) gives rise to a (up to natural isomorphism) unique functor $F\colon \A \rightarrow \A'$ such that $i' \circ F$ and $G \circ i$ are naturally isomorphic. The functor $F$ is an equivalence if $G$ and $H$ were and in that case, it also is compatible with the adjoint functors of $i'$ (in the sense of \ref{nn:morrec}(\ref{nn:morrec:1})--(\ref{nn:morrec:3})). Thus naming an equivalence $(F,G,H) \colon \mathcal R \rightarrow \mathcal R'$ of recollements is the same as naming two equivalences $G\colon \B \rightarrow \B'$ and $H\colon \C \rightarrow \C'$ of categories such that $j' \circ G$ and $H \circ j$ are naturally isomorphic.
\end{nn}

\begin{defn}[see {\cite[Def.\,4.1]{PsVi14}}]\label{defn:recequ}
Let $\mathcal R = \mathcal R(\A,\B,\C)$ and $\mathcal R' = \mathcal R(\A',\B',\C')$ be recollements of abelian categories. The recollements $\mathcal R(\A,\B,\C)$ and $\mathcal R(\A',\B',\C')$ are \textit{equivalent} in case there are equivalences $G \colon \B \rightarrow \B'$ and $H\colon \C \rightarrow \C'$ of categories such that the functors $j' \circ G$ and $H \circ j$ are naturally isomorphic. In other words, $\mathcal R$ and $\mathcal R'$ are equivalent if there is an equivalence $(F,G,H) \colon \mathcal R \rightarrow \mathcal R'$ of recollements.
\end{defn}

\begin{nn}
A recollement $\mathcal R(\A,\B, \C)$ is a \textit{recollement of module categories}, if the categories $\A$, $\B$ and $\C$ are module categories over some rings $A$, $B$ and $C$. In that case, we write $\mathcal R(A,B,C)$ instead of $\mathcal R(\Mod (A), \Mod (B), \Mod(C))$. Recall from, for instance, \cite{AnFu92} that an equivalence $(G,H) \colon \mathcal R(A,B,C) \rightarrow \mathcal R(A',B',C')$ between recollements of module categories is determined by progenerators $P$ for $B$ and $Q$ for $C$ such that $B' \cong \End_B(P)^\op$ and $C' \cong \End_C(Q)^\op$, with the functors $G$ and $H$ thus being equivalent to $\Hom_B(P,-) \cong \Hom_B(P,B) \otimes_B -$ and $\Hom_C(Q,-) \cong \Hom_C(Q,C) \otimes_C -$ respectively. A choice of quasi-inverse functors is given by $P \otimes_{B'} -$ and $Q \otimes_{C'} -$.

Any ring $R$ along with any choice of an idempotent $e \in R$ gives rise to a recollement of module categories, which we will refer to as the \textit{standard recollement} associated to $(R,e)$.

\begin{exa}[Standard recollement]
Let $R$ be a ring and let $e \in R$ be idempotent. Let us put $\overline R = R/ReR$ and let $\pi \colon R \rightarrow \overline{R}$ be the natural projection. If we write $\pi_\star \colon \Mod(\overline R) \rightarrow \Mod(R)$ for the restriction along $\pi$, then the diagram
\begin{equation*}
\begin{aligned}
\xymatrix{\\\mathcal R(\overline R, R, eRe) \qquad \equiv \qquad\\}
\xymatrix@!C=50pt{\\ \Mod(\overline{R}) \ar[r]^-{\pi_\star} & \Mod(R) \ar[r]^-{e(-)} \ar@/^2pc/[l]^-{\Hom_{R}(\overline R,-)} \ar@/_2pc/[l]_-{\overline R \otimes_{R} -} & \Mod(eRe) \ar@/^2pc/[l]^-{\Hom_{eRe}(eR,-)} \ar@/_2pc/[l]_-{Re \otimes_{eRe} -}\\}
\end{aligned}
\end{equation*}
of functors defines a recollement which we will denote by $\mathcal R (R,e)$. This recollement will be $K$-linear, in case $R$ is an algebra over $K$.
\end{exa}
\end{nn}

\begin{thm}[see {\cite[Thm.\,5.3]{PsVi14}}]\label{thm:recmodeq}
Let $\mathcal R(A,B,C)$ be a recollement of module categories. Then there is a ring $R$ and an idepotent $e \in R$ such that $\mathcal R(A,B,C)$ is equivalent to $\mathcal R(R,e)$ in the sense of Definition $\ref{defn:recequ}$.
\end{thm}

Let $\mathcal R(A,B,C)$ be as above. The theorem yields that if $B$ is a $K$-algebra, being $K$-projective, then so is $C$. This follows from the following two observations for a $K$-projective $K$-algebra $\Gamma$.

\begin{enumerate}[\rm(1)]
\item Let $P$ be a finitely generated projective module over $\Gamma$. Then $\End_\Gamma(P)$ is $K$-projective. Indeed, if $P \oplus P' \cong \Gamma^n$ for some $n \geqslant 0$, then the $K$-module $\End_\Gamma(P)$ is a direct summand of $\Gamma^{n^2}$.
\item Let $\varepsilon \in \Gamma$ be an idempotent. The \textit{Peirce decomposition} of $\Gamma$ with respect to $\varepsilon$ yields
$$
\Gamma = \varepsilon\Gamma \varepsilon \oplus \varepsilon\Gamma (1-\varepsilon) \oplus (1-\varepsilon)\Gamma \varepsilon \oplus (1-\varepsilon)\Gamma (1-\varepsilon)\,,
$$
implying that $\varepsilon \Gamma \varepsilon$ is $K$-projective.
\end{enumerate}

Since the functors $\Mod(R/ReR) \rightarrow \Mod(R)$, $e(-) = eR \otimes_R (-) \colon \Mod(R) \rightarrow \Mod(eRe)$, as well as their respective left adjoints $(R/ReR) \otimes_R (-) \colon \Mod(R) \rightarrow \Mod(R/ReR)$, $Re \otimes_{eRe}(-) \colon \Mod(eRe) \rightarrow \Mod(R)$ are right exact and preserve arbitrary direct sums, the following consequence is imminent.

\begin{cor}\label{cor:recolldirectsums}
Let $\mathcal R(A,B,C)$ be a recollement of module categories. Then the functors $i \colon \Mod(A) \rightarrow \Mod(B)$ and $j\colon \Mod(B) \rightarrow \Mod(C)$, along with their left adjoints $i_\lambda$ and $j_\lambda$ are right exact and preserve arbitrary direct sums.\qed
\end{cor}

Let $R$ and $S$ be $K$-algebras. Watt's theorem (see \cite{Wa60}) tells that every right exact functor $\mathfrak X \colon \Mod(R) \rightarrow \Mod(S)$ which commutes with arbitrary direct sums is of the form $M \otimes_R -$ for some $S \otimes R^\op$-module $M$. In fact, $M \cong \mathfrak X(R)$ necessarily is implied. Note that $\mathfrak X$ gives rise to a $K$-algebra homomorphism 
$$
R^\op \cong \End_R(R) \longrightarrow \End_S(\mathfrak X(R)) \subseteq \End_K(\mathfrak X(R))\,,
$$
that is, $\mathfrak X(R)$ is $S$-$R$-bimodule. Theorem \ref{thm:recmodeq} yields Watt's theorem for functors within a recollement.

\begin{cor}
Let $\mathcal R(A,B,C)$ be a recollement of module categories. Then $i(-) \cong i(A) \otimes_A -$, $j(-) \cong j(B) \otimes_B -$ and $i_\lambda(-) \cong i_\lambda(B) \otimes_B -$, $j_\lambda(-) \cong j_\lambda(C) \otimes_C -$.
\end{cor}

\begin{proof}
Assume that $\mathcal R(A,B,C)$ is equivalent to $\mathcal R(R,e)$, and let $P$ and $Q$ be progenerators for $B$ and $C$ with $R \cong \End_B(P)^\op$ and $eRe \cong \End_C(Q)^\op$ giving rise to the defining equivalences $\Mod(B) \xrightarrow \sim \Mod(R)$ and $\Mod(C) \xrightarrow \sim \Mod(eRe)$. Then
\begin{align*}
j(-) & \cong Q \otimes_{eRe} eR \otimes_R \Hom_B(P,-)\\
& \cong Q \otimes_{eRe} eR \otimes_R \Hom_B(P, B) \otimes_B -\\
& \cong j(B) \otimes_B - \, .
\end{align*}
Similarly, $i(-) \cong i(A) \otimes_A -$ follows, as well as the statements for the left adjoints.
\end{proof}

\begin{nn}\label{nn:Rev}
Let $R$ and $S$ be $K$-algebras and $\mathfrak X \colon \Mod(R) \rightarrow \Mod(S)$ be an additive functor. By taking the dual of $\mathfrak X(R)$ with respect to $R$ we get the $R$-$S$-bimodule $\mathfrak X(R)^\vee = \Hom_{R^\op}(\mathfrak X(R),R)$. The functor
$$
\mathfrak X^\ev \colon \Mod(R^\ev) \longrightarrow \Mod(S^\ev), \ \mathfrak X^\ev(M) = \mathfrak X(R) \otimes_R M \otimes_R \mathfrak X(R)^\vee
$$
will be referred to as the \textit{enveloping functor} of $\mathfrak X$.
\end{nn}

\begin{prop}\label{prop:envrecoll}
Let $A$, $B$ and $C$ be $K$-algebras, and let $\mathcal R(A,B,C)$ be a recollement which is equivalent to a standard recollement $\mathcal R(R,e)$. Let $S = R \otimes R^\op$, $f$ be the idempotent $e \otimes e^\op$ in $S$ and $\overline S = S/SfS$. Then the functor $j^\ev$ fits inside a recollement $\mathcal R(\overline S, B^\ev, C^\ev)$ of module categories which is equivalent to the standard recollement $\mathcal R(S, f)$.
\end{prop}

\begin{proof}
Let $P$ be a progenerator for $B$ with $R \cong \End_{B}(P)^\op$ and $Q$ be a progenerator for $C$ with $eRe \cong \End_C(Q)^\op$. Let $G \colon \Mod(B^\ev) \rightarrow \Mod(R^\ev)$ and $H \colon \Mod(C^\ev) \rightarrow \Mod((eRe)^\ev)$ be the equivalences
$$
G(M) = \Hom_{B^\ev}(P \otimes \Hom_B(P,B), M), \quad H(N) = \Hom_{C^\ev}(Q \otimes \Hom_C(Q,C), N) \, .
$$
It suffices to show that the functor $j^\ev$ is equivalent to $H^{-1} \circ e(-)e \circ G$. Indeed, as
\begin{align*}
G(M) & \cong \Hom_{B^\ev}(P \otimes \Hom_B(P,B), B^\ev) \otimes_{B^\ev} M\\
& \cong \big(\Hom_B(P,B) \otimes \Hom_{B^\op}(\Hom_B(P,B), B)\big) \otimes_{B^\ev} M\\
& \cong \Hom_B(P,B) \otimes_B M \otimes_B P \, ,
\end{align*}
\begin{align*}
\Hom_{B}(Pe \otimes_{eRe} \Hom_C(Q,C),B) &\cong \Hom_{eRe}(\Hom_C(Q,C), \Hom_B(Pe,B))\\
& \cong Q \otimes_{eRe} \Hom_B(Pe,B)\\
& \cong Q \otimes_{eRe} \Hom_R(Re,\Hom_B(P,B))\\
& \cong Q \otimes_{eRe} e \Hom_B(P,B)\\
& \cong j(B)
\end{align*}
and $H^{-1}(N') \cong Q \otimes_{eRe} N' \otimes_{eRe} \Hom_C(Q,C)$, we may conclude
\begin{align*}
(H^{-1} \circ e (-)& e \circ G)(M) \\
& \cong Q \otimes_{eRe} e \Hom_B(P,B) \otimes_B M \otimes_B P e \otimes_{eRe} \Hom_C(Q,C)\\
& \cong j(B) \otimes_B M \otimes_B j(B)^\vee
\end{align*}
to finish the proof.
\end{proof}

\begin{nn}
Keep the notation from Proposition \ref{prop:envrecoll}. Note that in general, the category $\Mod(S/SfS)$ will be far away from being equivalent to the category of modules over the enveloping algebra of $\overline{R} = R / ReR$. There is, however, a canonical epimorphism
$$
\pi \colon  \overline S = S/SfS \longrightarrow \overline R \otimes \overline R^\op \cong \left(\frac{(1-e)R(1-e)}{(1-e)ReR(1-e)} \right)^\ev
$$
of algebras, where the righthand isomorphism is induced by
$$
\overline R = \frac{R}{ReR} \xrightarrow{\ \sim \ } \frac{(1-e)R(1-e)}{(1-e)ReR(1-e)}, \, (r + ReR) \mapsto (1-e)r(1-e) + (1-e)ReR(1-e)\,.
$$
The epimorphism $\pi$ fits inside the diagram
\begin{equation}\label{eq:shortpi}\tag{$\dagger$}
\begin{aligned}
\xymatrix@R=14pt@C=14pt{
& 0 \ar[d] & 0 \ar[d] & &\\
& ReR \otimes ReR \ar@{=}[r] \ar[d] & SfS \ar[d] & &\\
0 \ar[r] & (ReR \otimes R) + (R \otimes ReR) \ar[r] \ar[d] & S \ar[r]^-p \ar[d] & \displaystyle\left(\frac{(1-e)R(1-e)}{(1-e)ReR(1-e)} \right)^\ev \ar[r] \ar@{=}[d] & 0\\
0 \ar[r] & \displaystyle\frac{(ReR \otimes R) + (R \otimes ReR)}{ReR \otimes ReR} \ar[r] \ar[d] & \overline{S} \ar[r]^-{\pi} \ar[d] & \displaystyle\left(\frac{(1-e)R(1-e)}{(1-e)ReR(1-e)} \right)^\ev \ar[r] & 0\\
& 0 & 0 & &
}
\end{aligned}
\end{equation}
of $S$-modules -- which has exact rows provided that $R$ and $ReR$ are $K$-flat, and $\Tor_1^K(\overline{R}, \overline{R}) = 0$. Indeed, the commutative diagramm
$$
\xymatrix{
& 0 \ar[d] & 0 \ar[d] & 0 \ar[d] &\\
0 \ar[r] & ReR \otimes ReR \ar[dr]^-{i} \ar@{ >->}[d] \ar@{ >->}[r] & ReR \otimes R^\op \ar@{ >->}[d] \ar[r] & ReR \otimes \overline{R}^\op \ar[d] \ar[r] & 0\\
0 \ar[r] & R \otimes ReR \ar[d] \ar@{ >->}[r] & R \otimes R^\op \ar@/_3pc/[ddd]_-{\coker(i)} \ar[d] \ar[r] \ar[dr]^-p  & R \otimes \overline{R}^\op \ar[d] \ar[r] & 0\\
0 \ar[r] & \overline{R} \otimes ReR \ar@{ >->}[r]|(.575){ \ \ } \ar[d] & \overline{R} \otimes R^\op \ar[d] \ar[r] & \overline{R} \otimes \overline{R}^\op \ar[r] \ar[d] & 0\\
& 0 & 0 & 0 &\\
& & S/SfS \ar@{-->}[uur]_-\pi & &
}
$$
has exact rows and columns by the assumed $K$-flatness and $\Tor$-vanishing, whence the kernel of $p$ is $(ReR \otimes R) + (R \otimes ReR)$ by \cite[Sec.\,II.6, Exer.\,3]{Mac95}. Further, as $p \circ i = 0$, the universal property of the cokernel map induces the (necessarily surjective) dashed arrow as indicated. The following lemma can be read off from (\ref{eq:shortpi}) immediately.
\end{nn}

\begin{lem}\label{lem:PIisoiff}
Assume that $R$ and $ReR$ are $K$-flat, and that $\Tor_1^K(\overline{R}, \overline{R}) = 0$. Then the $K$-algebra
$$
\overline S = \frac{R \otimes R^\op}{(R \otimes R^\op)(e \otimes e^\op)(R \otimes R^\op)} = \frac{R \otimes R^\op}{(ReR) \otimes (R^\op e^\op R^\op)}
$$
is isomorphic to $\overline R^\ev = \overline R \otimes \overline R^\op$, through $\pi$, if, and only if, $(ReR \otimes R) + (R \otimes ReR) = SfS$. \qed
\end{lem}

\begin{nn}
The lemma tells us that in the case of finite dimensional algebras over a field $K$, $\overline R \otimes \overline R^\op$ being isomorphic to $\overline S$ is equivalent to $ReR \in \{0,R\}$, thus, that it is (almost) never the case.

However, the map $\pi$ is a split surjection of $\overline{S}$-modules in many cases. For instance, if $K$ is an algebraically closed field and $R = KQ$ is the path algebra of the quiver
$$
\xymatrix@R=0pt{
\bullet \ar[r] & \bullet \ar[r] & \cdots \ar[r] & \bullet\\
1 & 2 & & n
} \, ,
$$
i.e., of the Dynkin graph $A_n$ with linear orientation, then the quiver of $R \otimes R^\op $ is given by $Q^\ev = (A_n)^\ev$
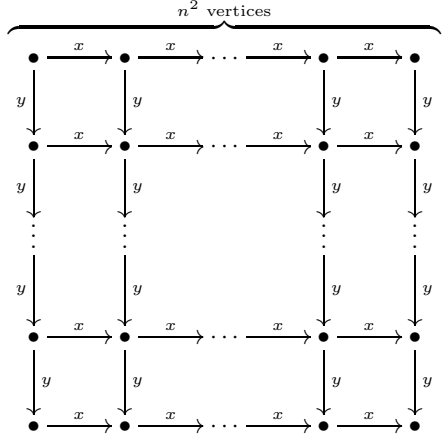
$$
\overbrace{
\xymatrix{
\bullet \ar[r]^-x \ar[d]_-y & \bullet \ar[r]^-x \ar[d]^-y & \cdots \ar[r]^-x & \bullet \ar[d]^-y \ar[r]^-x & \bullet \ar[d]^-y\\
\bullet \ar[r]^-x \ar[d]_-y & \bullet \ar[r]^-x \ar[d]^-y & \cdots \ar[r]^-x & \bullet \ar[d]^-y \ar[r]^-x & \bullet \ar[d]^-y \\
\tvdots \ar[d]_-y & \tvdots \ar[d]^-y & & \tvdots \ar[d]^-y & \tvdots \ar[d]^-y \\
\bullet \ar[r]^-x \ar[d]^-y & \bullet \ar[r]^-x \ar[d]^-y & \cdots \ar[r]^-x & \bullet \ar[r]^-x \ar[d]^-y & \bullet \ar[d]^-y\\
\bullet \ar[r]^-x & \bullet \ar[r]^-x & \cdots \ar[r]^-x & \bullet \ar[r]^-x & \bullet 
}}^\text{$n^2$ vertices}
$$
bounded by the commutativity relations $xy = yx$. Considering the primitive idempotent $e = \varepsilon_1$ that corresponds to the vertex $1$, the algebra $\overline S$ comes from the quiver 
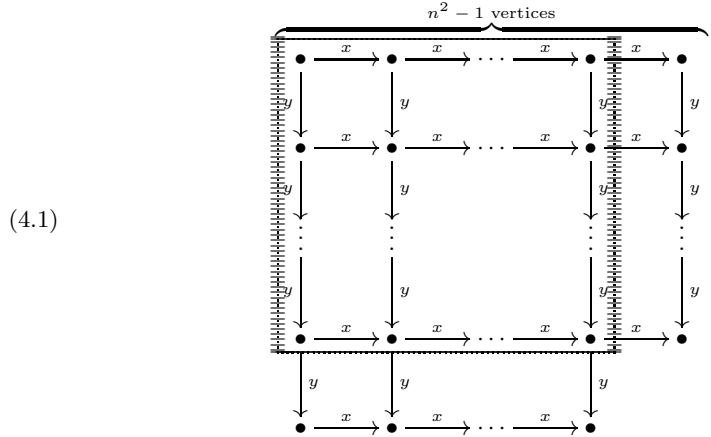
\begin{equation}\label{eq:quiverSbar}
\begin{aligned}
\overbrace{
\xymatrix{
\bullet \ar[r]^-x \ar[d]_-y & \bullet \ar[r]^-x \ar[d]^-y & \cdots \ar[r]^-x & \bullet \ar[d]^-y \ar[r]^-x & \bullet \ar[d]^-y\\
\bullet \ar[r]^-x \ar[d]_-y & \bullet \ar[r]^-x \ar[d]^-y & \cdots \ar[r]^-x & \bullet \ar[d]^-y \ar[r]^-x & \bullet \ar[d]^-y \\
\tvdots \ar[d]_-y & \tvdots \ar[d]^-y & & \tvdots \ar[d]^-y & \tvdots \ar[d]^-y \\
\bullet \ar[r]^-x \ar[d]^-y & \bullet \ar[r]^-x \ar[d]^-y & \cdots \ar[r]^-x & \bullet \ar[r]^-x \ar[d]^-y & \bullet\\
\bullet \ar[r]^-x & \bullet \ar[r]^-x & \cdots \ar[r]^-x & \bullet &
\save "4,1"+DL \PATH ~={**@{.}}
       '+<10.2pc,0pc> '+<0pc,9.85pc> '+<-10.5pc,0pc> '+<0pc,-9.85pc> '+<0.3pc,0pc>
\restore
}}^\text{$n^2 - 1$ vertices}
\end{aligned}
\end{equation}
with commutativity relations $xy = yx$, whereas $\overline R \otimes \overline R^\op$ is the path algebra of $(A_{n-1})^\ev$, bounded by the commutativity relations. It is thus apparent that $\pi$ splits, and a right inverse is given by identifying $\overline R \otimes \overline R^\op$ with the submodule of $\overline S$ corresponding to the highlighted piece of the quiver (\ref{eq:quiverSbar}).
\end{nn}

\begin{nn}
As $\Hom_{\overline{S}}(\overline R \otimes \overline R^\op, - )$ is right adjoint to the full and faithful restriction functor $\Mod(\overline R \otimes \overline R^\op) \rightarrow \Mod(\overline S)$ along $\pi$, and $\overline R \otimes \overline R^\op \cong \End_{\overline{S}}(\overline R \otimes \overline R^\op)^\op$, Morita theory yields the following statement.
\end{nn}

\begin{prop}
Assume that $R$ and $ReR$ are $K$-flat, and that $\Tor_1^K(\overline{R}, \overline{R}) = 0$. Then the restriction functor
$$
\pi_\star \colon \Mod(\overline R \otimes \overline R^\op) \longrightarrow \Mod(\overline S)
$$
is an equivalence of categories if, and only if, the $\overline S$-module epimorphism $\pi$ splits and there is a surjective $\overline S$-module homomorphism $(\overline R \otimes \overline R^\op)^n \rightarrow \overline S$ for some integer $n \geqslant 1$. \qed
\end{prop}

The main theorem of this section is the following.

\begin{thm}\label{thm:mainthm}
Let $\mathcal R = \mathcal R(A,B,C)$ be a $K$-linear recollement of module categories. Assume that $B$ is projective as a $K$-module and that 
$$
\Tor^C_1(j(B)^\vee, j(B)) = 0 \, .
$$
The image of $B$ under $j^\ev$ is canonically isomorphic to $C$ $($as a $C^\ev$-module$)$ and the graded map
$$
j^\ev_\ast \colon \HH^\ast(B) \xrightarrow{\ j^\ev \ } \Ext^\ast_{C^\ev}(j^\ev(B), j^\ev(B)) \xrightarrow{\ \sim \ } \HH^\ast(C)
$$
induced by $\mathcal R^\ev$ is a homomorphism of strict Gerstenhaber algebras.
\end{thm}

The proof of the theorem is spread over the next couple of statements. We begin by proving a fact that should be well-known, for which, however, we were not able to track down a reference.

\begin{lem}\label{lem:Tor1equ}
Let $\mathcal R(A,B,C)$ be a recollement of module categories being equivalent to a standard recollement $\mathcal R(R,e)$. Then
\begin{align*}
\Tor^C_1(j(B)^\vee,j(L)) = 0& \ \ \text{$($for all $L \in \P(B))$}\\ &\Longleftrightarrow \quad \Tor^C_1(j(B)^\vee, j(B)) = 0\\ 
&\Longleftrightarrow \quad \Tor^{eRe}_1(Re,eR) = 0 \\
&\Longleftrightarrow \quad \Tor^{eRe}_1(Me,eN) = 0 \ \ \text{$($for all $M,N \in \P(R))$}\, .
\end{align*}
\end{lem}

\begin{proof}
Indeed, the first equivalence is due to
$$
\Tor^C_1(j(B)^\vee,j(F)) \cong \bigoplus_{i \in I}\Tor^C_1(j(B)^\vee,j(B))
$$
for each free $B$-module $F = \bigoplus_{i \in I}{B}$ (see Corollary \ref{cor:recolldirectsums}). For the second equivalence, let $P$ and $Q$ be progenerators for $B$ and $C$ giving rise to the equivalence $\mathcal R(A,B,C) \xrightarrow \sim \mathcal R(R,e)$. We start by taking a projective resolution $\mathbb P(eR) \rightarrow eR \rightarrow 0$ of $eR$ over $eRe$ and assume that $\Tor^{eRe}_1(Re,eR) = 0$. The complex
$$
Re \otimes_{eRe} \mathbb P(eR) \longrightarrow Re \otimes_{eRe} eR \longrightarrow 0
$$
consists of projective $R$-modules in degrees $\geqslant 0$ and, moreover, has vanishing homology in degrees $0$ and $1$ (as its $i$-th homology computes as $\Tor^{eRe}_i(Re,eR)$). Thus we obtain an exact sequence
$$
0 \longrightarrow Re \otimes_{eRe} X \longrightarrow Re \otimes_{eRe} P_1 \longrightarrow Re \otimes_{eRe} P_0 \longrightarrow Re \otimes_{eRe} eR \longrightarrow 0 \,,
$$
where $P_0$ and $P_1$ are projective $eRe$-modules. The composition $j \circ (P \otimes_R -)$ of $j$ with the quasi-inverse functor $P \otimes_R -$ of $\Hom_B(P,-)$ takes that sequence to
\begin{equation}\label{eq:sequence}\tag{$\diamond$}
0 \longrightarrow j(Pe \otimes_{eRe} X) \longrightarrow j(Pe \otimes_{eRe} P_1) \longrightarrow j(Pe \otimes_{eRe} P_0) \longrightarrow j(Pe \otimes_{eRe} eR) \longrightarrow 0 \,.
\end{equation}
As $j \circ (P \otimes_R -)$ is naturally isomorphic to $(Q \otimes_{eRe} -) \circ e(-)$, the exact sequence above is (as a complex of $C$-modules) isomorphic to
$$
0 \longrightarrow Q \otimes_{eRe} X \longrightarrow Q \otimes_{eRe} P_1 \longrightarrow Q \otimes_{eRe} P_0 \longrightarrow Q \otimes_{eRe} eR \longrightarrow 0 \,.
$$
In particular, $j(Pe \otimes_{eRe} P_1)$ and $j(Pe \otimes_{eRe} P_0)$ are projective $C$-modules. Let us write $\overline{j}(-) = j(Pe \otimes_{eRe} (-))$ for abbreviation. We show that the exact sequence (\ref{eq:sequence}) remains exact after applying $j(B)^\vee \otimes_C -$, which will imply that $\Tor^C_1(j(B)^\vee, j(Pe \otimes_{eRe} eR)) = 0$. Indeed, as already mentioned above, we have natural $C$-module isomorphisms
$$
\overline{j}(E) = j(Pe \otimes_{eRe} E) \cong Q \otimes_{eRe} E \quad \text{(for all modules $E \in \Mod(eRe)$)}
$$
and the $B$-$C$-bimodule isomorphism
\begin{align*}
j(B)^\vee &\cong \Hom_{B^\op}(Q \otimes_{eRe} eR \otimes_R \Hom_B(P,B), B)\\
& \cong \Hom_{(eRe)^\op}(Q, \Hom_{R^\op}(eR, \Hom_{B^\op}(\Hom_B(P,B),B)))\\
& \cong \Hom_{(eRe)^\op}(Q, Pe)\\
&\cong Pe \otimes_{eRe} \Hom_{(eRe)^\op}(Q,eRe) \,.
\end{align*}
The map $\Hom_{(eRe)^\op}(Q,eRe) \otimes_C Q \xrightarrow{\sim} \Hom_{C}(Q,C) \otimes_C Q \xrightarrow{\ \, } eRe, \, \varphi \otimes q \mapsto \varphi(q)$, see for instance \cite[\S 22]{AnFu92}, is an $eRe$-bimodule isomorphism which gives rise to an isomorphism of complexes,
$$
\xymatrix@C=11.5pt{
0 \ar[r] & j(B)^\vee \otimes_C \overline{j}(X) \ar[r] \ar[d] & j(B)^\vee \otimes_C \overline{j}(eP_1) \ar[r] \ar[d] & j(B)^\vee \otimes_C \overline{j}(eP_0) \ar[r] \ar[d] & j(B)^\vee \otimes_C \overline{j}(eR) \ar[r] \ar[d] & 0\, \,\\
0 \ar[r] & Pe \otimes_{eRe} X \ar[r] & Pe \otimes_{eRe} eP_1 \ar[r] & Pe \otimes_{eRe} eP_0 \ar[r] & Pe \otimes_{eRe} eR \ar[r] & 0\, ,
}
$$
whence the top complex is acyclic as claimed. It thus follows that, for some $n \geqslant 1$,
\begin{align*}
\Tor^C_1(j(B)^\vee,j(B)) &\cong \Tor^C_1(j(B)^\vee, Q \otimes_{eRe} eR \otimes_R \Hom_B(P,B))\\ 
&\subseteq \Tor^C_1(j(B)^\vee, Q \otimes_{eRe} eR \otimes_R \Hom_B(P,P^n))\\
& \cong \bigoplus_{i = 1}^n \Tor^C_1(j(B)^\vee, Q \otimes_{eRe} eR)\\
& \cong \bigoplus_{i = 1}^n \Tor^C_1(j(B)^\vee, j(Pe \otimes_{eRe} eR)) = 0
\end{align*}
whence one implication of the second equivalence follows. Similar arguments apply to deduce the implication
$$
\Tor^C_1(j(B)^\vee, j(B)) = 0 \quad \Longrightarrow \quad \Tor^{eRe}_1(Re,eR) = 0 \,.
$$
Finally, the last equivalence follows from
$$
\Tor^{eRe}_1(F_1 e, eF_2) \cong \bigoplus_{i_1 \in I_1}\bigoplus_{i_2 \in I_2}\Tor^{eRe}_1(Re,eR)
$$
for all free right $R$-modules $F_1 = \bigoplus_{i_1 \in I_1}{R}$ and all free left $R$-modules $F_2 = \bigoplus_{i_2 \in I_2}{R}$.
\end{proof}

\begin{cor}
Let $\mathcal R(A,B,C)$ and $\mathcal R(A',B',C')$ be equivalent recollements of module categories. Then
$$
\Tor^C_1(j(B)^\vee, j(B)) = 0 \quad \Longleftrightarrow \quad \Tor^{C'}_1(j'(B')^\vee, j'(B')) = 0 \,.
$$\qed
\end{cor}

For the next proposition only, we let $B$ be a $K$-algebra, $e \in B$ be an idempotent, $C = eBe$ and $\mathfrak A_e$ the functor $\Mod(B^\ev) \rightarrow \Mod(C^\ev)$, $\mathfrak A_e(M) = eMe$.

\begin{prop}\label{prop:funcproper}
The following statements hold true for the functor $$\mathfrak A_e = e(-)e \colon \Mod(B^\ev) \rightarrow \Mod(C^\ev).$$
\begin{enumerate}[\rm(1)]
\item\label{prop:funcproper:1} $\mathfrak A_e$ is exact.
\item\label{prop:funcproper:2} $\mathfrak A_e$ is an almost costrong monoidal functor. It will be a strong monoidal functor, if the restricted multiplication map $\mu_e \colon Be \otimes_{eBe} eB \rightarrow B$ is an isomorphism.
\item\label{prop:funcproper:3} The functor
$$
Be \otimes_C (-) \otimes_C eB \colon \Mod(C^\ev) \longrightarrow \Mod(B^\ev)
$$
is a right invers of $\mathfrak A_e$, that is, $\mathfrak A_e (Be \otimes_C (-) \otimes_C eB) \cong \Id_{\Mod(C^\ev)}$. More precisely, $(Be \otimes_C (-) \otimes_C eB, \mathfrak A_e)$ forms an adjoint pair, whose unit is an isomorphism. For modules $M \in \Mod(C^\ev)$ and $N \in \Mod(B^\ev)$, unit and counit are given by
$$
\eta_M \colon M \longrightarrow eBe \otimes_C M \otimes_C eBe, \ m \mapsto e \otimes m \otimes e,
$$
$$
\varepsilon_N\colon Be \otimes_C e N e \otimes_C eB \longrightarrow N, \ be \otimes ene \otimes eb' \mapsto (be) n (eb').
$$
\item\label{prop:funcproper:4} If $eB$ is projective as a left $C=eBe$-module $($$Be$ is projective as a right $C=eBe$-module$)$, and if $M$ is a $B^\ev$-module which is projective as a left $B$-module $($as a right $B$-module$)$, then $\mathfrak A_e(M)$ is projective as a left $C$-module $($as a right $C$-module$)$.
\end{enumerate}
\end{prop}

\begin{proof}
As $e(-)e \cong eB \otimes_B (-) \otimes_B Be$, the statement (\ref{prop:funcproper:1}) is obvious. As for the second item, we have to name an isomorphism $t_0\colon C = eBe \rightarrow \mathfrak A_e(B)$ as well as natural homomorphisms
$$
t_{M,N} \colon \mathfrak A_e(M) \otimes_C \mathfrak A_e(N) \longrightarrow \mathfrak A_e(M \otimes_C N) \quad \text{(for all $M, N \in \Mod(B^\ev)$)}.
$$
To begin with, notice that $\mathfrak A_e(B) = eBe = C$. The multiplication map $\mu_e \colon Be \otimes_{eBe} eB \rightarrow B$ gives rise to the homomorphism $t'_{M,N} = eB \otimes_B M \otimes_B \mu_e \otimes_B N \otimes_B Be$,
$$
t'_{M,N} \colon eB \otimes_B M \otimes_B (Be \otimes_{eBe} eB) \otimes_B N \otimes_B Be \longrightarrow eB \otimes_B M \otimes_B B \otimes_B N \otimes_B Be \, ,
$$
which leads to a homomorphism $t_{M,N}$ between
\begin{align*}
\mathfrak A_e(M) \otimes_C \mathfrak A_e(N) 
&= eMe \otimes_{eBe} eNe\\ 
&= (eB \otimes_B M \otimes_B Be) \otimes_{eBe} (eB \otimes_B N \otimes_B Be)\\
&= eB \otimes_B M \otimes_B (Be \otimes_{eBe} eB) \otimes_B N \otimes_B Be
\intertext{and}
eB \otimes_B M \otimes_B B \otimes_B N \otimes_B Be
&\cong eB \otimes_B M \otimes_B N \otimes_B Be\\ 
&\cong eM \otimes_B Ne\\
&= \mathfrak A_e(M \otimes_B N) \, .
\end{align*}
Naturality of the $t_{M,N}$ is easily established. Item (\ref{prop:funcproper:3}) is either a straightforward verification, or a reminder of the fact that $\mathfrak A_e$ actually belongs to the standard recollement $\mathcal R(B^\ev, e \otimes e^\op)$. As for the last item, if $eB$ is $C$-projective, then so is $eM$, and thus
$$
\Hom_{C}(eMe, -) \cong \Hom_{C}(eM \otimes_B Be, -) \cong \Hom_B(Be, \Hom_{C}(eM,-))
$$
is exact as required.
\end{proof}

\begin{prop}\label{prop:moritamonoidal}
Let $R$ be a $K$-algebra, $P$ be a progenerator for $R$ and let $S = \End_R(P)^\op$. Let $\mathfrak A_P$ be the equivalence
$$
\Hom_{R^\ev}(P \otimes \Hom_R(P,R), -) \colon \Mod(R^\ev) \rightarrow \Mod(S^\ev)
$$
of categories. The following statements hold true.
\begin{enumerate}[\rm(1)]
\item $\mathfrak A_P$ is exact.
\item\label{prop:moritamonoidal:2} $\mathfrak A_P$ is a strong monoidal functor.
\item\label{prop:moritamonoidal:3} If $M$ is an $R^\ev$-module that is projective as a left $R$-module $($as a right $R$-module$)$, then $\mathfrak A_P(M)$ is projective as a left $S$-module $($as a right $S$-module$)$.
\end{enumerate}
\end{prop}

\begin{proof}
We will show that $\mathfrak A_P$ actually may be expressed as the composite of functors arising from certain idempotents, so that all statements will follow directly from Proposition \ref{prop:funcproper}. The matrix
$$
B = \left(\begin{matrix}
R & P\\
\Hom_R(P, R) & \End_R(P)^\op
\end{matrix}\right) = 
\left(\begin{matrix}
R & P\\
\Hom_R(P, R) & S
\end{matrix}\right)
$$
can be understood as an algebra through the following natural bimodule isomorphisms:
$$
P \otimes_S \Hom_R(P,R) \longrightarrow R, \ p \otimes \varphi \mapsto \varphi(p) \,,
$$
$$
\Hom_R(P, R) \otimes_R P \longrightarrow S, \ \varphi \otimes p \mapsto (p' \mapsto \varphi(p')p) \, .
$$
We have two canonical idempotens in $B$, namely
$$
e_R = \left(\begin{matrix}
1 & 0\\
0 & 0
\end{matrix}\right)
\quad \text{and} \quad
e_S = \left(\begin{matrix}
0 & 0\\
0 & 1
\end{matrix}\right) \, .
$$
The quotient algebras $\overline{B}_R = B/B e_R B$ and $\overline{B}_S = B/B e_S B$ both vanish, whence the functors
$$
e_\Sigma (-) e_\Sigma \colon \Mod(B^\ev) \longrightarrow \Mod((e_\Sigma B e_\Sigma)^\ev) = \Mod(\Sigma^\ev) \quad \text{(inside of $\mathcal R(B, e_\Sigma)$)}
$$
are equivalences for $\Sigma \in \{R, S\}$ by Lemma \ref{lem:PIisoiff}. Thus, the right adjoint of $e_R(-)e_R$, given by
$$
\Hom_{R^\ev}(e_R B \otimes B e_R, -) = \Hom_{R^\ev}(\left(\begin{smallmatrix}
R & P\\
0 & 0
\end{smallmatrix}\right) \otimes \left(\begin{smallmatrix}
R & 0\\
\Hom_R(P,R) & 0
\end{smallmatrix}\right), -)
\, ,
$$
is also an isomorphism and, as $e_R (-) e_R$ is strong monoidal, a strong monoidal functor. When composed with $e_S(-)e_S$, we obtain the functor
\begin{align*}
e_S\Hom_{R^\ev}(e_R B \otimes B e_R, -)e_S &\cong \Hom_{R^\ev}((e_R B e_S)\otimes (e_S B e_R), -)\\ &= \Hom_{R^\ev}(P \otimes \Hom_R(P,R), -)\\
&= \mathfrak A_P(-)
\end{align*}
which is a strong monoidal equivalence $(\Mod(R^\ev), \otimes_R, R) \xrightarrow{\sim} (\Mod(S^\ev),\otimes_S, S)$, as desired. It takes left projective modules to left projective modules by Proposition \ref{prop:funcproper}.
\end{proof}

We can now prove our main theorem.

\begin{proof}[Proof of Theorem $\ref{thm:mainthm}$]
The functor $j^\ev\colon \Mod(B^\ev) \rightarrow \Mod(C^\ev)$ restricts to an exact and almost costrong monoidal functor
$$
j^\ev \colon (\P(B), \otimes_B, B) \longrightarrow \Mod(C^\ev)
$$
by Propositions \ref{prop:funcproper} and \ref{prop:moritamonoidal}. In order to apply Theorem \ref{thm:gersthochcomp}, it suffices to show (since $\P(B)$ is closed under kernels of epimorphisms) that if $0 \rightarrow L \rightarrow M \rightarrow N \rightarrow 0$ is an admissible exact sequence in $\P(B)$, then
$$
0 \longrightarrow j^\ev(L) \otimes_C j^\ev(X) \longrightarrow j^\ev(M) \otimes_C j^\ev(X) \longrightarrow j^\ev(N) \otimes_C j^\ev(X) \longrightarrow 0
$$
is exact for all modules $X \in \P(B)$. In view of Proposition \ref{prop:moritamonoidal}(\ref{prop:moritamonoidal:2})+(\ref{prop:moritamonoidal:3}), this is equivalent to showing that if $0 \rightarrow L \rightarrow M \rightarrow N \rightarrow 0$ is an admissible exact sequence in $\P(R)$, then
$$
0 \longrightarrow eLe \otimes_{eRe} eXe \longrightarrow eMe \otimes_{eRe} eXe \longrightarrow eNe \otimes_{eRe} eXe \longrightarrow 0
$$
is exact for all modules $X \in \P(R)$. Indeed, one has the exact sequence
$$
\Tor^{eRe}_1(Ne,eX) \xrightarrow{\ a \ } Le \otimes_{eRe} eX \longrightarrow Me \otimes_{eRe} eX \longrightarrow Ne \otimes_{eRe} eX \longrightarrow 0
$$
wherein the map $a$ vanishes by Lemma \ref{lem:Tor1equ}; thus, applying $e(-)e$ yields the exact sequence
$$
0 \longrightarrow eLe \otimes_{eRe} eXe \longrightarrow eMe \otimes_{eRe} eXe \longrightarrow eNe \otimes_{eRe} eXe \longrightarrow 0
$$
as desired.
\end{proof}

In the following sections, we will illustrate the theorem's potential by a couple of applications. However, let us close this section by observing that a recollement of module categories $\mathcal R(A,B,C)$ gives rise to very nicely behaved exact and monoidal functors if we assume that that $j(B)$ or $\Hom_{B^\op}(j(B),B)$ are flat over $C$.

\begin{nn}\label{nn:exactsubcats}
Let us define a subcategory ${\P}_\lambda(B;e) \subseteq \F_\lambda(B) \subseteq \Mod(B^\ev)$ by letting $M \in \P(B)$ belong to ${\P}_\lambda(B;e)$ if, and only if, $\mathfrak A_e(M)$ belongs to $\F_\lambda(C)$. The subcategory ${\P}_\lambda(B;e)$ is non-empty, as $B$ belongs to it. Define further $\overline{\P}_\lambda(B;e) \subseteq {\P}_\lambda(B;e)$ to be the full subcategory that consists of all modules $M$ in ${\P}_\lambda(B;e)$ such that $M \otimes_B N \in {\P}_\lambda(B;e)$ for all $N \in \P_e(B)$; it is evident, that $B$ is a member of that category. By replacing $\F_\lambda(B)$ by $\F_\varrho(B)$ we obtain, in a similar way, a full subcategory $\overline{\P}_\varrho(B;e)$ of $\Mod(B^\ev)$. Note that ${\P}_\lambda(B;1) = \overline{\P}_\lambda(B;1) = \P_\lambda(B)$ and ${\P}_\varrho(B;1) = \overline{\P}_\varrho(B;1) = \P_\varrho(B)$. The following observation is immediate from Lemma \ref{lem:monoclosure}.
\end{nn}

\begin{lem}
The subcategories $\overline{\P}_\lambda(B;e)$ and $\overline{\P}_\varrho(B;e)$ of $\Mod(B^\ev)$ are closed under taking arbitrary direct sums and direct summands. Furthermore, they are closed under extensions and under taking tensor products over $B$. Thus, $(\overline{\P}_\lambda(B;e), \otimes_B, B)$ and $(\overline{\P}_\varrho(B;e), \otimes_B, B)$ are exact monoidal subcategories of $(\Mod(B^\ev), \otimes_B, B)$. \qed
\end{lem}

\begin{lem}\label{lem:projflatidemp}
Let $\mathcal R = \mathcal R(A,B,C)$ be a recollement of module categories, and let $\mathcal R(R,e)$ be a standard recollement that is equivalent to it. Then
\begin{enumerate}[\rm(1)]
\item $j(B)$ is $C$-flat if, and only if, $eR$ is $eRe$-flat,
\item $j(B)^\vee = \Hom_{B^\op}(j(B), B)$ is $C$-flat if, and only if, $Re$ is $eRe$-flat,
\item\label{lem:projflatidemp:3} $j(B)$ is $C$-projective if, and only if, $eR$ is $eRe$-projective, and
\item $j(B)^\vee$ is $C$-projective if, and only if, $Re$ is $eRe$-projective.
\end{enumerate}
\end{lem}

\begin{proof}
Let $P$ and $Q$ be progenerators for $B$ and $C$ defining the equivalence $\mathcal R(A,B,C) \xrightarrow \sim \mathcal R(R,e)$. We only take care of the first item in detail, as the remaining items may be proven similarly. As $\mathcal R$ and $\mathcal R(R,e)$ are equivalent, there are progenerators $P$ for $B$ and $Q$ for $C$ with $R \cong \End_{B}(P)^\op$ and $eRe \cong \End_C(Q)^\op$ such that there is a natural isomorphism
\begin{align*}
j(-) \xrightarrow{\, \sim \,} \left( Q \otimes_{eRe} (-) \right)& \circ e(-) \circ \Hom_B(P,-)\\ &= Q \otimes_{eRe} e \Hom_B(P,-) \, .
\end{align*}
It thus follows that 
\begin{equation}\label{eq:isolem1}
\begin{aligned}
j(B) &\cong Q \otimes_{eRe} e \Hom_B(P,B)\\ &\cong Q \otimes_{eRe} eR \otimes_R \Hom_B(P,B)\,.
\end{aligned}
\end{equation}
As the functors $Q \otimes_{eRe} (-)$ and $\Hom_B(P,-)$ are invertible, with quasi-invers functors $\Hom_C(Q,-)$ and $P \otimes_R (-)$, we also have
\begin{align*}
eR \otimes_R (-) & \cong \Hom_C(Q, j(P \otimes_R(-)) \\
& \cong \Hom_C(Q,C) \otimes_C j(P \otimes_R (-)) \,,
\end{align*}
as $Q$ is finitely generated projective over $C$, and hence 
\begin{equation}\label{eq:isolem2}
eR \cong \Hom_C(Q,C) \otimes_C j(P)
\end{equation}
Observe that $- \otimes_C j(B)$ is exact if, and only if, $- \otimes_C j(P)$ is exact. The modules $Q \in \Mod(C)$, $\Hom_B(P,B) \in \Mod(R)$, $P \in \Mod(B)$ and $\Hom_C(Q,C) \in \Mod(eRe)$ are projective; thus, by the equations (\ref{eq:isolem1}) and (\ref{eq:isolem2}), it follows that $- \otimes_C j(B)$ is exact if, and only if, $- \otimes_{eRe} eR$ is exact.
\end{proof}

\begin{nn}
Let $\mathcal R = \mathcal R(A,B,C)$ be a recollement of module categories, being equivalent to a standard recollement $\mathcal R(R,e)$. Substituting $\mathfrak A_e$ in Paragraph \ref{nn:exactsubcats} by $j^\ev$, we obtain full, exact monoidal subcategories $(\overline{\P}_\lambda(\mathcal R), \otimes_B, B)$ and $(\overline{\P}_\varrho(\mathcal R), \otimes_B, B)$ of $(\Mod(B^\ev), \otimes_B, B)$. Note that $B$ indeed belongs to either of these categories, as $j^\ev(B) \cong C$.
\end{nn}

\begin{prop}\label{prop:flatprojrecoll}
Let $\mathcal R = \mathcal R(A,B,C)$ be a recollement of module categories.
\begin{enumerate}[\rm(1)]
\item\label{prop:flatprojrecoll:1} The functor $j^\ev: \Mod(B^\ev) \rightarrow \Mod(C^\ev)$ restricts to an exact and monoidal functor
$$
(\overline{\P}_\lambda(\mathcal R), \otimes_B, B) \longrightarrow (\mathsf F_\lambda(C), \otimes_C, C)
$$
between exact monoidal subcategories. Moreover, the inclusion functor induces an isomorphism
$$
\Ext^\ast_{\overline{\P}_\lambda(\mathcal R)}(B,B) \xrightarrow{\,\sim\,} \Ext^\ast_{B^\ev}(B,B) 
$$
if $j(B)$ is $C$-projective.

\item The functor $j^\ev\colon \Mod(B)^\ev \rightarrow \Mod(C^\ev)$ restricts to an exact and monoidal functor
$$
(\overline{\P}_\varrho(\mathcal R), \otimes_B, B) \longrightarrow (\mathsf F_\varrho(C), \otimes_C, C)
$$
between exact monoidal subcategories. Moreover, the inclusion functor induces an isomorphism
$$
\Ext^\ast_{\overline{\P}_\varrho(\mathcal R)}(B,B) \xrightarrow{\,\sim\,} \Ext^\ast_{B^\ev}(B,B)
$$
if $j(B)^\vee$ is $C$-projective.
\end{enumerate}
Furthermore, if for $? \in \{\lambda, \varrho\}$ the map $\Ext^\ast_{\overline{\P}_?(\mathcal R)}(B,B) \rightarrow \Ext^\ast_{B^\ev}(B,B)$ is an isomorphism, then the induced map
$$
\Ext^\ast_{B^\ev}(B,B) \xrightarrow{\,\sim\,} \Ext^\ast_{\overline{\P}_?(\mathcal R)}(B,B) \xrightarrow{\, \ \,} \Ext^\ast_{\mathsf F_?(C)}(C,C) \xrightarrow{\,\ \,} \Ext^\ast_{C^\ev}(C,C)
$$
agrees with the map $j^\ev_\ast$ as mentioned in the statement of Theorem $\ref{thm:mainthm}$.
\end{prop}

\begin{proof}
We only prove item (\ref{prop:flatprojrecoll:1}). The restriction of $j^\ev$ to $(\overline{\P}_\lambda(\mathcal R), \otimes_B, B)$ is of course exact, and also monoidal by Propositions \ref{prop:funcproper} and \ref{prop:moritamonoidal}. It is by definition that $j^\ev$ maps $\overline{\P}_\lambda(\mathcal R)$ to $\F_\lambda(C)$. Now assume that $j(B)$ is $C$-projective. Replacing $\F_\lambda(eBe)$ by $\P_\lambda(C)$, and $\mathfrak A_e$ by $j^\ev$ in Paragraph \ref{nn:exactsubcats}, we obtain extension closed subcategories $\overline{\P}^1_\lambda(\mathcal R) \subseteq {\P}^1_\lambda(\mathcal R) \subseteq \Mod(B^\ev)$, the former being an exact monoidal subcategory of $(\Mod(B^\ev), \otimes_B, B)$. Further, $\overline{\P}^1_\lambda(\mathcal R)$ is closed under kernels of epimorphisms, as $\P_\lambda(C)$ is, and contains $\Proj(B^\ev)$. Indeed, the functor $j^\ev$ sends $B \otimes B$ to
$$
Q \otimes_{eRe} \left(e\Hom_{B^\ev}(P \otimes \Hom_B(P,B), B \otimes B) e\right) \otimes_{eRe} \Hom_C(Q,C) \,.
$$
The module $\Hom_{B^\ev}(P \otimes \Hom_B(P,B), B^\ev)$ is a progenerator for $R^\ev$ which consequently is a direct summand of $(R^\ev)^n$ for an integer $n \geqslant 1$. Further $eR \otimes Re$ is a projective left $eRe$-module by Lemma \ref{lem:projflatidemp}(\ref{lem:projflatidemp:3}), therefore so is
$$
Q \otimes_{eRe} (eR \otimes Re) \otimes_{eRe} \Hom_C(Q,C)
$$
over $C$; thus $j^\ev(B \otimes B)$ is a direct summand of a $C$-projective $C^\ev$-module, hence itself $C$-projective. It follows that $\overline{\P}^1_\lambda(\mathcal R)$ is entirely extension closed by Proposition \ref{prop:entireextclo}, whence the commutative diagram
$$
\xymatrix@C=4pt{
& \Ext^\ast_{\overline{\P}^1_\lambda(\mathcal R)}(B,B) \ar[dr]^-{\sim} \ar[dl]_-{\sim} & \\
\Ext^\ast_{\overline{\P}_\lambda(\mathcal R)}(B,B) \ar[rr] & & \Ext^\ast_{B^\ev}(B,B)
}
$$
induced by the inclusion functors $\overline{\P}^1_\lambda(\mathcal R) \hookrightarrow \overline{\P}_\lambda(\mathcal R)$ and $\overline{\P}^1_\lambda(\mathcal R) \hookrightarrow \Mod(B^\ev)$ yields the claim.
\end{proof}


\section{The long exact sequence of Buchweitz}\label{sec:buchweitz}

\begin{nn}
In this section, we investigate a long exact sequence that has been introduced by Buchweitz in \cite{Bu03}. By $B$ we will denote a $K$-algebra which is projective as a $K$-module. Further, we let $e \in B$ be an idempotent element such that
$$
\Tor^{eBe}_1(Be,eB) = 0
$$
and abbriviate $\overline B = B / BeB$, $C = eBe$. Of course, $C$ is a $K$-projective $K$-algebra as well, whereas $\overline B$ does not have to be. We like to point out that the above $\Tor$-vanishing condition will only be relevant for the proof of Theorem \ref{thm:chigersten} (as one easily believes in light of Theorem \ref{thm:mainthm}), but not for the construction of the map it concerns itself with.
\end{nn}

\begin{nn}\label{nn:buchweitz}
As usual, let $\mathbb B B$ denote the bar resolution of $B$, and let $\mathbb B C$ be the bar resolution for $C$. There is an injection
$$
\tilde{\mu}_e \colon Be \otimes_C \mathbb B C \otimes_C eB \longrightarrow \mathbb B B
$$
of complexes that lifts the multiplication map $\mu_e \colon Be \otimes_C eB \rightarrow B$. It is induced by the canonical isomorphisms
$$
\xymatrix@C=16pt{
Be \otimes_C C^{\otimes (n + 2)} \otimes_C eB \ar[r]^-{\sim} & Be \otimes (eBe)^{\otimes n} \otimes eB \subseteq B^{\otimes(n+2)}
}
$$
for $n \geqslant 0$, which are compatible with the occurring differentials and therefore identify $Be \otimes_C \mathbb B C \otimes_C eB$ with a subcomplex of $\mathbb B B$. We put
$$
\mathbb B(B/C) = \Coker(\tilde{\mu}_e \colon Be \otimes_C \mathbb B C \otimes_C eB \rightarrow \mathbb B B)\,,
$$
and hence obtain a short exact sequence
$$
0 \longrightarrow Be \otimes_C \mathbb B C \otimes_C eB \longrightarrow \mathbb B B \longrightarrow \mathbb B(B/C) \longrightarrow 0
$$
of complexes. It is semi-split, in that in each degree, the $B^\ev$-module homomorphisms are split. Indeed, $Be \otimes (eBe)^{\otimes n} \otimes eB$ is an evident direct summand of 
$$
B^{\otimes (n+2)} = (Be \oplus Be') \otimes (eBe \oplus eBe' \oplus e'Be \oplus e'Be')^{\otimes n} \otimes (eB \oplus e'B)\,,
$$
where $e' = 1 - e$. Since
$$
\Hom_{B^\ev}(Be \otimes_C C^{\otimes(n+2)} \otimes_C eB, M) \cong \Hom_{C^\ev}(C^{\otimes(n+2)}, eMe) \quad (\text{for $n \geqslant 0$})
$$
by adjunction, we obtain a long exact sequence in Hochschild cohomology:
$$
\cdots \longrightarrow \HH^i(B/C,M) \longrightarrow \HH^i(B,M) \longrightarrow \HH^i(C, eMe) \longrightarrow \HH^{i+1}(B/C,M) \longrightarrow \cdots \, ,
$$
where $\HH^\ast(B/C,M)$ denotes the cohomology of $\Hom_{B^\ev}(\mathbb B(B/C),M)$. If we specialise to $M = B$, we in particular get maps
$$
\chi_{B/C}^i = \HH^i(\mu_e,B) \colon \HH^{i}(B) \longrightarrow \HH^i(C) \quad \text{(for $i \geqslant 0$).}
$$
\end{nn}

\begin{thm}[see {\cite[Thm.\,4.7]{Bu03}}]\label{thm:chiringhom}
The graded map $\chi_{B/C} = \chi^\ast_{B/C} \colon \HH^\ast(B) \rightarrow \HH^\ast(C)$ is a homomorphism of graded $K$-algebras.
\end{thm}

The goal of the upcoming considerations is to show that the algebra homomorphism $\chi^\ast_{B/C}$ also preserves the Gerstenhaber bracket and the squaring map.

\begin{prop}\label{prop:Ieischi}
Let $\mathfrak A_e$ be the functor $e(-)e\colon \Mod(B^\ev) \rightarrow \Mod(C^\ev)$ as in Proposition $\ref{prop:funcproper}$. The induced map $$\mathfrak A_e^\ast\colon \HH^\ast(B) \longrightarrow \HH^\ast(C)$$ coincides with the map $\chi_{B/C}$ occurring in the long exact sequence of Paragraph $\ref{nn:buchweitz}$. In particular, $\chi_{B/C}$ is a homomorphism of graded $K$-algebras $($cf.\,Theorem $\ref{thm:chiringhom})$.
\end{prop}

\begin{proof}
We show a slightly stronger statement. Let $\mathsf B \subseteq \Mod(B^\ev)$ and $\C \subseteq \Mod(C^\ev)$ be entirely extension closed subcategories, respectively containing $B$ and $C$, as well as their enveloping algebras, such that $\mathfrak A_e \mathsf B \subseteq \C$. Then the induced homomorphism
$$
\mathfrak A_e^\ast \colon \Ext^\ast_{\B}(B,B) \longrightarrow \Ext^\ast_\C(C,C)
$$
coincides with $\chi_{B/C}$. To begin with, through the adjunction isomorphism of the adjoint pair $(Be \otimes_C (-) \otimes_C eB, \mathfrak A_e)$, the morphism $\tilde{\mu}_e$ yields an inclusion
\begin{align*}
\tilde{\mu}_e^\sharp \colon \mathbb B C \cong \Hom_{C^\ev}&(C \otimes C, \mathfrak A_e(Be \otimes_C \mathbb B C \otimes_C eB))\\
&\longrightarrow \Hom_{C^\ev}(C \otimes C, \mathfrak A_e\mathbb B B)\\
&\xrightarrow{\hspace*{1.7pt}\sim\hspace*{1.7pt}} \Hom_{B^\ev}(Be \otimes eB, \mathbb B B) \cong e (\mathbb B B) e
\end{align*}
of complexes, which can also be obtained from the fact that $C^{\otimes (n+2)}$ is a direct summand of $e (B^{\otimes(n+2)}) e = eB \otimes B^{\otimes n} \otimes Be$. Observe that $\tilde \mu^\natural_e$ agrees with the map
$$
\xymatrix@C=20pt{
\mathbb B C \ar[r]^-\sim & eBe \otimes_{eBe} \mathbb B C \otimes_{eBe} eBe \ar[r]^-{e(\tilde{\mu}_e) e} & e (\mathbb B B) e
},
$$
that is, the diagram
$$
\xymatrix{
\Hom_{B^\ev}(\mathbb B B, B) \ar[r]^-{- \circ \tilde{\mu}_e} \ar[d]_-{\mathfrak A_e} & \Hom_{B^\ev}(Be \otimes_C \mathbb B C \otimes_C eB, B) \ar[d]^-{\mathrm{adj}} \\
\Hom_{C^\ev}(e(\mathbb B B) e, C) \ar[r]^-{- \circ \tilde\mu_e^\sharp} & \Hom_{C^\ev}(\mathbb B C, C)
}
$$
commutes, where $\mathrm{adj}$ denotes the adjunction isomorphism from above, given by
\begin{align*}
& \ \ \Hom_{B^\ev}(Be \otimes_C \mathbb B C \otimes_C eB, B) \longrightarrow \Hom_{C^\ev}(\mathbb B C, C)\\
& \mathrm{adj}(\varphi)(c_0 \otimes \cdots \otimes c_{n+1}) = e\varphi(e \otimes c_o \otimes \cdots \otimes c_{n+1} \otimes e)e.
\end{align*}
Thus, the $i$-th component of the map $\chi_{B/C}$ is given by $\chi_{B/C}^i = H^i(- \circ \mu^\sharp_e) \circ H^i(\mathfrak A_e(-))$.

Due to the assumptions made, $\mathbb B B$ is a complex in $\B$, and similarly, $\mathbb BC$ is a complex in $\C$. Pick an $n$-extension
$$
\xymatrix@C=20pt{
S & \equiv & 0 \ar[r] & B \ar[r]^-{d_n} & E_{n-1} \ar[r]^-{d_{n-1}} & \cdots \ar[r]^-{d_1} & E_0 \ar[r]^-{d_0} & B \ar[r] & 0
}
$$
in $\B$ along with a comparison map
$$
\xymatrix@C=16pt{
& & \cdots \ar[r] & B^{\otimes(n+2)} \ar[d]_-{\varphi_{n}} \ar[r] & B^{\otimes(n+1)} \ar[d]_-{\varphi_{n-1}} \ar[r] & \cdots \ar[r] & B \otimes B \ar[r] \ar[d]^-{\varphi_0} & B \ar[r] \ar@{=}[d] & 0 \\
S & \equiv & 0 \ar[r] & B \ar[r] & E_{n-1} \ar[r] & \cdots \ar[r] & E_0 \ar[r] & B \ar[r] & 0
}
$$
The map $\mathfrak A_e^n$ sends $S$ to $eSe$,
$$
\xymatrix@C=18pt{
\mathfrak A_e^n S & \equiv & 0 \ar[r] & C \ar[r] & eE_{n-1}e \ar[r] & \cdots \ar[r] & eE_0e \ar[r] & C \ar[r] & 0\,,
}
$$
whereas $\chi_{B/C}^n$ maps it to the lower sequence in the pushout diagram below.
$$
\xymatrix@C=18pt{
\cdots \ar[r] & C^{\otimes(n+2)} \ar[d]_-{e\varphi_{n}e} \ar[r] & C^{\otimes(n+1)} \ar[d] \ar[r] & \cdots \ar[r] & C \otimes C \ar[r] \ar@{=}[d] & C \ar[r] \ar@{=}[d] & 0 \\
0 \ar[r] & C \ar[r] & C \oplus_{C^{\otimes(n+2)}} C^{\otimes(n+1)} \ar[r] & \cdots \ar[r] & C \otimes C \ar[r] & C \ar[r] & 0
}
$$
To make sense out of the map $e(\varphi_n)e\colon C^{\otimes(n+2)} \rightarrow C$, remember that $C^{\otimes(n+2)} = (eBe)^{\otimes(n+2)}$ is a direct summand of $e(B^{\otimes(n+2)})e$ to which any map from $e(B^{\otimes(n+2)})e$ may be restricted to; in fact, $(e \varphi_n e) \mathord{\upharpoonright}_{C^{\otimes(n+2)}}$ is simply $(e \varphi_n e) \circ \mu_e^\sharp$. However, the commutative diagram
$$
\xymatrix@C=16pt{
0 \ar[r] & C \ar[r] \ar@{=}[d] & C \oplus_{C^{\otimes(n+2)}} C^{\otimes(n+1)} \ar[r] \ar@{-->}[d] & C^{\otimes n} \ar[r] \ar[d]_-{e\varphi_{n-2}e} & \cdots \ar[r] & C \otimes C \ar[r] \ar[d]^-{e\varphi_{0}e} & C \ar[r] \ar@{=}[d] & 0 \, \,\\
0 \ar[r] & eBe \ar[r] & eE_{n-1}e \ar[r] & eE_{n-2}e \ar[r] & \cdots \ar[r] & eE_0e \ar[r] & eBe \ar[r] & 0 \, ,
}
$$
with the dashed arrow being induced by the map
$$
C \oplus C^{\otimes(n+1)} \longrightarrow eE_{n-1}e, \ (c, c_0 \otimes \cdots \otimes c_n) \mapsto e\left(d_n(c) + \varphi_{n-1}(c_0 \otimes \cdots \otimes c_n)\right)e\,,
$$
now shows that both sequences define the same equivalence class in $\Ext^n_{C^\ev}(C,C)$ (even in $\Ext_\C^\ast(C,C)$), whence the claim follows.
\end{proof}

The proposition yields the direct consequence below.

\begin{cor}
The map $\chi_{B/C} \colon \HH^\ast(B) \rightarrow \HH^\ast(C)$ is bijective if $\overline B = B / BeB = 0$.\qed
\end{cor}

\begin{thm}\label{thm:chigersten}
The graded map $\chi_{B/C} \colon \HH^\ast(B) \rightarrow \HH^\ast(C)$ is a homomorphism of strict Gerstenhaber algebras.
\end{thm}

\begin{proof}
By Proposition \ref{prop:Ieischi}, the map $\chi_{B/C}$ is given (on extensions) by applying the functor $e(-)e = eB \otimes_B (-) \otimes_B Be \colon \Mod(B^\ev) \rightarrow \Mod(C^\ev)$, that is, it is induced by the standard recollement $\mathcal R(B^\ev,e \otimes e^\op)$. The assertion now follows from Theorem \ref{thm:mainthm}.
\end{proof}

Recall the following homological invariant of a module (see \cite[Def.\,2.1]{Bu03}).

\begin{defn}
Let $X$ be a module over a ring $R$. The \textit{grade} of $X$ is given by
$$
\grade_R X = \inf\{i \geqslant 0 \mid \Ext^i_R(X,R) \neq 0\} \quad \in \quad \mathbb Z \cup \{\infty\}.
$$
\end{defn}

It is obvious that $\overline B = 0$ implies $\grade_B \overline B = \infty$. However, the converse, in general, does not need to hold true. Several conjectures though predict a strong mutual dependance of $\overline B = 0$ and $\grade_B \overline B = \infty$ for certain pairs $(B,e)$. See \cite[Sec.\,3]{Bu03} for a very clear exposition of those conjectures, and their interplay. 

The following result is \cite[Thm.\,5.5]{Bu03}. It tells us that the grade of $\overline{B}$, to some extend, controls whether $\chi_{B/C}$ is an isomorphism or not. Indeed, its proof utilises the long exact sequence we briefly introduced in Paragraph \ref{nn:buchweitz} and vanishing criteria for its somewhat mysterious members $\HH^i(B/C,B)$ (for $i \geqslant 0$) in terms of $\mathrm{grade}_B \overline B$ (see \cite[Sec.\,5]{Bu03}).

\begin{thm}\label{thm:grade}
Let $\grade_B \overline B = g$ and $E = \Ext^g_C(\overline{B},B)$. Then the maps
$$
\chi_{B/C}^j \colon \HH^{j}(B) \longrightarrow \HH^j(C)
$$
are isomorphisms for $j \leqslant g-2$, and there is an exact sequence
$$
0 \longrightarrow \HH^{g-1}(B) \longrightarrow \HH^{g-1}(C) \longrightarrow \HH^0(B,E) \longrightarrow \HH^g(B) \longrightarrow \HH^g(C).
$$\qed
\end{thm}

Let $e' = 1 - e$ the complementary idempotent of $e$ in $B$, and $C' = e'Be'$. As a consequence of the Theorems \ref{thm:chigersten} and \ref{thm:grade}, we obtain the following result.

\begin{thm}\label{thm:isogerstengrade}
The following hold true.
\begin{enumerate}[\rm(1)]
\item If $\grade_B \overline B \geqslant 3$, then $\chi_{B/C}^1$ defines an isomorphism
$$
\frac{\Der_K(B)}{\Inn_K(B)} = \HH^1(B) \longrightarrow \HH^1(C) = \frac{\Der_K(C)}{\Inn_K(C)}
$$
of Lie algebras. Thus, if $\Tor^{C'}_1(Be',e'B) = 0$, the map $\chi^1_{B/C'} \circ (\chi^1_{B/C})^{-1}$ defines a $($in general non-trivial$)$ Lie algebra homomorphism $\HH^1(C) \rightarrow \HH^1(C')$.
\item If $\grade_B \overline B = \infty$, then $\chi^\ast_{B/C}$ defines an isomorphism $\HH^\ast(B) \rightarrow \HH^\ast(C)$ of strict Gerstenhaber algebras. Thus, there is a homomorphism $\HH^\ast(C) \rightarrow \HH^\ast(C')$ of strict Gerstenhaber algebras.\qed
\end{enumerate}
\end{thm}


\section{The long exact sequence of Cibils, Green--Solberg and Michelena--Platzeck}\label{sec:greensol}

\begin{nn}
In this section, we investigate long exact sequences independently studied and introduced by Cibils, Green--Solberg and Michelena--Platzeck in \cite{Ci00}, \cite{GrSo02} and \cite[Sec.\,2]{MiPl00}; see also \cite{BeGu04}, \cite{GeSch85} and \cite{GoGo01}. The results we will present here are fairly similar in taste to the ones in Section \ref{sec:buchweitz}, however with different outcome, and different applications. Let us fix a $K$-algebra $B$ which we assume to be projective over $K$. Let $e \in B$ be an idempotent, and $e' = 1 - e$ be its complementary idempotent. As usual, we abbreviate $C = eBe$ and $C' = e'Be'$.
\end{nn}

\begin{lem}\label{lem:isoidealtor}
Assume that $\Tor^{eBe}_i(Be,eB) = 0$ for $i > 0$. Then the functor
$$
Be \otimes_{eBe} (-) \otimes_{eBe} eB \colon \Mod((eBe)^\ev) \longrightarrow \Mod(B^\ev)
$$
gives rise to an isomorphism 
$$
\Ext^\ast_{B^\ev}(Be \otimes_{eBe} eB, B) \xrightarrow{\ \sim \ } \Ext^\ast_{(eBe)^\ev}(eBe,eBe) = \HH^\ast(eBe)
$$
of graded $K$-modules.
\end{lem}

\begin{proof}
Indeed, if $\mathbb PC \rightarrow C \rightarrow 0$ is a projective resolution of $C$ over $C^\ev$, then $Be \otimes_C \mathbb PC \otimes_C eB$ is a complex of projectives, since $e(-)e$ is exact and right adjoint to $Be \otimes_C (-) \otimes_C eB$. Moreover,
$$
H_i (\mathbb PC \otimes_C eB) = \Tor_i^C(C, eB) = 0 \quad \text{(for all $i \neq 0$)}
$$
for $\mathbb PC \rightarrow C \rightarrow 0$ in particular is a resolution of $C$ by flat right $C$-modules ($C \otimes C$ is $C$-flat on both sides since $B$, and therefore $C$ are $K$-flat). Furthermore,
$$
(-) \otimes_C ((C \otimes C) \otimes_C eB) \cong (- \otimes_C C) \otimes (C \otimes_C eB) \cong (-) \otimes eB
$$
as functors $\Mod(C^\op) \rightarrow \Mod(K)$, whence $(C \otimes C) \otimes_C eB$ is a flat left $C$-module. It follows that $\mathbb PC \otimes_C eB$ resolves $eB$ by flat left $C$-modules. From
$$
H_i (Be \otimes_C \mathbb PC \otimes_C eB) = \Tor_i^C(Be, eB) = 0 \quad \text{(for all $i \neq 0$)}
$$
we deduce that $Be \otimes_C \mathbb PC \otimes_C eB \rightarrow Be \otimes_C eB \rightarrow 0$ is a $B^\ev$-projective resolution of $Be \otimes_C eB$. Thus,
\begin{align*}
\Ext^n_{B^\ev}(Be \otimes_{eBe} eB,B) = H^n & \Hom_{B^\ev}(Be \otimes_{eBe} \mathbb P(eBe) \otimes_{eBe} eB, B)\\ &\xrightarrow{\hspace*{1.7pt}\sim\hspace*{1.7pt}} H^n \Hom_{C^\ev}(\mathbb P(eBe), C) = \Ext^n_{(eBe)^\ev}(eBe,eBe) \, ,
\end{align*}
where the isomorphism is by adjunction.
\end{proof}

\begin{nn}\label{nn:fundexact}
We have the following fundamental short exact sequence of $B^\ev$-modules.
\begin{equation}\tag{$\diamond$}\label{eq:fundexactsequ}
0 \longrightarrow \Omega^1_B \longrightarrow (Be \otimes_C eB) \oplus (Be' \otimes_{C'} e'B) \longrightarrow B \longrightarrow 0 \, .
\end{equation}
The second map is induced by the multiplications $\mu_e\colon (Be \otimes_C eB) \rightarrow B$ and $\mu_{e'}\colon Be' \otimes_{C'} e'B \rightarrow B$, where surjectivity follows from
$$
\mu_e(ce \otimes e) + \mu_{e'}(ce' \otimes e') = c \quad \text{(for $c \in B$)}.
$$
The sequence (\ref{eq:fundexactsequ}) gives rise to a long exact sequence (after applying $\Hom_{B^\ev}(-,B)$) which, locally, looks as follows.
$$
\cdots \xrightarrow{\ \ } \Ext^{n}_{B^\ev}(B,B) \xrightarrow{\ \ } \Ext^{n}_{B^\ev}((Be \otimes_C eB) \oplus (Be' \otimes_{C'} e'B),B) \xrightarrow{\ \ } \Ext^{n}_{B^\ev}(\Omega_B^1,B) \xrightarrow{\ \ } \cdots \, .
$$
By Lemma \ref{lem:isoidealtor}, this long exact sequence reads as
$$
\xymatrix@C=14.5pt{
\cdots \ar[r] & \Ext^{n}_{B^\ev}(B,B) \ar[r] & \Ext^n_{C^\ev}(C,C) \oplus \Ext^n_{(C')^\ev}(C',C') \ar[r] & \Ext^n_{B^\ev}(\Omega^1_B,B) \ar[r] & \cdots \, .
}
$$
To summarise, we have obtained a stronger statement than \cite[Thm.\,1.2]{GrSo02}, in the following sense.
\end{nn}

\begin{thm}\label{thm:greensol}
Let $e \in B$ be an idempotent, $e' = 1 - e$ its complementary idempotent and $C = eBe$, $C' = e'Be'$. If
$$
\Tor^{C}_i(Be,eB) = 0 = \Tor^{C'}_i(Be',e'B) \quad \text{$($for all $i > 0)$},
$$
then there is a long exact sequence
$$
\xymatrix@C=14.5pt{
\cdots \ar[r] & \Ext^{n}_{B^\ev}(B,B) \ar[r]^-{g_n} & \Ext^n_{C^\ev}(C,C) \oplus \Ext^n_{(C')^\ev}(C', C') \ar[r] & \Ext^{n}_{B^\ev}(\Omega^1_B,B) \ar[r] & \cdots
}
$$
induced from the canonical short exact sequence
$$
\xymatrix@C=14pt{
0 \ar[r] & {\Omega}^1_B \ar[r] & \displaystyle (Be \otimes_{C} eB) \oplus (Be' \otimes_{C'} e'B) \ar[r] & B \ar[r] & 0
}.
$$
If, moreover, $B$ is $K$-projective, then the map $g_\ast\colon \HH^\ast(B) \rightarrow \HH^\ast(eBe) \times \HH^\ast(e'Be')$ is a homomorphism of strict Gerstenhaber algebras.
\end{thm}

\begin{proof}
It suffices to prove that the component maps $g^C_\ast$ and $g^{C'}_\ast$ of $g_\ast$ coincide with the maps $\chi = \chi_{B/C}$ and $\chi' = \chi_{B/C'}$ that we have studied earlier. We will only take care of $g_n^C = \chi$ as the proof for $g_\ast^{C'} = \chi'$ is identical. Indeed, recall that by Proposition \ref{prop:Ieischi}, the map $\chi$ is given by applying the functor
$$
e(-)e \colon \Mod(B^\ev) \longrightarrow \Mod(C^\ev) \, ,
$$
i.e., it is induced by the standard recollement $\mathcal R(B^\ev,e \otimes e^\op)$. We show that the same holds true for $g_\ast$. To begin with, let $\mathbb PC \rightarrow C \rightarrow 0$ be a $C^\ev$-projective resolution of $C$. By the proof of Lemma \ref{lem:isoidealtor}, $Be \otimes_C \mathbb PC \otimes_C eB \rightarrow Be \otimes_{C} eB \rightarrow 0$ is a projective resolution of $Be \otimes_C eB$ over $B^\ev$. Further, pick a lifting
$$
\small{\xymatrix@C=11pt{
Be \otimes_C \mathbb PC \otimes_C eB \ar[d]_-{\Phi} & \cdots \ar[r] & Be \otimes_C P_n \otimes_{C} eB \ar[d]_-{\varphi_n} \ar[r] & \cdots \ar[r] & Be \otimes_C P_0 \otimes_{C} eB \ar[r] \ar[d]^-{\varphi_0} & Be \otimes_C eB \ar[r] \ar[d]^-{\mu_e} & 0\\
\mathbb B B & \cdots \ar[r] & B^{\otimes(n+2)} \ar[r] & \cdots \ar[r] & B \otimes B \ar[r] & B \ar[r] & 0
}}
$$
of $\mu_e \colon Be \otimes_C eB \rightarrow B$. If now $\alpha \in H^n \Hom_{B^\ev}(\mathbb B B, B)$ is represented by a cocycle $\varphi\colon B^{\otimes(n+2)} \rightarrow B$, the map $g_n^C$ will take $\alpha$ to the equivalence class of the cocycle
$$
\varphi \circ \varphi_n \colon Be \otimes_C P_n \otimes_{C} eB \longrightarrow B 
$$
in $H^n \Hom_{B^\ev}(Be \otimes_C \mathbb PC \otimes_C eB, B)$, which, in turn, will be taken to the map
$$
(e\varphi e) \circ (e \varphi_n e) = e(\varphi \circ \varphi_n)e \colon P_n \xrightarrow{\, \sim \,} eBe \otimes_C P_n \otimes_{C} eBe \xrightarrow{\, \ \, } eBe = C
$$
by the canonical isomorphism $\Hom_{B^\ev}(Be \otimes_C \mathbb PC \otimes_C eB, B) \cong \Hom_{C^\ev}(\mathbb P C,C)$.

On the other hand, $\varphi$ corresponds to the Yoneda extension $S \in \mathcal Ext^n_{B^\ev}(B,B)$ given by the lower row in the pushout diagram
$$
\xymatrix@C=18pt{
\cdots \ar[r] & B^{\otimes(n+2)} \ar[d]_-{\varphi} \ar[r] & B^{\otimes(n+1)} \ar[r] \ar[d] & \cdots \ar[r] & B \otimes B \ar@{=}[d] \ar[r] & B \ar@{=}[d] \ar[r] & 0 \, \, \\
0 \ar[r] & B \ar[r] & B \oplus_{B^{\otimes(n+2)}} B^{\otimes(n+1)} \ar[r] & \cdots \ar[r] & B \otimes B \ar[r] & B \ar[r] & 0 \, ,
}
$$
whereas $\psi = e(\varphi \circ \varphi_n) e$ corresponds to the lower sequence $T$ in 
$$
\xymatrix@C=18pt{
\cdots \ar[r] & P_n \ar[d]_-{\psi} \ar[r] & P_{n-1} \ar[r] \ar[d] & \cdots \ar[r] & P_0 \ar@{=}[d] \ar[r] & C \ar@{=}[d] \ar[r] & 0 \, \, \\
0 \ar[r] & C \ar[r] & C \oplus_{P_n} P_{n-1} \ar[r] & \cdots \ar[r] & P_0 \ar[r] & C \ar[r] & 0 \, .
}
$$
Now the map
$$
C \oplus P_{n-1} \longrightarrow C \oplus e (B^{\otimes(n+1)}) e, \ (c, p) \mapsto (c,e\varphi_{n-1}(p)e)
$$
induces a well-defined map
$$
\tilde{\varphi}_{n-1} \colon C \oplus_{P_n} P_{n-1} \longrightarrow e \left(B \oplus_{B^{\otimes(n+2)}} B^{\otimes(n+1)}\right)e \,.
$$
To conclude, observe that we obtain the commutative diagram
$$
\xymatrix@C=15pt{
T \ar[d] & 0 \ar[r] & C \ar[r] \ar@{=}[d] & C \oplus_{P_n} P_{n-1} \ar[r] \ar[d]_-{\tilde{\varphi}_{n-1}} & \cdots \ar[r] & P_0 \ar[d]^-{e\varphi_0 e} \ar[r] & C \ar[r] \ar@{=}[d] & 0\,\, \\
eSe & 0 \ar[r] & eBe \ar[r] & e\left(B \oplus_{B^{\otimes(n+2)}} B^{\otimes(n+1)}\right)e \ar[r] & \cdots \ar[r] & eB \otimes Be \ar[r] & eBe \ar[r] & 0 \, ,
}
$$
that is, a morphism $T \rightarrow eSe$ of $n$-extensions over $C^\ev$, which finishes the proof.
\end{proof}

\begin{exa}\label{exa:matrixalg}
Let $R$ and $S$ be $K$-algebras, and let $M$ be a $R \otimes S^\op$-module. Assume further, that $R$, $S$ and $M$ are projective as $K$-modules. Then the matrix algebra
$$
B = \left(\begin{matrix}
R & M\\
0 & S
\end{matrix}
\right)
$$
along with the canonical idempotents
$$
e = \left(\begin{matrix}
1 & 0\\
0 & 0
\end{matrix}\right) \quad \text{and} \quad e' = \left(\begin{matrix}
0 & 0\\
0 & 1
\end{matrix}\right)
$$
fulfils the requirements of Theorem \ref{thm:greensol}, because of $\Tor^{eBe}_i(Be,eB) = \Tor^R_i(0,M) = 0$ and $\Tor^{e'Be'}_i(Be',e'B) = \Tor^S_i(M,0) = 0$ for all $i > 0$.
\end{exa}

\begin{cor}\label{cor:tensorvanish}
Let $B$ be a matrix algebra as above. Then there is a long exact sequence
\begin{align*}
\cdots \longrightarrow \HH^n(B) \xrightarrow{\, g_n \, } \HH^n(R) \oplus \HH^n(S) \xrightarrow{\, h_n \, } \Ext^{n}_{R \otimes S^\op}(M,M) \longrightarrow \cdots
\end{align*}
in Hochschild cohomology. The graded map $g_\ast\colon \HH^\ast(B) \rightarrow \HH^\ast(R) \times \HH^\ast(S)$ is a homomorphism of strict Gerstenhaber algebras.
\end{cor}

\begin{proof}
Under the given assumptions, the module $\Omega^1_B$ evidently is isomorphic to $\left(\begin{smallmatrix}0 & M\\ 0 & 0 \end{smallmatrix}\right) = M$. Moreover, we have$$
\Ext^{n}_{B^\ev}(M,B) \xrightarrow{\, \sim \, } \Ext^{n}_{R \otimes S^\op}(M,M)
$$
as any projective resolution $\mathbb P M \rightarrow M \rightarrow 0$ over $R \otimes S^\op$ is already a $B^\ev$-projective resolution $\mathbb P M \rightarrow M \rightarrow 0$. Indeed, we have canonical isomorphisms
$$
Be \otimes_R \mathbb P M \otimes_S e'B \cong R \otimes_R \mathbb P M \otimes_S S \cong \mathbb P(M)
$$
which yield the desired identification. Thus, by adjunction,
$$
H^n \Hom_{B^\ev}(\mathbb P M, B) \cong H^n \Hom_{R \otimes S^\op}(\mathbb PM, eB \otimes_B B \otimes_B Be') \cong \Hom_{R \otimes S^\op}(\mathbb PM, M)
$$
and the assertion follows from Theorem $\ref{thm:greensol}$.
\end{proof}

\begin{cor}\label{cor:strexc}
Let $B$ be a matrix algebra as above. 
\begin{enumerate}[\rm(1)]
\item If $d = \pd_{R \otimes S^\op}(M)$ is finite, then $g_n \colon \HH^n(B) \rightarrow \HH^n(R) \times \HH^n(S)$ is an isomorphism for $n > d$.

\item If $M$ is \emph{strongly exceptional} over $R \otimes S^\op$ $($i.e., $\Ext^n_{R \otimes S^\op}(M,M) = 0$ for $n > 0$ and $\End_{R \otimes S^\op}(M) = K)$, then there is the short exact sequence
$$
0 \longrightarrow \HH^0(B) \xrightarrow{\ g_0 \, } \HH^0(R) \oplus \HH^0(S) \longrightarrow K \longrightarrow 0
$$
and $g_{> 0} \colon \HH^{> 0}(B) \rightarrow \HH^{> 0}(R) \times \HH^{> 0}(S)$ is an isomorphism of non-unital strict Gerstenhaber algebras.
\end{enumerate}\qed
\end{cor}

The special case $M = 0$ also yields an interesting consequence.

\begin{cor}
Let $R$ and $S$ be $K$-algebras being projective as $K$-modules. Then
$$
\HH^\ast(R \times S) \cong \HH^\ast(R) \times \HH^\ast(S)
$$
as strict Gerstenhaber algebras; that is, the strict Gerstenhaber structure on Hochschild cohomology commutes with finite products.\qed
\end{cor}

We continue with an analysis of the map $\HH^\ast(R) \times \HH^\ast(S) \rightarrow \Ext^\ast_{R \otimes S^\op}(M,M)$.

\begin{lem}\label{lem:grsolfunctor}
Let $B$ be a matrix algebra as in Example $\ref{exa:matrixalg}$. The respective restrictions of the functors $- \otimes_R M \colon \Mod(R^\ev) \rightarrow \Mod(R \otimes S^\op)$ and $M \otimes_S -\colon \Mod(S^\ev) \rightarrow \Mod(R \otimes S^\op)$ to $\P(R)$ and $\P(S)$ are exact. Moreover, the following statements hold true.
\begin{enumerate}[\rm(1)]
\item\label{lem:grsolfunctor:1a} Assume that $M = eBe'$ is a projective left $R = eBe$-module. Then the functors $- \otimes_R M \colon \Mod(R^\ev) \rightarrow \Mod(R \otimes S^\op)$ and $M \otimes_S - \colon \Mod(S^\ev) \rightarrow \Mod(R \otimes S^\op)$ restrict to exact functors
$$
- \otimes_R M \colon \mathsf P(R) \longrightarrow \mathsf P_R(R,S), \quad M \otimes_S - \colon \P(S) \longrightarrow \mathsf P_R(R,S)
$$
where $\mathsf P_R(R,S) \subseteq \Mod(R \otimes S^\op)$ denotes the full subcategory of those modules that are projective as left $R$-modules.

\item\label{lem:grsolfunctor:1b} Assume that $M = eBe'$ is a projective right $S = e'Be'$-module. Then the functors $- \otimes_R M \colon \Mod(R^\ev) \rightarrow \Mod(R \otimes S^\op)$ and $M \otimes_S - \colon \Mod(S^\ev) \rightarrow \Mod(R \otimes S^\op)$ restrict to exact functors
$$
- \otimes_R M \colon \mathsf P(R) \longrightarrow \mathsf P_S(R,S), \quad M \otimes_S - \colon \P(S) \longrightarrow \mathsf P_S(R,S)
$$
where $\mathsf P_S(R,S) \subseteq \Mod(R \otimes S^\op)$ denotes the full subcategory of those modules that are projective as right $S$-modules.

\item\label{lem:grsolfunctor:2} The induced maps
$$
\HH^\ast(R) \cong \Ext^\ast_{\P(R)}(R,R) \xrightarrow{\ \ } \Ext^\ast_{\mathsf P_S(R,S)}(M, M) \xrightarrow{\, \sim \, } \Ext^\ast_{R \otimes S^\op}(M,M)
$$
and
$$
\HH^\ast(S) \cong \Ext^\ast_{\P(S)}(S,S) \xrightarrow{\ \ } \Ext^\ast_{\mathsf P_R(R,S)}(M, M) \xrightarrow{\, \sim \, } \Ext^\ast_{R \otimes S^\op}(M,M)
$$
are the component maps in the long exact sequence of Corollary $\ref{cor:tensorvanish}$ $($up to the sign $(-1)^{\delta_{\Sigma,S}}$, where $\delta_{\Sigma,S}$ is the Kronecker delta for $\Sigma \in \{R, S\})$. In particular, those $($signed$)$ component maps are homomorphisms of graded $K$-algebras.
\end{enumerate}
\end{lem}

\begin{proof}
Exactness of the restricted functors is automatic, as they are additive and the bimodules they are defined on are projective on either side. If $X$ is a module in $\mathsf P(R)$, then
$$
\Hom_{R}(X \otimes_R M, -) \cong \Hom_{R}(M, \Hom_{R}(X,-))
$$
is exact by assumption; thus $X \otimes_R M \in \mathsf P_R(R,S)$. Further, for $Y \in \mathsf P(S)$,
$$
\Hom_{R}(M \otimes_S Y,-) \cong \Hom_{S^\op}(Y, \Hom_{R}(M,-)),
$$
so that $M \otimes_S Y \in \P_R(R,S)$, whence (\ref{lem:grsolfunctor:1a}) follows. Similar arguments apply to deduce (\ref{lem:grsolfunctor:1b}). As for item (\ref{lem:grsolfunctor:2}), notice that the bar resolution $\mathbb B R \rightarrow R \rightarrow 0$ gives rise to a projective resolution $Be \otimes_R \mathbb B R \otimes_R eB \rightarrow Be \otimes_R eB \rightarrow 0$ over $R \otimes S^\op$ (indeed, as $M$ is projective as a right $S$-module, the functor $\Hom_{R \otimes S^\op}((R \otimes R) \otimes_R M, -) \cong \Hom_R(R, \Hom_{S^\op}(M,-))$ is exact). The latter resolution may be expressed as
$$
\xymatrix{
Be \otimes_R \mathbb B R \otimes_R eB \ar[r] \ar@{=}[d] & Be \otimes_R eB \ar[r] \ar@{=}[d] & 0\, \,\\
\displaystyle {\left(\begin{matrix}
\mathbb B R & \mathbb B R \otimes_R M\\
0 & 0
\end{matrix}\right)} \ar[r] & \displaystyle {\left(\begin{matrix}
R & M\\
0 & 0
\end{matrix}\right)} \ar[r] & 0 \, .
}
$$
The adjunction isomorphism yields, for $n \geqslant 0$,
\begin{align*}
H^n \Hom_{R^\ev}(\mathbb B R, R) \xrightarrow{\, \sim \, } H^n \Hom_{B^\ev}&(Be \otimes_R \mathbb B R \otimes_R eB, B)\\ &= H^n \Hom_{B^\ev}\left({\left(\begin{matrix}
\mathbb B R & \mathbb B R \otimes_R M\\
0 & 0
\end{matrix}\right)}, B\right)
\end{align*}
which sends the equivalence class of a cocylce $\varphi \colon R^{\otimes(n + 2)} \rightarrow R$ to the map $\psi(\varphi)$ given by
$$
\psi(\varphi){\left(\begin{matrix}
r & r' \otimes m\\
0 & 0
\end{matrix}\right)} = \left(\begin{matrix}
\varphi(r) & \varphi(r')m\\
0 & 0
\end{matrix}\right) \ \in \ B \
\quad (\text{for $r, r' \in R^{\otimes(n + 2)}$ and $m \in M$}).
$$
The cocycle $\psi = \psi(\varphi)$ corresponds to the middle sequence in the diagram with exact rows below, whereas the rightmost sequence represents its image (under $h^R_n$) in $\Ext^n_{B^\ev}(M,B)$.
$$
\xymatrix@C=24pt{
\ar[d] \tvdots & 0 \ar[d] & 0 \ar[d]\\
{\left(\begin{matrix}
R^{\otimes(n+2)} & R^{\otimes(n+2)} \otimes_R M\\
0 & 0
\end{matrix}\right)} \ar[r]^-{\psi} \ar[d] & B \ar[d] & \ar@{=}[l] \ar[d] B\\
{\left(\begin{matrix}
R^{\otimes(n+1)} & R^{\otimes(n+1)} \otimes_R M\\
0 & 0
\end{matrix}\right)} \ar[r] \ar[d] & P(\psi) \ar[d] & P(\psi) \ar[d] \ar@{=}[l] \\
\tvdots \ar[d] & \tvdots \ar[d] & \tvdots \ar[d] \\
{\left(\begin{matrix}
R \otimes R & (R \otimes R) \otimes_R M\\
0 & 0
\end{matrix}\right)} \ar@{=}[r] \ar[d] & {\left(\begin{matrix}
R \otimes R & (R \otimes R) \otimes_R M\\
0 & 0
\end{matrix}\right)} \ar[d] & \ar[l] Q(\mathrm{inc}) \ar[d]\\
{\left(\begin{matrix}
R & M\\
0 & 0
\end{matrix}\right)} \ar@{=}[r] \ar[d] & {\left(\begin{matrix}
R & M\\
0 & 0
\end{matrix}\right)} \ar[d] & M \ar@{ >->}[l]_-{\mathrm{inc}} \ar[d]\\
0 & 0 & 0
}
$$
Here $P(\psi)$ denotes the pushout and $Q(\mathrm{inc})$ denotes the pullback of the obvious diagrams. Observe that $Q = Q(\mathrm{inc})$ is canonically isomorphic to $(R \otimes R) \otimes_R M \cong R \otimes M$. Further, if $\beta_n \colon R^{\otimes(n+2)} \rightarrow R^{\otimes(n+1)}$ denotes the $n$-th differential in $\mathbb B R$ and $\tilde{\beta}_n$ is the map
$$
{\left(\begin{matrix}
\beta_n & \beta_n \otimes_R M\\
0 & 0
\end{matrix}\right)} \colon {\left(\begin{matrix}
R^{\otimes(n+2)} & R^{\otimes(n+2)} \otimes_R M\\
0 & 0
\end{matrix}\right)} \longrightarrow {\left(\begin{matrix}
R^{\otimes(n+1)} & R^{\otimes(n+1)} \otimes_R M\\
0 & 0
\end{matrix}\right)} \, ,
$$
the module $P(\psi)$ is given by the cokernel in the exact sequence
$$
\xymatrix@C=24pt{
{\left(\begin{matrix}
R^{\otimes(n+2)} & R^{\otimes(n+2)} \otimes_R M\\
0 & 0
\end{matrix}\right)} \ar[r]^-{
\left(
\begin{smallmatrix}
\psi\\
- \tilde{\beta}_n
\end{smallmatrix}
\right)
} & B \oplus {\left(\begin{matrix}
R^{\otimes(n+1)} & R^{\otimes(n+1)} \otimes_R M\\
0 & 0
\end{matrix}\right)} \ar@{->>}[r] & \Coker{\left(\begin{smallmatrix}
\psi\\
- \tilde{\beta}_n
\end{smallmatrix}\right)}
}
$$
and therefore applying the exact functor $e(-)e' \cong eB \otimes_B (-) \otimes_B Be'$ yields the exact sequence
$$
\small{
\xymatrix@C=10pt{
e{\left(\begin{matrix}
R^{\otimes(n+2)} & R^{\otimes(n+2)} \otimes_R M\\
0 & 0
\end{matrix}\right)}e' \ar[r] \ar@{=}[d] & eBe' \oplus e{\left(\begin{matrix}
R^{\otimes(n+1)} & R^{\otimes(n+1)} \otimes_R M\\
0 & 0
\end{matrix}\right)}e' \ar@{->>}[r] \ar@{=}[d] & e\Coker{\left(\begin{smallmatrix}
\psi\\
- \tilde{\beta}_n
\end{smallmatrix}\right)}e' \ar[d]^-\cong\\
R^{\otimes(n+2)} \otimes_R M \ar[r]^-{\left(\begin{smallmatrix}
\varphi \otimes_R M\\
- {\beta}_n \otimes_R M
\end{smallmatrix}\right)} & M \oplus R^{\otimes(n+1)} \otimes_R M \ar@{->>}[r] & \Coker{\left(\begin{smallmatrix}
\varphi \otimes_R M\\
- {\beta}_n \otimes_R M
\end{smallmatrix}\right)}
}}
$$
As ${\left(\begin{smallmatrix}
\varphi \otimes_R M\\
- {\beta}_n \otimes_R M
\end{smallmatrix}\right)} = {\left(\begin{smallmatrix}
\varphi\\
- {\beta}_n
\end{smallmatrix}\right) \otimes_R M}$, it thus follows that $eP(\psi)e'$ is the module $P(\varphi) \otimes_R M$, where $P(\varphi)$ denotes $\Coker{\left(\begin{smallmatrix}
\varphi\\
- {\beta}_n
\end{smallmatrix}\right)}$, i.e., the pushout
$$
\xymatrix{
R^{\otimes(n+2)} \ar[r]^-{\beta_n} \ar[d]_-{\varphi} & R^{\otimes(n+1)} \, \, \ar[d] \\
R \ar[r] & P(\varphi) \, .
}
$$
Thus, under the isomorphism $\Ext^n_{B^\ev}(M,B) \cong \Ext^n_{R \otimes S^\op}(M,M)$, the sequence that corresponds to $\psi$ is taken to
$$
\small{
\xymatrix@C=10pt@R=20pt{
0 \ar[r] & eBe' \ar@{=}[d] \ar[r] & eP(\psi)e' \ar@{=}[d] \ar[r] & e{\left(\begin{matrix}
R^{\otimes n} & R^{\otimes n} \otimes_R M\\
0 & 0
\end{matrix}\right)}e' \ar@{=}[d] \ar[r] & \cdots \ar[r] & e Q e' \ar[r] \ar@{=}[d] & eMe' \ar@{=}[d] \ar[r] & 0\\
0 \ar[r] & M \ar[r] & P(\varphi)\otimes_R M \ar[r] & R^{\otimes n} \otimes_R M \ar[r] & \cdots \ar[r] & (R \otimes R) \otimes_R M \ar[r] & M \ar[r] & 0
}}
$$
which is precisely the image under $- \otimes_R M$ of the sequence that corresponds to $\varphi$. The second assertion follows similarly.
\end{proof}

\begin{nn}
Assume that $R = S$ and $M = R$; then $\P_R(R,S) = \P_\lambda(R)$, and the functors
$$
- \otimes_R R,\, R \otimes_R - \colon \P(R) = \P_\lambda(R) \cap \P_\varrho(R) \longrightarrow \P_\lambda(R)
$$
are isomorphic to the inclusion functor. In particular, these are very strong monoidal functors
$$
(\P(R), \otimes_R, R) \longrightarrow (\P_\lambda(R), \otimes_R, R) \, .
$$
Thus Lemma \ref{lem:grsolfunctor} implies, when combined with Theorem \ref{thm:gersthochcomp} and Theorem \ref{thm:greensol}, the following statement.
\end{nn}

\begin{cor}\label{cor:tripleR}
Let $R$ be a $K$-algebra which is projective as a $K$-module. Further, let $B$ be the matrix algebra
$$
B = \left(\begin{matrix}
R & R\\
0 & R
\end{matrix}
\right) \, .
$$
Then the long exact sequence
\begin{align*}
\cdots \longrightarrow \HH^n(B) \xrightarrow{\, g_n \, } \HH^n(R) \oplus \HH^n(R) \xrightarrow{\, h_n \, } \Ext^{n}_{R \otimes R^\op}(R,R) = \HH^n(R) \longrightarrow \cdots
\end{align*}
as stated in Theorem $\ref{thm:greensol}$, or Corollary $\ref{cor:tensorvanish}$, decomposes into short exact sequences
$$
0 \longrightarrow \HH^n(B) \xrightarrow{\, g_n \, } \HH^n(R) \oplus \HH^n(R) \xrightarrow{\, h_n \, } \HH^n(R) \longrightarrow 0\,,
$$
wherein the graded map $g_\ast \colon \HH^\ast(B) \rightarrow \HH^\ast(R) \times \HH^\ast(R)$ and the component maps $\HH^\ast(R) \rightarrow \Ext^\ast_{R^\ev}(R,R) \cong \HH^\ast(R)$, $\HH^\ast(R) \rightarrow \Ext^\ast_{R^\ev}(R,R) \cong \HH^\ast(R)$ of $h_\ast$ are homomorphisms of strict Gerstenhaber algebras.\qed
\end{cor}

\subsection{The long exact sequence of Happel}\label{subsec:happel}
By specialising the long exact sequence in Theorem \ref{thm:greensol} to a matrix algebra $B = \left(\begin{smallmatrix} R & M\\ 0 & K \end{smallmatrix}\right)$, where $M$ is an $R$-module, and $B = R[M]$ is thus the \textit{one-point-extension} of $R$ by $M$, we arrive at a long exact sequence
$$
\xymatrix@C=24pt@R=6pt{
0 \ar[r] & \HH^0(B) \ar[r]^-{g_0} & \HH^0(R) \oplus K \ar[r]^-{(h_0^R \, h_0^K)} & \Hom_R(M,M) \ar[r] & \HH^{1}(B) \ar[r] & \cdots\\
\cdots \ar[r] & \HH^n(B) \ar[r]^-{g_n} & \HH^n(R) \ar[r]^-{h_n^R} & \Ext^{n}_R(M,M) \ar[r] & \HH^{n+1}(B) \ar[r] & \cdots
}
$$
first discovered by Happel in \cite{Ha89}, and picked up for further analysis by Green-Marcos-Snashall in \cite{GrMaSn03}, resulting in the insight, that the maps $g_\ast$ and $h_\ast^R$ define homomorphisms of graded $K$-algebras. By the results established previously, we immediately deduce the following.

\begin{prop}\label{prop:happel}
Let $R$ be a $K$-algebra and $M$ be an $R$-module. Assume further, that $R$ and $M$ are projective as $K$-modules. Then the one-point-extension $B = R[M]$ of $R$ by $M$ fits inside a long exact sequence
$$
\xymatrix@C=20pt{
\cdots \ar[r] & \HH^n(B) \ar[r]^-{g_n} & \HH^n(R) \ar[r]^-{h_n^R} & \Ext^{n}_R(M,M) \ar[r] & \HH^{n+1}(B) \ar[r] & \cdots
}
$$
such that the map $g_\ast$ is a homomorphism of strict Gerstenhaber algebras, and $h^R_\ast$ is a homomorphism of graded $K$-algebras.\qed
\end{prop}

\begin{rem}
Happel's sequence comes in handy, when examining \textit{accessible algebras}, as it has been done in work by Lenzing and de la Pe{\~n}a (see \cite{LedlP08} and also \cite{BaLe03}, \cite{Gei02}). Recall that, for a $K$-algebra $B$, a $B$-module $E$ is called \textit{exceptional} if $\End_B(E) \cong K$ and $\Ext^1_B(E,E) = 0$. The $K$-algebra $B$ is called \textit{accessible} if there is an \textit{accessible tower}
$$
K = B_0 , B_1, \dots, B_n = B \,;
$$
that is, a sequence of $K$-algebras as indicated above, such that $B_{t} \cong B_{t-1}[E_{t-1}]$ is the one-point-extension of $B_{t-1}$ by some exceptional $(B_{t-1})$-module $E_{t-1}$ (for $1 \leqslant t \leqslant n$). Thus, the Gerstenhaber algebra structure on $\HH^\ast(B)$ may be understood via the same for the Hochschild cohomology of its \textit{steps} $B_t$. For instance, we immediately deduce from Happel's sequence, and the exceptionality of the $E_t$, that there are short exact sequences
$$
0 \longrightarrow Z(B_t) \longrightarrow Z(B_{t-1}) \oplus K\longrightarrow K \longrightarrow 0 \quad \text{(for $1 \leqslant t \leqslant n$)}
$$
which induce algebra isomorphisms $Z(B_t) \cong Z(B_{t-1})$, and Lie algebra isomorphisms
$$
\HH^1(B_{t}) = \Out_K(B_t) \xrightarrow{\, \sim \, } \Out_K(B_{t-1}) = \HH^1(B_{t-1}) \quad \text{(for $1 \leqslant t \leqslant n$)}.
$$
Compare this with Corollary \ref{cor:strexc}.
\end{rem}

\begin{nn}
Let us specialise further. Let $R$ be a $K$-projective $K$-algebra and consider
$$
B = \left(\begin{matrix}
R \otimes R^\op & R\\
0 & K
\end{matrix}\right)\, ,
$$
that is, the one-point-extension $B = U[R]$ of $U = R^\ev$ by $R$. By Proposition \ref{prop:happel}, we already know that the map $g_\ast$ defines a homomorphism of strict Gerstenhaber algebras. In the following, we take the map $h_\ast$ under further analysis.
\end{nn}

\begin{nn}
To this end, consider the full subcategory $\overline{\P}(U)$ of $\P(U)$ consisting of all objects $X \in \P(U)$ such that $X \otimes_U R$ belongs to $\P(R)$. Of course, $U$ and $U \otimes U$ belong to $\overline{\P}(U)$. Further, let $\tilde{\P}(U)$ be the full subcategory of $\overline{\P}(U)$ consisting of all objects $X \in \overline{\P}(U)$ such that $X \otimes_U Y$ belongs to $\overline{\P}(U)$ for all objects $Y \in \overline{\P}(U)$. We will argue that $\tilde{\P}(U)$ defines an exact monoidal subcategory of $(\Mod(U^\ev), \otimes_U, U)$.

Firstly, $\tilde{\P}(U)$ is non-empty as $U$ and $U \otimes U$ belong to $\tilde{\P}(U)$. Indeed, for $Y \in \overline{\P}(U)$,
$$
U \otimes_U Y \cong Y \ \in \ \overline{\P}(U)
\quad
\text{and}
\quad
(U \otimes U) \otimes_U Y \cong U \otimes Y \ \in \ \overline{\P}(U)\,.
$$
The subcategory $\tilde{\P}(U)$ is closed under arbitrary direct sums and summands (taken in $\Mod(U^\ev)$). Moreover, $\tilde{\P}(U)$ is extension closed and closed under kernels of epimorphisms, as any admissible short exact sequence $0 \rightarrow X'' \rightarrow X \rightarrow X' \rightarrow 0$ in $\P(U)$ with $X' \in \tilde{\P}(U)$ gives rise to a split exact sequence
$$
0 \longrightarrow (X'' \otimes_U Y) \otimes_U R \longrightarrow (X \otimes_U Y) \otimes_U R \longrightarrow (X' \otimes_U Y) \otimes_U R \longrightarrow 0 \quad \text{(for $Y \in \overline{\P}(U)$)}
$$
of left and right $R$-modules. It follows from Proposition \ref{prop:entireextclo} that $\tilde{\P}(U)$ is entirely extension closed in $\Mod(U^\ev)$. Finally, $\tilde{\P}(U)$ is closed under taking tensor products over $U$, as $(X \otimes_U X') \otimes_U Y \cong X \otimes_U (X' \otimes_U Y) \in \overline{P}(U)$ for all $X, X' \in \tilde{P}(U)$ and $Y \in \overline{P}(U)$.
\end{nn}

\begin{lem}
The functor $- \otimes_{U} R \colon \Mod(U^\ev) \rightarrow \Mod(U) = \Mod(R^\ev)$ restricts to an exact functor
$$
- \otimes_{U} R \colon (\tilde{\P}(U), \otimes_U, U) \longrightarrow (\P(R), \otimes_R, R),
$$
taking $U$ to $R$. The induced map
$$
\HH^\ast(U) \cong \Ext^\ast_{\tilde{\P}(U)}(U,U) \longrightarrow \Ext^\ast_{\P(R)}(R,R) \cong \HH^\ast(R)
$$
is a split surjection $($in the category of graded $K$-algebras$)$ and agrees with $h_\ast^U \colon \HH^\ast(U) \rightarrow \HH^\ast(R)$.
\end{lem}

\begin{proof}
By definition of $\tilde{\P}(U)$ it is apparent, that $- \otimes_U R$ restricts as claimed; since $\tilde{\P}(U) \subseteq \P(U)$, the restriction will also be exact. Moreover, we have the natural isomorphism
$$
R \xrightarrow{\ \sim \ } U \otimes_U R, \ r \mapsto 1 \otimes r \, .
$$
The exact functor
$$
- \otimes_R U = - \otimes_K R \colon (\P(R), \otimes_R, R) \longrightarrow (\Mod(U^\ev), \otimes_U, U)
$$
gives rise to an algebra homomorphism $\HH^\ast(R) \cong \Ext^\ast_{\P(R)}(R,R) \rightarrow \Ext^\ast_{U^\ev}(U,U) \cong \HH^\ast(U)$ which is right inverse to the one induced by $- \otimes_U R$ (see \cite[Thm.\,6.3.12]{He14b}). Finally, that we retain the map $h_\ast$ follows by a similar argument as in the proof of Lemma \ref{lem:grsolfunctor}(\ref{lem:grsolfunctor:2}).
\end{proof}

\begin{cor}
In the long exact sequence of Proposition $\ref{prop:happel}$ applied to $B = U[R]$,
$$
\xymatrix@C=20pt{
\cdots \ar[r] & \HH^n(B) \ar[r]^-{g_n} & \HH^n(U) \ar[r]^-{h_n} & \Ext^{n}_U(R,R) \ar[r] & \HH^{n+1}(B) \ar[r] & \cdots \, ,
}
$$
the graded map $g_\ast\colon \HH^\ast(U[R]) \rightarrow \HH^\ast(U) \times K$ is a homomorphisms of strict Gerstenhaber algebras and the graded map $h_\ast^U \colon \HH^\ast(U) \rightarrow \Ext^\ast_U(R,R) = \HH^\ast(R)$ is a homomorphism of graded $K$-algebras. Furthermore, the map $h_\ast^U$ is surjective, so that the long exact sequence decomposes into short exact sequences
$$
0 \longrightarrow Z(B) \longrightarrow Z(U) \oplus K \longrightarrow Z(R) \longrightarrow 0
$$
and
$$
0 \longrightarrow \HH^n(B) \longrightarrow \HH^n(U) \longrightarrow \HH^n(R) \longrightarrow 0
$$
for each integer $n \geqslant 1$. \qed
\end{cor}


\section{Homological epimorphisms}\label{sec:homepi}

\begin{nn}
The notion of a homological epimorphism has been introduced by Geigle--Lenzing in \cite{GeLe91}. Recall that a ring homomorphism $\pi \colon B \rightarrow A$ is an epimorphism in the category of rings, if $f \circ \pi = g \circ \pi$ for some ring homomorphisms $f, g\colon A \rightarrow R$ implies that $f = g$. Note that such homomorphisms are in general far from being surjective. For example, the canonical embedding $\mathbb Z \rightarrow \mathbb Q$ is an epimorphism in the above sense, but certainly not surjective. More generally, if $B$ is commutative and $\Sigma \subseteq B$ is a multiplicative subset, then the canonical map $B \rightarrow B[\Sigma^{-1}]$ is an epimorphism. However, any surjective ring homomorphism will, of course, be an epimorphism.
\end{nn}

\begin{nn}\label{nn:equhomepi}
By \cite[Prop.\,1.1]{Si67}, a ring homomorphism $\pi\colon B \rightarrow A$ is a ring epimorphism if, and only if, the multiplication map $A \otimes_B A \rightarrow A$ is an isomorphism, which is if, and only if, the restriction functor $\pi_\star = \mathrm{Res}(\pi)\colon \Mod(A) \rightarrow \Mod(B)$ is full and faithful.  A ring epimorphism $\pi\colon B \rightarrow A$ is \textit{homological} if $\Tor_i^B(A,A) = 0$ for $i > 0$. This is equivalent to the functor
$$
\mathbf D^b(\pi_\star) {\colon} \mathbf D^b(\Mod(A)) \longrightarrow \mathbf D^b(\Mod(B))
$$
being full and faithful, see \cite[Thm.\,4.4]{GeLe91}.
\end{nn}

In what follows, we will assume that $A$ and $B$ are $K$-algebras, and $\pi {\colon} B \rightarrow A$ is a $K$-linear homological epimorphism. We will further assume that $A$ and $B$ are \textit{projective} as $K$-modules.

\begin{lem}\label{lem:piev}
The $K$-algebra homomorphism $\pi^\ev = \pi \otimes \pi^\op {\colon} B^\ev \rightarrow A^\ev$ is a homological epimorphism.
\end{lem}

\begin{proof}
Indeed, one has
$$
\Tor^{B \otimes B^\op}_i(A \otimes A^\op, A \otimes A^\op) \cong \Tor^B_i(A, A \otimes A) = 0 \quad \text{(for $i > 0$)}
$$
by \cite[Chap.\,IX, Thm.\,2.8]{CaEi56} and \cite[Thm.\,4.4]{GeLe91}. Furthermore, if $f, g {\colon} A^\ev \rightarrow R$ are ring homomorphisms with $f \circ \pi^\ev = g \circ \pi^\ev$, they must fulfil $f(- \otimes 1) = g(- \otimes 1)$ and $f(1 \otimes - ) = g(1 \otimes -)$ as $\pi$ and $\pi^\op$ are epimorphisms. Thus $f = g$.
\end{proof}

We state an immediate consequence of Lemma \ref{lem:piev} and Paragraph \ref{nn:equhomepi}.

\begin{cor}
The graded $K$-algebras $\HH^\ast(A)$ and $\Ext^\ast_{B^\ev}(A,A)$ are isomorphic. \qed
\end{cor}

\begin{nn}
The restriction functor along $\pi^\ev$ has a left adjoint, given by
$$
A^\ev \otimes_{B^\ev} (-) \cong A \otimes_B (-) \otimes_B A.
$$
The following theorem states, that it gives rise to a very well behaved map between the Hochschild cohomology algebras of $B$ and $A$.
\end{nn}

\begin{thm}\label{thm:homepi}
The graded $K$-algebra map
$$
A \otimes_B^{\mathbf L} (-) \otimes_B^{\mathbf L} A {\colon} \HH^\ast(B) = \Hom_{\mathbf D^b(B^\ev)}(B,B[\ast]) \longrightarrow \Hom_{\mathbf D^b(A^\ev)}(A,A[\ast]) = \HH^\ast(A)
$$
is a homomorphism of strict Gerstenhaber algebras; that is, it renders the diagrams
$$
\xymatrix@C=30pt{
\HH^m(B) \times \HH^n(B) \ar[r]^-{\{-,-\}_B} \ar[d] & \HH^{m+n-1}(B) \ar[d] \\ 
\HH^m(A) \times \HH^n(A) \ar[r]^-{\{-,-\}_A} & \HH^{m+n-1}(A)
}\quad\quad
\xymatrix{
\HH^{2m}(B) \ar[r]^-{sq_B} \ar[d] & \HH^{4m-1}(B) \ar[d] \\ 
\HH^{2m}(A) \ar[r]^-{sq_A} & \HH^{4m-1}(A)
}
$$
commutative for all integers $m, n \geqslant 0$.
\end{thm}

The proof requires several preliminary observations whose main goal is to construct an exact and entirely extension closed monoidal subcategory $\mathsf C$ of $(\Mod(B^\ev), \otimes_B, B)$, such that $(A \otimes_B (-) \otimes_B A)\mathord{\upharpoonright}_{\C}$ fulfils the requirements of Theorem \ref{thm:gersthochcomp}.

\begin{nn}\label{nn:tinftycinfty}
Recall that $\P(B) = \P_\lambda(B) \cap \P_\varrho(B)$ denotes the full, exact monoidal subcategory of $(\Mod(B^\ev), \otimes_B, B)$ of those $B^\ev$-modules which are projective when considered as left and right $B$-modules. Let $\mathsf T_\infty(B)$ be the full subcategory of $\Mod(B^\ev)$ consisting of all $B^\ev$-modules $M$ with
$$
\Tor^{B \otimes B^\op}_i(A \otimes A^\op, M) = 0 \quad (\text{for all $i > 0$}).
$$
So, in different words, $\mathsf T_\infty(B) = \Ker \Tor^{B \otimes B^\op}_{> 0}(A \otimes A^\op,-)$. Further, we let $\P_\infty(B)$ be the full subcategory $\P(B) \cap \mathsf T_\infty(B)$ of $\Mod(B^\ev)$. As one has $\Tor^K_i(A,A) = 0 = \Tor^{B^\op}_i(A,X) = \Tor^B_i(X,A)$ for all $X \in \P(B)$ and $i > 0$, we may apply \cite[Chap.\,IX, Thm.\,2.8]{CaEi56} to conclude that
$$
\Tor^{B \otimes B^\op}_i(A \otimes A^\op, X) \cong \Tor^B_i(A, X \otimes_{B} A) \quad \text{(for $i > 0$)}.
$$
Thus, for $X \in \P(B)$,
$$
X \in \P_\infty(B) \quad \Longleftrightarrow \quad \Tor^{B}_i(A,X \otimes_B A) = 0 \quad (\text{for all $i > 0$}),
$$
and consequently $\P_\infty(B) = \P(B) \cap \Ker \Tor^B_{>0}(A, - \otimes_B A)$.
\end{nn}

\begin{lem}
The subcategory $\P_\infty(B)$ is extension closed in $\Mod(B^\ev)$, and contains $B$ and $B \otimes B$. Moreover, it is closed under taking direct summands and arbitrary direct sums in $\Mod(B^\ev)$. In particular, the subcategory $\P_\infty(B)$ is entirely extension closed in the sense of Paragraph $\ref{nn:extclosed} $. Lastly, the subcategory $\P_\infty(B)$ is, viewed as an exact category, closed under kernels of epimorphisms, that is, if $f{\colon} M \rightarrow N$ is a surjective map in $\P_\infty(B)$, then $\Ker(f)$ belongs to $\P_\infty(B)$.
\end{lem}

\begin{proof}
As the subcategory $\P(B)$ fulfils all of the properties stated above, it suffices to show that $\mathsf T_\infty(B)$ does as well. To begin with, $\Hom_B(B \otimes A, -) \cong \Hom_K(A,-)$ tells, that the module $(B \otimes B) \otimes_B A \cong B \otimes A$ is in particular $B$-flat, and thus
$$
\Tor^B_i(A,B \otimes_B A) \cong \Tor^B_i(A,A) = 0 = \Tor^B_i(A,(B \otimes B) \otimes_B A) \quad (\text{for $i > 0$}).
$$
Consequently $B$ and $B \otimes B$ belong to $\mathsf{T}_\infty(B)$. As $\Tor^{B \otimes B^\op}_i(A \otimes A^\op, -)$ commutes with (arbitrary) direct sums, $\mathsf{T}_\infty(B)$ is closed under taking summands and sums. Next, the long exact homology sequence
$$
\xymatrix@C=11.5pt@R=4pt{
\cdots \ar[r] & \Tor^{B^\ev}_{i+1}(A^\ev, N) \ar[r] & \Tor^{B^\ev}_i(A^\ev, L) \ar[r] & \Tor^{B^\ev}_i(A^\ev,M) \ar[r] & \Tor^{B^\ev}_i(A^\ev,N) \ar[r] & \cdots & \\
\cdots \ar[r] & \Tor^{B^\ev}_1(A^\ev,N) \ar[r] & A \otimes_B L \otimes_B A \ar[r] &  A \otimes_B M \otimes_B A \ar[r] & A \otimes_B N \otimes_B A \ar[r] & 0
}
$$
for some short exact sequence $0 \rightarrow L \rightarrow M \rightarrow N \rightarrow 0$ in $\Mod(B^\ev)$, shows that if $L$ and $N$ (or $M$ and $N$) belong to $\mathsf{T}_\infty(B)$, then so does $M$ (or, respectively, $L$). Finally, we conclude that $\mathsf T_\infty(B)$ is entirely extension closed in $\Mod(B^\ev)$ by Proposition \ref{prop:entireextclo}.
\end{proof}

\begin{lem}\label{lem:funcexactres}
The functor $A \otimes_B (-) \otimes_B A {\colon} \Mod(B^\ev) \rightarrow \Mod(A^\ev)$ gives rise to an almost costrong monoidal functor
$$
A \otimes_B (-) \otimes_B A {\colon} (\Mod(B^\ev), \otimes_B, B) \longrightarrow (\Mod(A^\ev), \otimes_A, A).
$$
Moreover, when restricted to $\P_\infty(B)$, it defines an exact functor $\P_\infty(B) \rightarrow \Mod(A^\ev)$.
\end{lem}

\begin{proof}
Let $M$ and $N$ be arbitrary $B^\ev$-modules. It is a straightforward verification that the natural homomorphisms
$$
\psi_0 {\colon} A \otimes_B B \otimes_B A \xrightarrow{\, \sim \, } A \otimes_B A \xrightarrow{\, \sim \, } A, \ a \otimes b \otimes a' \mapsto aba'
$$
and
\begin{align*}
\psi_{M,N} {\colon} A \otimes_B (M \otimes_B N) \otimes_B A & \xrightarrow{\ \sim \, } A \otimes_B (M \otimes_B B \otimes_B N) \otimes_B A\\ & \xrightarrow{\text{can}} A \otimes_B (M \otimes_B A \otimes_B N) \otimes_B A\\ & \xrightarrow{\ \sim \, } (A \otimes_B M \otimes_B A) \otimes_A (A \otimes_B N \otimes_B A)
\end{align*}
turn $A \otimes_B (-) \otimes_B A$ into an almost costrong monoidal functor $(\Mod(B^\ev), \otimes_B, B) \rightarrow \Mod(A^\ev, \otimes_A, A)$. As for the exactness, each choice of $X \in \P_\infty(B)$ and $i > 0$ satisfies the equality
$$
\Tor^{B \otimes B^\op}_i(A \otimes A^\op, X) = 0 = \Tor^B_i(A, X \otimes_{B} A) \quad \text{(for $i > 0$)}
$$
by Paragraph \ref{nn:tinftycinfty}. Therefore, the sequence
$$
0 \longrightarrow A \otimes_B L \otimes_B A \longrightarrow A \otimes_B M \otimes_B A \longrightarrow A \otimes_B N \otimes_B A \longrightarrow 0
$$
is exact for each admissible short exact sequence $0 \rightarrow L \rightarrow M \rightarrow N \rightarrow 0$ in $\P_\infty(B)$, whence the exactness of the restriction of $A \otimes_B (-) \otimes_B A$ to $\P_\infty(B)$ follows. 
\end{proof}

\begin{prop}\label{prop:exactmono}
There is an exact and entirely extension closed monoidal subcategory $(\C, \otimes_B, B)$ of $(\Mod(B^\ev), \otimes_B, B)$ such that
\begin{enumerate}[\rm(1)]
\item $\C$ is a full and exact subcategory of $\P_\infty(B)$ which is closed under taking direct summands and arbitrary direct sums, and
\item\label{prop:exactmono:2} the functor $A \otimes_B (-) \otimes_B A$ restricts to an exact and almost costrong monoidal functor
$$
\mathfrak A = (A \otimes_B (-) \otimes_B A)\mathord{\upharpoonright}_{\C} {\colon} (\C, \otimes_B, B) \longrightarrow (\P(A), \otimes_A, A),
$$
that is, one has the following commutative diagram of almost costrong monoidal functors.
$$
\xymatrix@C=55pt{
\C \ar[r]^-{\mathfrak A} \ar@{ >->}[d]_-{\mathrm{inc}} & \P(A) \ar@{ >->}[d]^-{\mathrm{inc}} \\
\Mod(B^\ev) \ar[r]^-{A \otimes_B (-) \otimes_B A} & \Mod(A^\ev)
}
$$
\end{enumerate}
\end{prop}

\begin{proof}
Observe that as soon as we have named an appropriate category $\C \subseteq \P_\infty(B)$ with $A \otimes_B \C \otimes_B A \subseteq \P(A)$, the exactness of $\mathfrak A$, as well as the fact that it is almost costrong monoidal, will follow from Lemma \ref{lem:funcexactres} immediately. In order to find $\C$, we proceed in several steps.

\medskip

\textbf{Step 1}. Let $\C_1$ be the full subcategory of $\P_\infty(B)$ defined as follows: An object $M \in \P_\infty(B)$ belongs to $\C_1$ if, and only if, $M \otimes_B X$ belongs to $\P_\infty(B)$ for all $X \in \P_\infty(B)$. Apperently, $\C_1$ is extension closed in $\P_\infty(B)$. It is further closed under direct summands and arbitrary direct sums, as well as under kernels of epimorphisms. Furthermore, for any $X \in \P_\infty(B)$, $B \otimes_B X \cong X$ so that $B \in \C_1$. Also, the module $(B \otimes B) \otimes_B X \cong B \otimes X$ is $B$-projective on either side. Moreover, $(-) \otimes_B (B \otimes B) \otimes_B X \otimes_B A \cong (-) \otimes_B (B \otimes X) \otimes_B A \cong (-) \otimes (X \otimes_B A)$ is an exact functor, since $X \otimes_B A$ is $K$-projective, and hence
$$
\Tor^B_i(A, (B \otimes B) \otimes_B X \otimes_B A) = 0 \quad \text{(for $i > 0$)}.
$$
It follows that $B \otimes B$ belongs to $\C_1$. From Proposition \ref{prop:entireextclo} we deduce that $\C_1$ is entirely extension closed in $\Mod(B^\ev)$. Finally, $\C_1$ is a monoidal subcategory of $\Mod(B^\ev)$ as, for all $M, N \in \C_1$ and $X \in \P_\infty(B)$,
$$
(M \otimes_B N) \otimes_B X \cong M \otimes_B (N \otimes_B X) \ \in \ \P_\infty(B) \, .
$$
\medskip

\textbf{Step 2.} Let $\C_2$ be the full subcategory of $\C_1$ which is defined as follows. An object $M \in \C_1$ belongs to $\C_2$ if, and only if, $A \otimes_B M \otimes_B A$ belongs to $\P(A)$. As $A \otimes_B B \otimes_B A \cong A$ and $A \otimes_B (B \otimes B) \otimes_B A \cong A \otimes A$, the modules $B$ and $B \otimes B$ belong to $\C_2$. Except for being monoidal, the subcategory $\C_2 \subseteq \Mod(B^\ev)$ has the same properties as $\C_1$, that is, it is extension closed, it is closed under direct summands and arbitrary direct sums and it is closed under kernels of epimorphisms. Yet again, these properties combined with $B \otimes B \in \C_2$ yield that $\C_2$ is entirely extension closed. We therefore have enforced the existence of an exact functor $A \otimes_B (-) \otimes_B A {\colon} \C_1 \rightarrow \P(A)$.
\medskip

\textbf{Step 3.} In the third and conclusive step, we define a full subcategory $\C = \C_3$ of $\C_2$ as follows. A module $M \in \C_2$ belongs to $\C$ if, and only if, $M \otimes_B X$ belongs to $\C_2$ for all $X \in \C_2$. Evidently, $B$ has that property, thus $B \in \C$. Let $X \in \C_2$; by adjunction, $\Hom_A(A \otimes X \otimes_B A, -) \cong \Hom_K(X \otimes_B A, -)$ and thus $A \otimes_B (B \otimes B) \otimes_B X \otimes_B A \cong A \otimes X \otimes_B A$ is projective as a left $A$-module. On the other hand, $\Hom_{A^\op}(A \otimes X \otimes_B A, -) \cong \Hom_{B^\op}(A \otimes X, -)$. Since $A \otimes X$ is a projective right $B$-module, $A \otimes X \otimes_B A$ is projective as a right $A$-module. It thus follows that $B \otimes B \in \C$.

The subcategory $\C$ inherits all other properties that $\C_2$ has, that is, it is extension closed, it is closed under direct summands and arbitrary direct sums and it is closed under kernels of epimorphisms. Naturally, it is entirely extension closed as well. Moreover, it is a monoidal subcategory of $\Mod(B^\ev)$, as for all $M, N \in \C$ and $X \in \C_2$
$$
A \otimes_B (M \otimes_B N) \otimes_B X \otimes_B A \cong A \otimes_B M \otimes_B (N \otimes_B X) \otimes_B A  \ \in \ \P(A) \, .
$$
Thus the proof of the proposition is established.
\end{proof}

\begin{rem}
Let $\mathfrak A {\colon} \Mod(B^\ev) \rightarrow \Mod(A^\ev)$ be a $K$-linear right exact monoidal functor that commutes with arbitrary direct sums. Assuming that the left derived functors $L_i \mathfrak A$ (for $i > 0$) vanish on $B$, one can construct an exact monoidal subcategory $\C(\mathfrak A) \subseteq \P(B)$ of $(\Mod(B^\ev), \otimes_B, B)$ so that $\mathfrak A$ restricts to an exact functor $(\C(\mathfrak A), \otimes_B, B) \rightarrow (\P(A), \otimes_A, A)$.
\end{rem}

\begin{rem}
Observe that the categories $\C_1$ and $\C = \C_3$ in the above proof agree with $\overline{P}_\infty(B)$ and $\overline \C_2$, as introduced in Paragraph \ref{nn:monoclosure}, and thus are extension closed and monoidal by Lemma \ref{lem:monoclosure}. Further, $\C_2$ is an exact category that arises in the way as described in Lemma \ref{lem:smallexactfunc}.
\end{rem}

\begin{lem}\label{lem:barexact}
If $\mathbb B B$ denotes the bar resolution for $B$, then the complex
$$
\mathbb B B \otimes_B A \longrightarrow B \otimes_B A \cong A \longrightarrow 0
$$
is a flat resolution of $A$ over $B$ and, moreover,
$$
A \otimes_B \mathbb B B \otimes_B A \longrightarrow A \otimes_B A \cong A \longrightarrow 0
$$
is a projective resolution of $A$ over $A^\ev$.
\end{lem}

\begin{proof}
First of all, $A \otimes_B (-) \otimes_B A$ is left-adjoint to the restriction functor $\Mod(A^\ev) \rightarrow \Mod(B^\ev)$ along the epimorphism $\pi^\ev {\colon} B \otimes B^\op \rightarrow A \otimes A^\op$ (cf.\,Lemma \ref{lem:piev}) and thus takes projective $B^\ev$-modules to projective $A^\ev$-modules. When considered as a sequence of right $B$-modules, the bar resolution $\mathbb B B \rightarrow B \rightarrow 0$ splits in every degree (in that each degree is the direct sum of the kernel and the image of the bordering homomorphisms). It follows that $\mathbb B B \otimes_B A \rightarrow B \otimes_B A \cong A \rightarrow 0$ is an exact resolution of $A$ by left $B$-modules. Since
$$
(-) \otimes_B B^{\otimes n} \otimes_B A \cong (-) \otimes B^{\otimes(n-2)} \otimes A \quad \text{(for all $n \geqslant 2$)}
$$
this is even a resolution by flat left $B$-modules. Thus,
$$
H_i(A \otimes_B \mathbb B B \otimes_B A) = \Tor^B_i(A,A) = 0 \quad \text{(for all $i > 0$)},
$$
and, consequently, $A \otimes_B \mathbb B B \otimes_B A \rightarrow A \otimes_B A \cong A \rightarrow 0$ is a projective resolution of $A$ over $A^\ev$.
\end{proof}

\begin{lem}\label{lem:mapsequal}
Let $\C$ be as in the statement of Proposition $\ref{prop:exactmono}$. The maps
$$
A \otimes_B^{\mathbf L} (-) \otimes_B^{\mathbf L} A {\colon} \Hom_{\mathbf D^-(B^\ev)}(B,B[\ast]) \longrightarrow \Hom_{\mathbf D^-(A^\ev)}(A,A[\ast])
$$
and
$$
\HH^\ast(B) \cong \Ext^\ast_\C(B,B) \xrightarrow{\, \mathfrak A_\ast \,} \Ext^\ast_{\mathsf P(A)}(A,A) \cong \HH^\ast(A)
$$
identify under the canonical isomorphisms between $\Hom$ in the derived category and $\Ext$ in the module category.
\end{lem}

\begin{proof}
By Lemma \ref{lem:barexact}, $(A \otimes_B^{\mathbf L} (-) \otimes_B^{\mathbf L} A)(B)$ computes as $A \otimes_B \mathbb B B \otimes_B A$ which is, of course, quasi-isomorphic to $A$. If $\mathbb X = (X_\ast, \partial_\ast)$ is a complex in $\Mod(B^\ev)$, we will denote by $\mathbb X^\natural$ the complex obtained from $\mathbb X$ by replacing $X_n$ by $0$ for $n \leqslant -1$; that is, $\mathbb X^\natural$ is the truncation of $\mathbb X$ in negative degrees. Fix an integer $n \geqslant 1$ and an admissible $n$-extension
$$
\xymatrix@C=16pt{
S & \equiv & 0 \ar[r] & B \ar[r] & E_{n-1} \ar[r] & \cdots \ar[r] & E_0 \ar[r] & B \ar[r] & 0
}
$$
in $\C$. Moreover, fix a comparison map
$$
\xymatrix@C=16pt{
(\mathbb B B \rightarrow B) \ar[d]_-\Phi & & \cdots \ar[r] & B^{\otimes(n+2)} \ar[r]^-{\beta_n} \ar[d]_-{\varphi_{n}} & B^{\otimes (n+1)} \ar[d]_-{\varphi_{n-1}} \ar[r] & \cdots \ar[r] & B \otimes B \ar[d]^-{\varphi_{0}} \ar[r]^-{\beta_0} & B \ar[r] \ar@{=}[d] & 0 \,\,\\
S & & 0 \ar[r] & B \ar[r] & E_{n-1} \ar[r] & \cdots \ar[r] & E_0 \ar[r] & B \ar[r] & 0 \, .
}
$$
The map $\Phi^\natural {\colon} \mathbb B B \rightarrow S^\natural$ is a quasi-isomorphism. Recall from, for instance, \cite{HiSt97} or \cite{Wei94}, that the canonical isomorphism $\Ext^n_\C(B,B) \xrightarrow{\sim} \HH^n(B) \xrightarrow{\sim} \Hom_{\mathbf D^-(B)}(B,B[n])$ (the first being due to the fact that $\C$ is entirely extension closed) sends the equivalence class of $S$ to the equivalence class $f_S$ of the roof
$$
\xymatrix@!C=18pt@R=22pt{
B & \cdots \ar[r] & 0 \ar[r] & 0 \ar[r]^-{} & 0 \ar[r]^-{} & \cdots \ar[r]^{} & B \ar[r] & 0 \ar[r] & \cdots \\
S^\natural \ar[u]^-{\text{qis}} \ar[d]  & \cdots \ar[r] & 0 \ar[r] \ar[u] \ar[d] & B \ar[r] \ar@{=}[d] \ar[u] & E_{n-1} \ar[d] \ar[u] \ar[r] & \cdots \ar[r] & E_0 \ar[u] \ar[r] \ar[d] & 0 \ar[r] \ar[d] \ar[u] & \cdots \\
B[n] & \cdots \ar[r] & 0 \ar[r] & B \ar[r]^-{} & 0 \ar[r]^-{} & \cdots \ar[r]^{} & 0 \ar[r] & 0 \ar[r] & \cdots
}
$$
in $\Hom_{\mathbf D^-(B)}(B,B[n])$. We claim that there is a morphism $\Psi {\colon} \mathbb B B \rightarrow \mathbb B B[n]$ of chain complexes in $\Mod(B^\ev)$, fitting inside the commutative diagram below.
\begin{equation}\label{eq:diagramqis}
\begin{aligned}
\xymatrix{
\mathbb B B \ar[d]_-{\text{qis}} & \mathbb B B \ar@{=}[l] \ar[r]^-\Psi \ar[d]^-{\text{qis}}_-{\Phi^\natural} & \mathbb B B[n] \ar[d]^-{\text{qis}}\\
B & S^\natural \ar[r] \ar[l]_-{\text{qis}} & B[n]
}
\end{aligned}
\end{equation}
Indeed, $\varphi_n$ uniquely factors through $\Im(\beta_n)$ as $\varphi_n \circ \beta_{n+1} = 0$. Further, there is some $\overline{\varphi}_n {\colon} B^{\otimes(n+2)} \rightarrow B \otimes B$ with $(B \otimes B \twoheadrightarrow B) \circ \overline{\varphi}_n = \varphi_n$ since $B^{\otimes(n+2)}$ is a projective $B^\ev$-module. A lifting
$$
\xymatrix@C=16pt{
\cdots \ar[r] & B^{\otimes (n+4)} \ar[r] \ar[d]_-{\psi_{n+2}} & B^{\otimes (n+3)} \ar[d]_-{\psi_{n+1}} \ar[r] & B^{\otimes(n+2)} \ar[r] \ar[d]_-{\psi_{n}}|-{=}^-{\overline{\varphi}_n} & \Im(\beta_n) \ar[r] \ar[d] & 0\\
\cdots \ar[r] & B^{\otimes 4} \ar[r] & B^{\otimes 3} \ar[r] & B \otimes B \ar[r] & B \ar[r] & 0
}
$$
of $\overline{\varphi}_n$ gives rise to the desired morphism $\Psi$. It thus follows that $A \otimes_B^{\mathbf L} f_S \otimes_B^{\mathbf L} A$ is represented by the roof $(\id_{\mathbb BB}, A \otimes_B \Psi \otimes_B A)$ conjugated by the canonical quasi-isomorphism $\mathbb B B \rightarrow B$.

As the diagram (\ref{eq:diagramqis}) lives in $\C$, all the quasi-isomorphisms in it will remain so if we apply $A \otimes_B (-) \otimes_B A$ to it. After identifying $A$ with $A \otimes_B B \otimes_B A$ (which we are allowed to do as $\pi$ is homological; see Paragraph \ref{nn:equhomepi}), we arrive at the commutative diagram
$$
\xymatrix@!C=60pt{
A \otimes_B \mathbb B B \otimes_B A \ar[d]_-{\text{qis}} & A \otimes_B \mathbb B B \otimes_B A \ar@{=}[d] \ar[r] \ar@{=}[l] & A \otimes_B \mathbb B B \otimes_B A[n] \ar[d]^-{\text{qis}}\\
A & A \otimes_B \mathbb B B \otimes_B A \ar[d]^-{\text{qis}} \ar[r]^-{\text{qis}} \ar[l]_-{\text{qis}} & A[n]\\
& (A \otimes_B S \otimes_B A)^\natural \ar[ur] \ar[ul]^-{\text{qis}} &
}
$$
which tells us that $A \otimes_B^{\mathbf L} f_S \otimes_B^{\mathbf L} A$ is represented by the roof $A \leftarrow (A \otimes_B S \otimes_B A)^\natural \rightarrow A[n]$. To summarise, we have shown that within the diagram
$$
\xymatrix@C=18pt{
\Hom_{\mathbf D^-(B^\ev)}(B,B[\ast]) \ar[rrr]^-{A \otimes_B^{\mathbf L} (-) \otimes_B^{\mathbf L} A} & & &  \Hom_{\mathbf D^-(A^\ev)}(A,A[\ast])\\
& \Ext^\ast_\C(B,B) \ar@{-->}[ul]_-\cong \ar[dl]_-\cong \ar[r]^-{\mathfrak U_\ast} & \Ext^\ast_{\P(A)}(A,A) \ar[dr]^-\cong \ar@{-->}[ur]^-\cong &\\
\HH^\ast(B) \ar[uu]^-\cong \ar@{-->}[rrr] & & & \HH^\ast(A) \ar[uu]_-\cong
}
$$
the sub-diagram given by the solid arrows commutes, whence the result follows.
\end{proof}

After all preparations made, the proof of Theorem \ref{thm:homepi} is now a single-liner. Recall that it asserted that the functor $A \otimes_B^{\mathbf L} (-) \otimes_B^{\mathbf L} A$ defines a homomorphism $\HH^\ast(B) \rightarrow \HH^\ast(A)$ of strict Gerstenhaber algebras.

\begin{proof}[Proof of Theorem $\ref{thm:homepi}$]
Use Proposition \ref{prop:exactmono}(\ref{prop:exactmono:2}) and Lemma \ref{lem:mapsequal} with Theorem \ref{thm:gersthochcomp}.
\end{proof}

We will apply the insights acquired in this and previous sections to long exact sequences by Koenig--Nagase.


\section{The long exact sequences of Koenig--Nagase}\label{sec:koenagase}

\begin{nn}
In the case of \textit{stratifying ideals}, the long exact Hochschild cohomology sequences stated in this paragraph are due to Koenig--Nagase; see \cite{KoNa09}. In those sequences, we will identify certain homomorphisms as such of strict Gerstenhaber algebras. Proofs for them being multiplicative can, for $I$ being stratifying, already be found in \cite{KoNa09}. In the following, we let $B$ be a $K$-algebra, $I \subseteq B$ a two-sided ideal, $A = B/I$ and $\pi {\colon} B \rightarrow A$ the canonical surjection. We further assume that $B$ and $A$ are \textit{projective} as $K$-modules. Observe that one has the following commutative diagrams with exact rows:
\begin{equation*}
\begin{aligned}
\xymatrix@C=18pt@R=18pt{
0 \ar[r] & \Tor_1^B(I,A) \ar[d]_-\cong \ar[r] & I \otimes_B I \ar[r] \ar@{=}[d] & I \otimes_B B \ar[r] \ar[d]^-\cong & I \otimes_B A \ar[r] \ar[d]^-\cong & 0\, \ \\
0 \ar[r] & \Tor_2^B(A,A) \ar[r] & I \otimes_B I \ar[r]^-{\mu_I} & I \ar[r] & I/I^2 \ar[r] & 0 \, ,
}
\end{aligned}
\end{equation*}
where $\mu_I{\colon} I \otimes_B I \rightarrow I$ denotes the (restricted) multiplication map, and
\begin{equation*}
\begin{aligned}
\xymatrix@C=18pt@R=18pt{
0 \ar[r] & \Tor_1^B(A,A) \ar@{=}[d] \ar[r] & I \otimes_B A \ar[r]^-0 \ar[d]_-\cong & B \otimes_B A \ar[r] \ar[d]^-\cong & A \otimes_B A \ar[r] \ar@{=}[d] & 0\, \ \\
0 \ar[r] & \Tor_1^B(A,A) \ar[r] & I/I^2 \ar[r] & A \ar[r] & A \otimes_B A \ar[r] & 0 \, .
}
\end{aligned}
\end{equation*}
Thus, $\Tor_1^B(A,A) \cong I/I^2$ and $A \cong A \otimes_B A$ as $A^\ev$-modules. It follows that $\Tor_1^B(A,A) = 0 = \Tor_2^B(A,A)$ holds true if, and only if, the map $\mu_I{\colon} I \otimes_B I \rightarrow I$ is an isomorphism and $I = I^2$. See \cite{APT92} for a thorough study of ideals with such properties, and \cite{GeLe91} for some related observations. 
\end{nn}

\begin{lem}\label{lem:ideal}
If $\Tor^B_i(A,A) = 0$ for $i > 0$, then the functor $$A \otimes_B (-) \otimes_B A {\colon} \Mod(B^\ev) \longrightarrow \Mod(A^\ev)$$ gives rise to an isomorphism
$$
\HH^\ast(B,A) = \Ext^\ast_{B^\ev}(B,A) \xrightarrow{\ \sim \ } \Ext^\ast_{A^\ev}(A,A) = \HH^\ast(A)
$$
of graded modules over $\HH^\ast(A)$.
\end{lem}

\begin{proof}
By Lemma \ref{lem:barexact}, $A \otimes_B \mathbb B B \otimes_B A \rightarrow A \otimes_B A \cong A \rightarrow 0$ is a projective resolution of $A$ over $A^\ev$. Therefore, for all $n \geqslant 0$,
\begin{align*}
\Ext^n_{B^\ev}(B,A) & \cong H^n \Hom_{B^\ev}(\mathbb B B, A) && \text{(by definiton)}\\
&\cong H^n \Hom_{A^\ev}(A \otimes_B \mathbb B B \otimes_B A, A) && \text{(by adjunction)}\\
&\cong \Ext^n_{A^\ev}(A,A) && \text{(by the remark above)},
\end{align*}
whence the claim follows.
\end{proof}

The following definition is taken from \cite[Sec.\,2.1]{ClPaSc96}.

\begin{defn}
Let $\Lambda$ be a ring. An idempotent element $e \in \Lambda$ is a \textit{stratifying idempotent} if the multiplication map $\Lambda e \otimes_{e\Lambda e} e\Lambda \rightarrow \Lambda e \Lambda$ is an isomorphism and $\Tor^{e\Lambda e}_i(\Lambda e, e\Lambda) = 0$ for $i > 0$. A twosided ideal $J \subseteq \Lambda$ is a \textit{stratifying ideal} if there is a stratifying idempotent $e \in \Lambda$ that generates $J$, i.e., $J = \Lambda e \Lambda$.
\end{defn}

\begin{exas}
Let $R$ and $S$ be $K$-algebras, and $M$ an $R \otimes S^\op$-module. 
\begin{enumerate}[\rm(1)]
\item Consider the matrix algebra
$$
B = \left(\begin{matrix}
R & M\\
0 & S
\end{matrix}
\right)
$$
along with the canonical idempotents
$$
e = \left(\begin{matrix}
1 & 0\\
0 & 0
\end{matrix}\right) \quad \text{and} \quad e' = \left(\begin{matrix}
0 & 0\\
0 & 1
\end{matrix}\right).
$$
Clearly $I = BeB$ is isomorphic to $eB$, and $R = Be = eBe$. Hence $R \otimes_{R} eB = Be \otimes_{eBe} eB \rightarrow BeB = eB$ is an isomorphism and
$$
\Tor^{eBe}_i(Be, eB) = \Tor^{eBe}_i(eBe, eB) = 0 \quad \text{(for $i > 0$)}.
$$
Consequently, $e$ is a stratifying idempotent in $B$, and $I = BeB$ is a stratifying ideal. By similar arguments, $e'$ is a stratifying idempotent as well, and thus, $I' = Be'B$ is stratifying. 
\item The above example can be extended. Let $N$ be a $S \otimes R^\op$-module and $\phi \colon N \otimes_R M \rightarrow S$, $\psi \colon M \otimes_S N \rightarrow R$ be bimodule homomorphisms. These additional data give rise to a \textit{generalised matrix algebra}:
$$
B = \left(\begin{matrix}
R & M\\
N & S
\end{matrix}
\right)\,.
$$
By \cite[Prop.\,4.1]{GaPs15}, the idempotent $e$ is stratifying in $B$ if, and only if, $\phi$ is a monomorphism and $\Tor^{eBe}_i(Be,eB) = 0$ for all $i > 0$. Similarly, $e'$ is stratifying in $B$ if, and only if, $\psi$ is a monomorphism and $\Tor^{e'Be'}_i(Be',e'B) = 0$ for all $i > 0$.
\end{enumerate}
\end{exas}

\begin{rem}
A great deal of examples of algebras that are controlled by certain stratifying ideals is given by \textit{quasi-hereditary algebras}, as studied in, for instance, 
the very clear note \cite{DlRi89b}. Let $\Lambda$ be a ring and $\mathrm{Rad}(\Lambda)$ its Jacobsen radical. Recall that a twosided ideal $J$ of $\Lambda$ is called a \textit{heredity ideal}, if $J = \Lambda e \Lambda$ for some idepotent element $0 \neq e \in \Lambda$, $J \mathrm{Rad}(\Lambda) J = 0$ and $J$ is a projective right $\Lambda$-module. By \cite[Appendix, Statement 6]{DlRi89b} every idempotent ideal is of the form $\Lambda e \Lambda$ for some idempotennt $e \in \Lambda$, if $\Lambda$ is \textit{semiprimary} (that is, $\mathrm{Rad}(\Lambda)$ is nilpotent and $\Lambda / \mathrm{Rad}(\Lambda)$ is semisimple artinian; therefore, every finite dimensional algebra over a field is semiprimary), so that in this situation the first requirement translates to $J = J^2$.

Let $J = \Lambda e \Lambda \subseteq \Lambda$ be a heredity ideal. It follows from \cite[Appendix, Statement 7]{DlRi89b} that the multiplication map $\Lambda e \otimes_{e\Lambda e} e\Lambda \rightarrow J$ defines a bimodule isomorphism. Moreover, as $J$ is a projective right $\Lambda$-module,
$$
\Hom_{(e\Lambda e)^\op}(\Lambda e, (-)e) \cong \Hom_{\Lambda^\op}(\Lambda e \otimes_{e\Lambda e} e\Lambda, -) \cong \Hom_{\Lambda^\op}(J,-)
$$
is exact, thus $\Lambda e$ is a projective right $e\Lambda e$-module and, in particular, $J$ is a stratifying ideal in $\Lambda$. The ring $\Lambda$ is called \textit{quasi-hereditary}, if there is a \textit{heredity chain}
$$
0 = J_0 \subsetneq J_1 \subsetneq \cdots \subsetneq J_{n-1} \subsetneq J_n = \Lambda \, ,
$$
that is, a chain of ideals in $\Lambda$ such that $J_t/J_{t-1}$ is a heredity ideal in $\Lambda/J_{t-1}$ (for $1 \leqslant t \leqslant n$). We thus obtain canonical short exact sequences
$$
0 \longrightarrow J_t/J_{t-1} \longrightarrow \Lambda / J_{t-1} \longrightarrow \Lambda / J_t \longrightarrow 0 \quad \text{(for $1 \leqslant t \leqslant n$)},
$$
where the leftmost term in particular is a stratifying ideal of the middle term. Every semiprimary ring of global dimension $2$ is quasi-hereditary (see \cite[Thm.\,2]{DlRi89b}). Moreover, any semiprimary algebra $R$ may be realised as $e \Lambda e$ for some quasi-hereditary algebra $\Lambda$ and some idempotent $e \in \Lambda$ (see \cite{DlRi89a}). Over fields, the Brauer algebras $B(r,\delta)$ and the partition algebras $P(r,\delta)$ are quasi-hereditary for many choices of the parameters $r \geqslant 1$ and $\delta \in K$ (see \cite[Thm.\,1.3 and Thm.\,1.4]{KoXi99}).
\end{rem}

\begin{lem}\label{lem:homepi}
If the ideal $I$ is a stratifying ideal, then $\Tor^B_i(A,A) = 0$ for $i > 0$, that is, $\pi{\colon} B \rightarrow A$ is a homological epimorphism.
\end{lem}

\begin{proof}
Observe that for $i > 0$,
$$
\Tor_i^B(I,A) \cong \Tor^B_i(Be \otimes_{eBe} eB, A) \cong \Tor^{eBe}_i(Be, eB \otimes_B A) \cong \Tor^{eBe}_i(Be, eA) = 0
$$
by \cite[Chap.\,IX, Thm.\,2.8]{CaEi56} and, moreover,
$$
\Tor^B_1(A,A) = I/I^2 = 0, \quad \Tor_{i+1}^B(A,A) \cong \Tor_i^B(I,A)
$$
by the long exact homology sequence arsing from applying $- \otimes_B A$ to the canonical short exact sequence $0 \rightarrow I \rightarrow B \rightarrow A \rightarrow 0$.
\end{proof}

\begin{rem}\label{rem:isoextstr}
The lemma implies that the restriction functor induces isomorphisms
$$
\Ext^\ast_{A}(M,N) \xrightarrow{\, \sim \, } \Ext^\ast_B(M,N)\, ,
$$
for all $A$-modules $M$ and $N$, as $\mathbf D^b(\pi_\star)$ is full and faithful by Paragraph \ref{nn:equhomepi}. The following statement can also be found in \cite{GeLe91}. We still give the (short) proof for completeness.
\end{rem}

\begin{lem}\label{lem:strattorvan}
If the ideal $I$ is a stratifying ideal, then $\Tor^B_i(A,M) = 0$ for each $A$-module $M$ and each integer $i > 0$.
\end{lem}

\begin{proof}
Indeed, for all $i > 0$,
\begin{align*}
0 & = \Ext^i_A(A,\Hom_{\mathbb Z}(M, \mathbb Q / \mathbb Z))&& \text{(as $A$ is $A$-pojective)}\\
&\cong \Ext^i_B(A,\Hom_{\mathbb Z}(M, \mathbb Q / \mathbb Z)) && \text{(by Remark \ref{rem:isoextstr})}\\
&\cong \Hom_{\mathbb Z}(\Tor^B_i(A,M), \mathbb Q/ \mathbb Z) && \text{(by \cite[Chap.\,VI, Prop.\,5.1]{CaEi56})}.
\end{align*}
As $\mathbb Q/\mathbb Z$ is an injective cogenerator for $\mathbb Z$,
$$
\Hom_{\mathbb Z}(\Tor^B_i(A,M), \mathbb Q/ \mathbb Z) = 0 \quad \Longleftrightarrow \quad \Tor^B_i(A,M) = 0 \quad \text{(for $i > 0$)},
$$
thus the claim follows.
\end{proof}

The following observation is an immediate consequence of the above lemma, and Remark \ref{rem:isoextstr}.

\begin{lem}[see {\cite[Prop.\,3.3]{KoNa09}}]\label{lem:isostrat}
Assume that the ideal $I$ is a stratifying ideal, and let $e \in B$ be the stratifying idem\-po\-tent that generates $I$, i.e., $I = BeB$. If $M$ is a $B^\ev$-module, then
$$
\Ext^\ast_{B^\ev}(I,M) \xrightarrow{\ \sim \ } \Ext^\ast_{(eBe)^\ev}(eBe, eMe) = \HH^\ast(eBe, eMe)
$$
as graded $K$-modules.
\end{lem}

\begin{proof}
Since $I \cong Be \otimes_{eBe} eB$ as $B^\ev$-modules, we conclude
$$
\Ext^\ast_{B^\ev}(I,M) \xrightarrow{\ \sim \ } \Ext^\ast_{B^\ev}(Be \otimes_{eBe} eB, M) \xrightarrow{\ \sim \ } \HH^\ast(eBe, eMe)
$$
by Lemma \ref{lem:isoidealtor}.
\end{proof}

\begin{nn}
By applying $\Hom_{B^\ev}(B,-)$ and $\Hom_{B^\ev}(-,B)$ to the canonical short exact sequence $0 \rightarrow I \rightarrow B \rightarrow A \rightarrow 0$ we obtain long exact sequences
$$
\xymatrix@C=16pt{
\cdots \ar[r] & \HH^n(B,I) \ar[r] & \HH^n(B) \ar[r] & \HH^n(B,A) \ar[r] & \HH^{n+1}(B,I) \ar[r] & \cdots
}
$$
and
$$
\xymatrix@C=16pt{
\cdots \ar[r] & \Ext^n_{B^\ev}(A,B) \ar[r] & \HH^n(B) \ar[r] & \Ext^n_{B^\ev}(I,B) \ar[r] & \Ext^{n+1}_{B^\ev}(A,B) \ar[r] & \cdots
}\, .
$$
The following is our main observation in this section.
\end{nn}

\begin{thm}\label{thm:koenignagase}
\begin{enumerate}[\rm(1)]
\item\label{thm:koenignagase:1} If $\Tor_i^B(A,A) = 0$ for all $i > 0$, then the canonical short exact sequence $0 \rightarrow I \rightarrow B \rightarrow A \rightarrow 0$ gives rise to a long exact sequence
$$
\xymatrix@C=18pt{
\cdots \ar[r] & \HH^n(B,I) \ar[r] & \HH^n(B) \ar[r]^-{k_n} & \HH^n(A) \ar[r] & \HH^{n+1}(B,I) \ar[r] & \cdots
}
$$
in Hochschild cohomology. The map $k_\ast {\colon} \HH^\ast(B) \rightarrow \HH^\ast(A)$ is a homomorphism of strict Gerstenhaber algebras, that is, $k_\ast$ is a homomorphism of graded $K$-algebras and
$$
\xymatrix@C=35pt{
\HH^m(B) \times \HH^n(B) \ar[r]^-{\{-,-\}_B} \ar[d]_-{k_m \times k_n} & \HH^{m+n-1}(B) \ar[d]^-{k_{m + n -1}} \\
\HH^{m}(A) \times \HH^n(A) \ar[r]^-{\{-,-\}_A} & \HH^{m+n-1}(A)
}
\ \quad\quad \
\xymatrix@C=35pt{
\HH^{2n}(B) \ar[r]^{sq_B} \ar[d]_-{k_{2n}} & \HH^{4n-1}(B) \ar[d]^-{k_{4n-1}}\\
\HH^{2n}(A) \ar[r]^{sq_A} & \HH^{4n-1}(A)
}
$$
commute.

\item\label{thm:koenignagase:2} If the ideal $I$ is a stratifying ideal, with corresponding stratifying idempotent $e \in B$, then the short exact sequence $0 \rightarrow I \rightarrow B \rightarrow A \rightarrow 0$ gives rise to a long exact sequence
$$
\xymatrix@C=18pt{
\cdots \ar[r] & \Ext_{B^\ev}^n(A,B) \ar[r] & \HH^n(B) \ar[r]^-{l_n} & \HH^n(eBe) \ar[r] & \Ext_{B^\ev}^{n+1}(A,B) \ar[r] & \cdots
}
$$
in Hochschild cohomology. The map $l_\ast {\colon} \HH^\ast(B) \rightarrow \HH^\ast(eBe)$ is a homomorphism of strict Gerstenhaber algebras, that is, $l_\ast$ is a homomorphism of graded $K$-algebras and
$$
\xymatrix@C=30pt{
\HH^m(B) \times \HH^n(B) \ar[r]^-{\{-,-\}_B} \ar[d]_-{l_m \times l_n} & \HH^{m+n-1}(B) \ar[d]^-{l_{m + n -1}} \\
\HH^{m}(eBe) \times \HH^n(eBe) \ar[r]^-{\{-,-\}_{eBe}} & \HH^{m+n-1}(eBe)
}
\quad\quad
\xymatrix@C=30pt{
\HH^{2n}(B) \ar[r]^{sq_B} \ar[d]_-{l_{2n}} & \HH^{4n-1}(B) \ar[d]^-{l_{4n-1}}\\
\HH^{2n}(eBe) \ar[r]^{sq_{eBe}} & \HH^{4n-1}(eBe)
}
$$
commute. 
\end{enumerate}
\end{thm}

We will turn ourselves to the proof after a remark and an immediate corollary of the theorem.

\begin{rem}
Let $e \in B$ be a stratifying idempotent, $I = BeB$ and $A = B/I$. In \cite{KoNa09}, the authors also construct a long exact sequence
$$
\xymatrix@C=14pt{
\cdots \ar[r] & \Ext^n_{B^\ev}(A,I) \ar[r] & \HH^n(B) \ar[r] & \HH^n(A) \times \HH^n(eBe) \ar[r] & \Ext^{n+1}_{B^\ev}(A,I) \ar[r] & \cdots
} \, ,
$$
where the component maps of $\HH^\ast(B) \rightarrow \HH^\ast(A) \times \HH^\ast(eBe)$ precisely are $k_\ast$ and $l_\ast$ as introduced in Theorem \ref{thm:koenignagase}. Thus the map $\HH^\ast(B) \rightarrow \HH^\ast(A) \times \HH^\ast(eBe)$ in that sequence is a homomorphism of strict Gerstenhaber algebras as well. The first part of the corollary below makes use of this additional long exact sequence.
\end{rem}

\begin{cor}
Assume that the ideal $I$ is a stratifying ideal, with corresponding stratifying idempotent $e \in B$.
\begin{enumerate}[\rm(1)]
\item If $d = \id_{B^\ev}(I)$ is finite, then $k_n {\colon} \HH^n(B) \rightarrow \HH^n(A)$ is an isomorphism for $n > d$. Thus, $\HH^{>d}(B) \cong \HH^{>d}(A)$ as non-unital strict Gerstenhaber algebras. In particular, $\HH^{>d}(eBe) = 0$.
\item If $d' = \pd_{B^\ev}(A)$ is finite, then $l_n {\colon} \HH^n(B) \rightarrow \HH^n(eBe)$ is an isomorphism for $n > d'$. Thus, $\HH^{>d'}(B) \cong \HH^{>d'}(eBe)$ as non-unital strict Gerstenhaber algebras.
\qed
\end{enumerate}
\end{cor}

\begin{proof}[Proof of Theorem $\ref{thm:koenignagase}$]
Let us first prove item (\ref{thm:koenignagase:2}). As $\Ext^\ast_{B^\ev}(I,B) \xrightarrow \sim \HH(eBe)$ by Lemma \ref{lem:isostrat}, the existence of the exact sequence is imminent. By arguments being similar to those in the proof of Thereom \ref{thm:greensol}, one shows that the map $l_\ast$ is induced by the recollement $\mathcal R(B,e)$, thus, the assertion follows from Theorem \ref{thm:mainthm}. 

Thanks to Lemma \ref{lem:ideal}, we obtain a long exact sequence as claimed in (\ref{thm:koenignagase:1}). Let us conclude that $k_\ast$ is a homomorphism of strict Gerstenhaber algebras. In the following, we let $\pi: B \rightarrow A$ be the canonical surjection. By Proposition \ref{prop:exactmono}, one has the following diagram of graded maps.
$$
\xymatrix{
\Ext^\ast_\C(B,B) \ar[d]_-\cong \ar[r]^-{\mathfrak{A}_\ast} & \Ext^\ast_{\C(A)}(A,A) \ar[d]^-{\cong}\\
\HH^\ast(B) \ar[r]^-{k_\ast} & \HH^\ast(A)
}
$$
We are going to prove that this square commutes -- and thus an application of Theorem \ref{thm:homepi} yields the desired result, since the map $\mathfrak{A}_\ast$ coincides with the one discussed in there (see Lemma \ref{lem:mapsequal}). Let $\varphi{\colon} B^{\otimes(n+2)} \rightarrow B$ be a cocycle representing an element $\alpha \in \HH^n(B)$. The preimage of $\alpha$ in $\Ext^n_\C(B,B)$, under the vertical isomorphism, is given by the equivalence class of the lower sequence in the following pushout diagram.
$$
\xymatrix@C=18pt@R=18pt{
\cdots \ar[r] & B^{\otimes(n+2)} \ar[r]^-{\beta^B_n} \ar[d]_-{\varphi} & B^{\otimes(n+1)} \ar[r]^-{\beta^B_{n-1}} \ar[d] & \cdots \ar[r]^-{\beta^B_1} & B \otimes B \ar[r]^-{\beta^B_0} \ar@{=}[d] & B \ar[r] \ar@{=}[d] & 0\\
0 \ar[r] & B \ar[r] & B \oplus_{B^{\otimes(n+2)}} B^{\otimes(n+1)} \ar[r] & \cdots \ar[r] & B \otimes B \ar[r] & B \ar[r] & 0
}
$$
Put $P = B \oplus_{B^{\otimes(n+2)}} B^{\otimes(n+1)}$. Identifying $A$ with $A \otimes_B B \otimes_B A$, the homomorphism $\mathfrak A_n$ maps this equivalence class, by definition, to the equivalence class of the admissible $n$-extension
\begin{equation}\label{eq:imageA}\tag{$\diamond$}
0 \longrightarrow A \longrightarrow A \otimes_B P \otimes_B A \longrightarrow \cdots \longrightarrow A \otimes_B (B \otimes B) \otimes_B A \longrightarrow A \longrightarrow 0 \, ,
\end{equation}
which, thanks to the vertical isomorphism on the right (= the inclusion), sits inside $\HH^n(A)$. We show that this class coincides with $k_n(\alpha)$. Indeed, if $j_1 {\colon} A \otimes_B A \otimes_B A \xrightarrow{\sim} A$ and $j_2{\colon} A \otimes_B B \otimes_B A \xrightarrow{\sim} A$ denote the canonical isomorphisms, then $k_n$ maps $\alpha$ to the equivalence class of the cocycle $\tilde{\varphi} = j_1 \circ (A \otimes_B (\pi \circ \varphi) \otimes_B A) = j_2 \circ (A \otimes_B \varphi \otimes_B A)$,
$$
\tilde{\varphi}{\colon} A \otimes B^{\otimes n} \otimes A \longrightarrow A, \ \tilde{\varphi}(a \otimes b_1 \otimes \cdots \otimes b_n \otimes a') = a \cdot (\pi \circ \varphi)(1 \otimes b_1 \otimes \cdots \otimes b_n \otimes 1) \cdot a'\, ,
$$
in $\Hom_{A^\ev}(A \otimes_B \mathbb B B \otimes_B A, A)$. As $A \otimes_B B^{\otimes(m+2)} \otimes_B A \cong A \otimes B^{\otimes m} \otimes A$ belongs to $\P(A)$ for all $m \geqslant 0$, a preimage of the equivalence class of $\tilde{\varphi}$ in $\Ext^n_{\P(A)}(A,A)$ is given by the equivalence class of the lower sequence in the pushout diagram
$$
\xymatrix@C=18pt@R=18pt{
\cdots \ar[r] & A \otimes B^{\otimes n} \otimes A \ar[r] \ar[d]_-{\tilde{\varphi}} & A \otimes B^{\otimes(n-1)} \otimes A \ar[r] \ar[d] & \cdots \ar[r] & A \otimes A \ar[r] \ar@{=}[d] & A \ar[r] \ar@{=}[d] & 0\\
0 \ar[r] & A \ar[r] & Q \ar[r] & \cdots \ar[r] & A \otimes A \ar[r] & A \ar[r] & 0
}
$$
which, as $\mathfrak A$ is exact and thus commutes with pushouts along admissible epimorphisms, coincides with the equivalence class of the sequence (\ref{eq:imageA}).
\end{proof}

\subsection{The long exact sequence of Michelena--Platzeck} We review yet another long exact sequence which is also due to Michelena--Platzeck; see \cite[Sec.\,1]{MiPl00} for its original construction which, however, is given in a far more restrictive setting.

In what follows, we will consider upper triangular matrix algebras of the form
$$
B =
\left(
\begin{matrix}
R & M\\
0 & S
\end{matrix}
\right) \,,
$$
where, as in the previous sections, $R$ and $S$ are algebras over the fixed commutative base ring $K$ and $M$ is an $R \otimes S^\op$-module such that $R$, $S$ and $M$ are projective as $K$-modules. The canonical idempotents $e$ and $e' = 1 - e$ in $B$, associated to $R$ and $S$ respectively, will once again play a key role.

\begin{nn}
To begin with, notice that
$$
Be' = Be'B = \left(
\begin{matrix}
0 & M\\
0 & S
\end{matrix}
\right) \quad \text{and} \quad
eB = BeB = \left(
\begin{matrix}
R & M\\
0 & 0
\end{matrix}
\right)\,.
$$
As we have already observed, these are stratifying ideals in $B$. There are four algebra homomorphisms, relating $R$, $S$ and $B$, namely
$$
R \longrightarrow B, \ r \mapsto \left(
\begin{matrix}
r & 0\\
0 & 1_S
\end{matrix}
\right)\, \qquad B \longrightarrow R, \ \left(
\begin{matrix}
r & m\\
0 & s
\end{matrix}
\right) \mapsto r
$$
and
$$
S \longrightarrow B, \ s \mapsto \left(
\begin{matrix}
1_R & 0\\
0 & s
\end{matrix}
\right)\, \qquad B \longrightarrow S, \ \left(
\begin{matrix}
r & m\\
0 & s
\end{matrix}
\right) \mapsto s \,.
$$
Note that, for instance, the right $B$-module structure of $Be'$ coincides with the one given by restricting the right $S$-module structure along $B \rightarrow S$. We now deduce the following consequence of Theorem \ref{thm:koenignagase}, generalising the statement \cite[Thm.\,1.14]{MiPl00}.
\end{nn}

\begin{cor}\label{cor:mpla}
Put $I = Be'B = Be'$. Then there is a long exact sequence
$$
\xymatrix@C=18pt@R=6pt{
0 \ar[r] & \mathrm{Ann}_S(M) \cap Z(S) \ar[r] & \HH^0(B) \ar[r]^-{m_0} & \HH^0(R) \ar[r] & \Ext^{1}_{B^\op \otimes S}(S,I) \ar[r] & \cdots\\
\cdots \ar[r] & \Ext^{n}_{B^\op \otimes S}(S,I) \ar[r] & \HH^n(B) \ar[r]^-{m_n} & \HH^n(R) \ar[r] & \Ext^{n + 1}_{B^\op \otimes S}(S,I) \ar[r] & \cdots
}
$$
wherein the graded map $m_\ast{\colon} \HH^\ast(B) \rightarrow \HH^\ast(R)$ is a homomorphisms of strict Gerstenhaber algebras.
\end{cor}

\begin{proof}
Indeed, by Theorem \ref{thm:koenignagase}(\ref{thm:koenignagase:1}), we are given a long exact sequence which is induced by the short exact sequence
	$$
	0 \longrightarrow Be' = Be'B \longrightarrow B \longrightarrow B/Be' = R \longrightarrow 0
	$$
of $B^\ev$-modules, connecting $\HH^\ast(B)$ and $\HH^\ast(R)$ by a homomorphism of strict Gerstenhaber algebras. It thus remains to be shown, that $\Hom_{B^\ev}(B,I) \cong \mathrm{Ann}_S(M) \cap Z(S)$ and $\Ext^{n}_{B^\ev}(B,I) \cong \Ext^n_{B^\op \otimes S}(S,I)$ for all $n \geqslant 1$. If $f \in \Hom_{B^\ev}(B,I)$, then $f$ is determined by the image
	$$
	f(1_B) = \left(\begin{matrix}
0 & m\\
0 & s
\end{matrix}\right)\,.
	$$
	As
	$$
	\left(\begin{matrix}
0 & 0\\
0 & s's
\end{matrix}\right) = \left(\begin{matrix}
1 & 0\\
0 & s'
\end{matrix}\right) \cdot f(1_B) = f(1_B) \cdot \left(\begin{matrix}
1 & 0\\
0 & s'
\end{matrix}\right) = \left(\begin{matrix}
0 & ms'\\
0 & ss'
\end{matrix}\right)
$$
must hold for all $s' \in S$, it follows that $m = 0$ and $s \in Z(S)$. 	Further, the equation
	$$
	\left(\begin{matrix}
0 & m's\\
0 & 0
\end{matrix}\right) = \left(\begin{matrix}
0 & m'\\
0 & 0
\end{matrix}\right) \cdot f(1_B) = f(1_B) \cdot \left(\begin{matrix}
0 & m'\\
0 & 0
\end{matrix}\right) = \left(\begin{matrix}
0 & 0\\
0 & 0
\end{matrix}\right)\,,
$$
which holds for all $m' \in M$, implies that $s \in \mathrm{Ann}_S(M)$, so that 
	$$
	\Hom_{B^\ev}(B,I) \longrightarrow \mathrm{Ann}_S(M) \cap Z(S), \ f \mapsto f(1_B)
	$$
	is a $K$-linear isomorphism. As $\Tor^B_i(B, S) = 0 = \Ext^i_{S^\op}(S, I)$ for all $i > 0$, we may apply \cite[Chap.\,IX, Thm.\,2.8a]{CaEi56} to deduce
	\begin{align*}
	\Ext^n_{B^\ev}(B, \Hom_{S^\op}(S,I)) &\cong \Ext^n_{B^\op \otimes S}(B \otimes_B S, I)\\ &\cong \Ext^n_{B^\op \otimes S}(S, I)
	\end{align*}
	for $n > 0$. To conclude, we observe that $I \cong \Hom_{S^\op}(S,I)$ as $B \otimes S^\op$-modules, hence also as $B^\ev$-modules.
\end{proof}

\begin{rem}
Notice that by replacing $e$ by $e'$, and thus $R$ by $S$, we obtain a similar long exact sequence relating $\Ext^\ast_{B \otimes R^\op}(B,eB)$, $\HH^\ast(B)$ and $\HH^\ast(S)$ with the map $\HH^\ast(B) \rightarrow \HH^\ast(S)$ being a homomorphism of strict Gerstenhaber algebras.
\end{rem}

\begin{nn}
Let us close this section by summarising the observations made so far. Let $e \in B$ be a stratifying idempotent, $e' = 1 - e$ and $I = BeB$. The standard recollement
\begin{equation*}
\begin{aligned}
\xymatrix{\\\mathcal R(B,e) \qquad \equiv \qquad\\}
\xymatrix@!C=50pt{\\ \Mod(B/I) \ar[r]^-{\mathrm{res}} & \Mod(B) \ar[r]^-{e(-)} \ar@/^2pc/[l]^-{\Hom_{B}(B/I,-)} \ar@/_2pc/[l]_-{(B/I) \otimes_{B} -} & \Mod(eBe) \ar@/^2pc/[l]^-{\Hom_{eBe}(eB,-)} \ar@/_2pc/[l]_-{Be \otimes_{eBe} -}\\}
\end{aligned}
\end{equation*}
gives rise to two homomorphisms of strict Gerstenhaber algebras:
$$
\HH^\ast(B) \longrightarrow \HH^\ast(eBe), \ \quad \ \HH^\ast(B) \longrightarrow \HH^\ast(B/I).
$$
The first one is given by applying the functor $e(-)e$, whereas the latter comes from deriving $(B/I) \otimes_B (-) \otimes_B (B/I)$. These homomorphisms fit into several long exact sequences.
\begin{gather*}
\xymatrix@C=18pt{
\cdots \ar[r] & \HH^n(B,I) \ar[r] & \HH^n(B) \ar[r] & \HH^n(A) \ar[r] & \HH^{n+1}(B,I) \ar[r] & \cdots
}\\
\xymatrix@C=18pt{
\cdots \ar[r] & \Ext_{B^\ev}^n(A,B) \ar[r] & \HH^n(B) \ar[r] & \HH^n(eBe) \ar[r] & \Ext_{B^\ev}^{n+1}(A,B) \ar[r] & \cdots
}\\
\xymatrix@C=14pt{
\cdots \ar[r] & \Ext^n_{B^\ev}(A,I) \ar[r] & \HH^n(B) \ar[r] & \HH^n(A) \oplus \HH^n(eBe) \ar[r] & \Ext^{n+1}_{B^\ev}(A,I) \ar[r] & \cdots
}
\end{gather*}
and, provided that $\Tor^{e'Be'}_i(Be',e'B) = 0$ for $i > 0$,
\begin{align*}
\cdots \xrightarrow{\ \ } \HH^n(B) \xrightarrow{\ \ } \HH^n(eBe) \oplus \HH^n(e'Be') \xrightarrow{\ \ }  \Ext^n_{B^\ev}(\Omega_B^1, B) \xrightarrow{\ \ } \HH^{n+1}(B) \xrightarrow{\ \ } \cdots \, .
\end{align*}
Any upper or lower triangular matrix algebra, with canonical idempotents, is an example that gives rise to such a situation, and in that case, the first exact sequence simplifies to the one in Corollary \ref{cor:mpla}.
\end{nn}


\section{The finite generation conjecture of Snashall--Solberg}\label{sec:snashsol}

\begin{nn}
The Hochschild cohomology ring of an algebra is usually far away from being finitely generated. This changes, if one factors out the ideal generated by all homogeneous nilpotent elements. The resulting quotient still has very nice properties, to some extend even nicer ones than the ring we started with, since it will be a commutative ring (no signs involved!) if one works over a field. Also this quotient is finitely generated in many cases, but, however, not always, as it was proven by example in \cite{Xu08}, and then even further generalised in \cite{Sn09}, contrasting the initial conjecture by Snashall--Solberg; see \cite{SnSo04}.

However, it seems completely unknown and unstudied, whether or when the quotient by the (strict) Gerstenhaber ideal generated by the homogeneous nilpotent elements is finitely generated.
\end{nn}

\begin{nn}
For a $K$-algebra $B$ and a subset $S \subseteq \HH^\ast(B)$ of homogeneous elements, we will denote by
\begin{enumerate}[\rm(1)]
\item $I(S)$ the graded ideal of $\HH^\ast(B)$ generated by $S$.
	\item $\mathcal I(S)$ the graded Lie ideal of $\HH^\ast(B)$ generated by $S$.
	\item $G(S)$ the \textit{weak Gerstenhaber ideal} generated by $S$, that is, the smallest graded Lie subalgebra of $\HH^\ast(B)$ containing $S$ and being an ideal with respect to $\cup$.
	\item $\mathcal G(S)$ the \textit{Gerstenhaber ideal generated by $S$}, that is, the smallest subset of $\HH^\ast(B)$ containing $S$ and being both an ideal with respect to $\cup$ and $\{-,-\}_B$.
\end{enumerate}
Clearly, we have $\mathcal I(S) \subseteq \mathcal G(S)$ and $I(S) \subseteq G(S) \subseteq \mathcal G(S)$, whereas there does not need to be any relation between $\mathcal I(S)$ and $G(S)$. If $K$ is a field and $S$ is the set of homogeneous nilpotent elements, then $\HH^\ast(B)/G(S)$ and $\HH^\ast(B)/\mathcal G(S)$ are \textit{commutative} algebras (in the strict sense). In the following, we will review a couple of examples.
\end{nn}

\begin{nn}\label{nn:elementab}
For the remainder of the section, and in fact of this article, assume that $K$ is a field of characteristic $p > 0$. Let $\mathbb Z_ p = \mathbb Z/ p\mathbb Z$ and $\mathbb Z_p^r$ be the finite elementary abelian $p$-group of rank $r > 0$, with corresponding group algebra $\Gamma = K \mathbb Z_p^r$. Let $\{-,-\}$ be the Gerstenhaber bracket on $\HH^\ast(\Gamma)$. By \cite{LeZh13}, the Gerstenhaber algebra structure of $\HH^\ast(\Gamma)$ is given as described in the following two paragraphs. 
\end{nn}

\begin{nn}
If $p = 2$, then
$$
\HH^\ast(\Gamma) \cong \frac{K[x_1, \dots, x_r, y_1, \dots, y_r]}{(x_1^2, \dots, x_r^2)}\,, \quad \abs{x_i} = 0, \, \abs{y_i} = 1
$$
and
$$
\{x_i,x_j\} = \{y_i,y_j\} = 0\, , \quad \{x_i, y_j\} = \delta_{ij}
$$
for all $i,j = 1, \dots, r$. The ideal $I(N) \subseteq \HH^\ast(\Gamma)$ generated by the homogeneous nilpotent elements $N$ in $\HH^\ast(\Gamma)$ is precisely the ideal $(x_1, \dots, x_r)$, which also coincides with $G(N)$, whereas $\mathcal I(N) = K \oplus \mathrm{Span}_K\{x_iy_j^{2v} \mid 1 \leq i,j \leq r, v \in \mathbb Z_{\geq 0}\}$ and $\mathcal G(N) = \HH^\ast(\Gamma)$ as $\{x_i, y_i^n\} = n y_i^{n-1}$ for $n \geq 1$. It thus follows that
$$
\HH^\ast(\Gamma)/I(N) = \HH^\ast(\Gamma)/G(N) \cong K[y_1, \dots, y_r] \quad \text{and} \quad \HH^\ast(\Gamma)/\mathcal G(N) \cong 0 \, .
$$
A squaring map for $\{-,-\}$ is given by $sq(\alpha) = \alpha^2$ (for $\alpha$ of even degree). 
\end{nn}

\begin{nn}
If $p \neq 2$, one has
$$
\HH^\ast(\Gamma) \cong \frac{K[x_1, \dots, x_r, z_1, \dots, z_r]}{(x_1^p, \dots, x_r^p)} \otimes \Lambda(y_1, \dots, y_r)\,, \quad \abs{x_i} = 0, \, \abs{y_i} = 1, \, \abs{z_i} = 2
$$
and
$$
\{x_i,x_j\} = \{y_i,y_j\} = 0\, , \quad \{x_i, y_j\} = \delta_{ij}\,, \quad \{z_i,-\} = 0
$$
for all $i,j = 1, \dots, r$. Denoting $\{a, a+1, \dots, b\}$ by $[a,b]$ for $b-a \in \mathbb N$, we get $I(N) = (x_1, \dots, x_r, y_1, \dots, y_r)$, $\mathcal I(N) = \mathrm{Span}_K\{x_i^u y_j^{v} \mid i,j \in [1,r], u \in [0,p-1], v \in [0,1]\} \oplus \mathrm{Span}_K\{x_i^u z_j^w \mid i,j \in [1,r], u \in [0,p-1], w \in \mathbb N\}$ and $G(N) = \mathcal G(N) = \HH^\ast(\Gamma)$, whence
$$
\HH^\ast(\Gamma)/I(N) \cong K[z_1, \dots, z_r] \quad \text{and} \quad \HH^\ast(\Gamma)/G(N) = \HH^\ast(\Gamma)/\mathcal G(N) \cong 0 \, .
$$
As we assumed that $p \neq 2$, \textit{the} squaring map for $\{-,-\}$ is given by $sq(\alpha) = 2^{-1} \{\alpha, \alpha\}$ (for $\alpha$ of even degree). It sends the $x_i, y_i^2$ and $z_i$ to zero, but $x_i + y_i^{2n}$ to $n y_i^{2n-1} \neq 0$ (for $p \nmid n$).
\end{nn}

\begin{nn}\label{nn:E_0}
Let $\C$ be a \textit{finite EI-category}, that is, a category with $\abs{\Mor(\C)} < \infty$ and $\End_\C(X) = \Aut_\C(X)$ for all $X \in \Ob(\C)$. Finite groups and finite acyclic quivers are examples of such categories. The \textit{category algebra of $\C$} (over $K$) has $K \Mor(\C)$ as underlying $K$-module, and the product $f \cdot g$ of morphisms $f, g$ in $\C$ is $f \circ g$, whenever this makes sense, and $0$ otherwise. If $\C$ is a finite group or a finite acylcic quiver, $K\C$ agrees with the corresponding group or path algebra. For an example that is neither a group algebra nor a path algebra of a quiver, consider the finite EI-category $\E_0$
$$
\xymatrix@C=1pt{
{\ X\ } \ar@(u,l)_h \ar@(l,d)_g \ar@(d,r)_{gh} \ar@(r,u)_{\id_X} & & \ar@<3pt>[rrrrrr]^-\alpha
\ar@<-3pt>[rrrrrr]_-\beta & &&&&& Y \ar@(ur,dr)^{\id_Y}
}
$$
subject to the relations
\begin{equation*}
g^2 = \id_X = h^2, \ gh = hg, \ \alpha h= \beta g = \alpha, \ \alpha g = \beta h = \beta.
\end{equation*}

Let $N \subseteq \HH^\ast(K\E_0)$ denote the set of all homogeneous nilpotent elements, and $I(N)$ the graded ideal generated by $N$. It is one of the main results of \cite{Xu08}, that $\HH^\ast(K\E_0) / I(N)$ is not finitely generated as a graded $K$-algebra if $K$ is a field of characteristic $2$. In this setting, one has the following obvious interpretation of the algebra $K\E_0$.
\end{nn}

\begin{lem}\label{lem:isoonepoint}
Assume that $\mathrm{char}(K) = 2$. The category algebra $K\E_0$ is isomorphic to the one-point-extension $B = \Gamma[M] = \left( \begin{smallmatrix} K (\mathbb Z_2 \times \mathbb Z_2) & K\mathbb Z_2\\ 0 & K \end{smallmatrix} \right)$ of $\Gamma = K(\mathbb Z_2 \times \mathbb Z_2)$ by $M = K\mathbb Z_2$. \qed
\end{lem}

The following questions shall occupy us in future work.

\begin{quest}\label{conj:snashsol}
Let $A$ be a finite dimensional algebra over a field $K$. If $N$ denotes the set of all homogeneous nilpotent elements in $\HH^\ast(A)$, is the commutative $K$-algebra
$$
\HH^\ast_G(A) = \HH^\ast(A) / G(N)
$$
finitely generated? More specifically, can we use (any) of the long exact sequences discussed in this article, along with \ref{nn:elementab}, to deduce that
$$
\HH^\ast_G(K\E_0) = \HH^\ast(K\E_0) / G(N)
$$
is finitely generated (where $\E_0$ denotes Xu's EI-category reviewed above and $\mathrm{char}(K) = 2$)? Especially the results obtained in Subsection \ref{subsec:happel} should be of help.
\end{quest}

\begin{quest}
Under the assumptions of the above question, does the commutative algebra $\HH^\ast_G(A)$ possess enough structure to establish a sensible support variety theory for (special classes of) $A$-modules, in a similar way as described in \cite{SnSo04}?
\end{quest}


\end{document}